\documentclass[11pt,reqno]{article}

\usepackage[margin=1in]{geometry}
\usepackage{amsmath,amsthm,amssymb,amsfonts}

\makeatletter
\newcommand{\subjclass}[2][2010]{%
  \let\@oldtitle\@title%
  \gdef\@title{\@oldtitle\footnotetext{#1 \emph{Mathematics subject classification.} #2}}%
}
\makeatother

\usepackage[colorlinks=true, pdfstartview=FitV, linkcolor=blue,citecolor=blue, urlcolor=blue]{hyperref}

\usepackage[abbrev,lite,nobysame]{amsrefs}
\usepackage{times}
\usepackage[usenames,dvipsnames]{color}

\usepackage{esint}

\usepackage{mathtools,enumitem,mathrsfs}

\usepackage[compact]{titlesec}

\usepackage[title]{appendix}

\usepackage{comment}

\mathtoolsset{showonlyrefs=true}
    
\newcommand{\ep}{\epsilon}
\newcommand{\eps}{\epsilon}

\newcommand{\leqc}{\lesssim}

\newcommand{\geqc}{\gtrsim}

\newcommand{\grad}{\nabla}
\newcommand{\MalD}{\cD}
\newcommand{\Wbf}{\mathbf{W}}

\newcommand{\Vc}{{\mathcal V}}
\newcommand{\Dc}{{\mathcal D}}

\newcommand{\norm}[1]{\left|\left| #1 \right|\right|}
\newcommand{\abs}[1]{\left| #1 \right|}
\newcommand{\set}[1]{\left\{ #1 \right\}}
\newcommand{\brak}[1]{\left\langle #1 \right\rangle} 

\newcommand{\R}{\mathbb{R}}
\newcommand{\Real}{\R}
\newcommand{\C}{\mathbb{C}}

\newcommand{\Z}{\mathbb{Z}}

\newcommand{\T}{\mathbb{T}}

\newcommand{\tensor}{\otimes}

\newcommand{\Zc}{\mathcal{Z}}

\newcommand{\cC}{\mathcal{C}}

\newcommand{\cD}{\mathcal{D}}

\newcommand{\Hbf}{{\bf H}}

\newcommand{\dee}{\mathrm{d}}
\newcommand{\ds}{\dee s}
\newcommand{\dt}{\dee t}

\newcommand{\dx}{\dee x}

\newcommand{\dy}{\dee y}

\newcommand{\dr}{\dee r}

\DeclareMathOperator{\Div}{\mathrm{div}}
\DeclareMathOperator{\Id}{\mathrm{Id}}

\DeclareMathOperator{\curl}{\mathrm{curl}}

\usepackage{bbm}
\newcommand{\1}{\mathbbm{1}}

\renewcommand{\P}{\mathbf{P}}

\newcommand{\E}{\mathbf{E}}
\newcommand{\EE}{\mathbf E}
\newcommand{\PP}{\mathbf P}
\newcommand{\Leb}{\operatorname{Leb}}

\newtheorem{theorem}{Theorem}[section]
\newtheorem{proposition}[theorem]{Proposition}
\newtheorem{corollary}[theorem]{Corollary}
\newtheorem{lemma}[theorem]{Lemma}
\newtheorem*{lemma*}{Lemma}

\newtheorem{assumption}{Assumption}
\newtheorem{system}{System}
\newtheorem{condition}{Condition}
\theoremstyle{definition}
\newtheorem{definition}[theorem]{Definition}
\newtheorem{remark}[theorem]{Remark}

\numberwithin{equation}{section}
    
\begin{document}

\title{Almost-sure exponential mixing of passive scalars by the stochastic Navier-Stokes equations}
\author{Jacob Bedrossian\thanks{\footnotesize Department of Mathematics, University of Maryland, College Park, MD 20742, USA \href{mailto:jacob@math.umd.edu}{\texttt{jacob@math.umd.edu}}. J.B. was supported by NSF CAREER grant DMS-1552826 and NSF RNMS \#1107444 (Ki-Net)} \and Alex Blumenthal\thanks{\footnotesize Department of Mathematics, University of Maryland, College Park, MD 20742, USA \href{mailto:alex123@math.umd.edu}{\texttt{alex123@math.umd.edu}}. This material was based upon work supported by the National Science Foundation under Award No. DMS-1604805. A.B. would like Dmitry Dolgopyat for useful insights and helpful discussions.} \and Samuel Punshon-Smith\thanks{\footnotesize Division of Applied Mathematics,  Brown University, Providence, RI 02906, USA \href{mailto:punshs@brown.edu}{\texttt{punshs@brown.edu}}. This material was based upon work supported by the National Science Foundation under Award No. DMS-1803481.}} 

\maketitle

\begin{abstract}
We deduce almost-sure exponentially fast mixing of passive scalars advected by solutions of the stochastically-forced 2D Navier-Stokes equations and 3D hyper-viscous Navier-Stokes equations in $\mathbb T^d$ subjected to non-denegenerate $H^\sigma$-regular noise for any $\sigma$ sufficiently large.
That is, for all $s > 0$ there is a deterministic exponential decay rate such that all mean-zero $H^s$ passive scalars decay in $H^{-s}$ at this same rate with probability one.
This is equivalent to what is known as \emph{quenched correlation decay} for the Lagrangian flow 
in the dynamical systems literature.
This is a follow-up to our previous work, which establishes a positive Lyapunov exponent
for the Lagrangian flow-- in general, almost-sure exponential mixing is much stronger than this.
Our methods also apply to velocity fields evolving according to finite-dimensional fluid models, for example Galerkin truncations of Navier-Stokes or the Stokes equations with very degenerate forcing. For all $0 \leq k < \infty $ we exhibit many examples of $C^k_t C^\infty_x$ random velocity fields that are almost-sure exponentially fast mixers. 
\end{abstract}

\setcounter{tocdepth}{2}
{\small\tableofcontents}

\section{Introduction}\label{sec:Intro}

Passive scalar mixing by fluid motion is an important aspect of many physical phenomena (see e.g. \cite{Provenzale1999,ShraimanSiggia00,W00,WF04})
and understanding mixing has been a topic of active research in mathematics recently; see e.g. \cite{Bressan03,LDT11,LLNMD12,S13,IyerXu14,MC18,TZ18,CZDT18,CRS19,GGM19,IF19} and the references therein (see below for more discussion).
In this paper, we will study the mixing of a passive scalar $g_t$ advected by an incompressible velocity field $u_t$ in the absence of diffusivity, 
\begin{align}
	\partial_t g_t + u_t \cdot \grad g_t = 0. \label{eq:gt}
\end{align}
The previous work in mathematics does not yet consider the case that arises most commonly in physics: 
velocity fields evolving under the nonlinear dynamics of an ergodic system. 
The velocity fields we consider are governed by a variety of stochastic fluid models, for example, the stochastically-forced 2D Navier-Stokes equations on $\mathbb T^2$: 
\begin{equation} 
\begin{dcases}
\partial_t u_t + u_t \cdot \grad u_t + \grad p_t = \nu \Delta u_t + Q \dot W_t \\
\Div u_t = 0 \\ 
u_0 =: u, 
\end{dcases} 
\end{equation}
where $W_t$ is a cylindrical Wiener process and $Q$ is a smoothing operator that spatially colors the noise. We will show that under suitable conditions on $Q$, the velocity field $u_t$ is \emph{almost-surely exponentially mixing}. 
Roughly speaking (see Theorem \ref{thm:-sdecay} for the rigorous statement and discussion), this means that $\forall s>0$, there is a \emph{deterministic} constant $\hat\gamma > 0$ and an \emph{almost-surely finite} random constant $D = D(u,\omega)$ (depending on the initial velocity $u$ and the Brownian path $\omega$), such that the following holds for \emph{all} initial  $g_0 = g \in H^s$ with $\int_{\mathbb T^2} g \dx = 0$,
\begin{align}
\norm{g_t}_{H^{-s}} = \sup_{\norm{f}_{H^s} = 1} \int_{\mathbb T^2} f(x) g_t(x) \dx \leq D e^{-\hat\gamma t} \norm{g}_{H^s}, \label{ineq:HmsDec}
\end{align}
where the supremum is taken over mean zero functions $f\in H^s$. 
Alternatively, \eqref{ineq:HmsDec} can be formulated in terms of the \emph{Lagrangian flow map}  $\phi^t : \mathbb T^2 \mapsto \T^2$, solving
\begin{equation}  
\frac{\dee}{\dt}\phi^t(x)  = u_t(\phi^t(x)),\quad \phi^0(x)  = x. \label{def:LagIntro} 
 \end{equation}
Indeed, by incompressibility and $g_t = g \circ (\phi^t)^{-1}$, \eqref{ineq:HmsDec} is equivalent to
\begin{align}
\int_{\mathbb T^2} f(\phi^t (x))g(x) \dx \leq D e^{-\hat\gamma t} \norm{f}_{H^s} \norm{g}_{H^s}. \label{ineq:quenched}
\end{align}
We see that \eqref{ineq:quenched} quantifies an almost sure, exponentially fast decay of correlations. 
Such estimates are known in the  dynamics literature as \emph{quenched correlation decay} \cite{bahsoun2017quenched, dragivcevic2018spectral,  baladi2002almost, ayyer2007quenched, aimino2015annealed}. 
Our study is built mostly on analyzing \eqref{def:LagIntro} using ideas from random dynamical systems and the geometric ergodicity of Markov processes. 
In our previous work \cite{BBPS18}, we showed that the flow map $\phi^t$ has a positive Lyapunov exponent, that is $D_x \phi^t(x)$ grows exponentially fast $\forall x$ with probability 1 (see \eqref{eq:lyapgrowhtIntro} below for rigorous statement). However, quenched correlation decay is a much stronger statement a priori than simply having a positive Lyapunov exponent (see e.g. \cite{sarig2002subexponential} and the discussion in Section \ref{sec:mainres}). 
In fact the positive Lyapunov exponent proved in \cite{BBPS18} plays a crucial role in the proof of Theorem \ref{thm:-sdecay} (see Section \ref{subsec:outline2PTDecay} for more information).  

Let us now set up the rigorous statement of the results. We first state the results for the infinite-dimensional evolution equations, namely 2D Navier-Stokes and 3D hyper-viscous Navier-Stokes. Then we specialize the results to finite dimensional evolutions, e.g. Galerkin truncations of Navier-Stokes or the Stokes equations subject to very degenerate noise.
In the latter case, a more general class of problems can be treated, producing families of random
 velocity fields that are $C^k_t C^\infty_x$ and almost-surely exponential mixers for any fixed $k < \infty$.

\subsection{Infinite-dimensional fluid models}\label{subsubsec:noiseProcess}

We will consider two (infinite-dimensional) stochastic fluid models on the periodic box $\T^d = (- \pi, \pi]^d$, the 2D Navier-Stokes equations and the 3D hyperviscous Navier-Stokes equations. However, each model has a different natural $L^2$-based ``energy''. For 2D Navier-Stokes it is the enstrophy $\norm{\curl u}_{L^2}^2$ and for 3D hyperviscous Navier-Stokes it is the kinetic energy $\norm{u}_{L^2}^2$. With a slight abuse of notation, we will define a natural Hilbert space on velocity fields $u:\T^d \to \R^d$ by
\begin{equation}
\Wbf := \set{u \in L^2(\mathbb T^d;\Real^d)\,: \,\int u\, \dx = 0,\quad \Div u = 0, \quad \sqrt{\brak{u,u}_{\Wbf}} = \|u\|_{\Wbf} < \infty},
\end{equation}
with the inner product defined via
\begin{equation}\label{defineWbfNorm}
\brak{u,v}_{\Wbf} := 
\begin{dcases}
\brak{\curl u, \curl v}_{L^2} \quad &\mathrm{if}\quad  d=2 \\
\brak{u,v}_{L^2} \quad & \mathrm{if} \quad d=3.
\end{dcases}
\end{equation}

We will also need to keep track of higher regularity: for $s>0$, define
\begin{align}
\Hbf^{s} = \set{u \in H^s(\T^d, \R^d)\, :\, \int u \, \dx = 0,\quad \Div u = 0} \, .
\end{align}
Following the convention used in \cite{E2001-lg,BBPS18}, we define a natural real Fourier basis on $L^2$ divergence free fields by defining for each $m = (k,i) \in \mathbb{K} := \Z^d_0 \times \{1,\ldots,d-1\}$
 \[\label{eq:Fourier-Basis}
 e_m(x) = \begin{cases}
 c_d\gamma_k^i\sin(k\cdot x), \quad& k \in \Z^d_+\\
 c_d\gamma_k^i\cos(k\cdot x),\quad& k\in \Z^d_-,
 \end{cases}
 \]
where $\Z^d_0 := \Z^d \setminus \set{0,\ldots, 0}$, $\Z_+^d = \{k\in \Z^d_0 : k^{(d)} >0\}\cup\{k\in \Z^d_0 \,:\, k^{(1)}>0, k^{(d)}=0\}$ and $\Z_-^d = - \Z_+^d$, and for each $k\in \Z_0^d$, $\{\gamma_k^i\}_{i=1}^{d-1}$ is a set of ${d-1}$ orthonormal vectors spanning the plane perpendicular to $k \in \R^d$ with the property that $\gamma_{-k}^i = - \gamma_{k}^i$. The constant $c_d = \sqrt{2}(2\pi)^{-d/2}$ is a normalization factor so that $e_m(x)$ are a complete orthonormal basis on $L^2$ divergence free vector fields. Note that in dimension $d=2$  $\mathbb{K} = \Z^d_0$, hence $\gamma_k^1 = \gamma_k$ is just a vector in $\R^2$ perpendicular to $k$ and is therefore given by $\gamma_k = \pm k^{\perp}/|k|$.

The white-in time noise that we will force our equations with is $Q\dot{W_t}$, where $W_t$ is a cylindrical Wiener process on $\Wbf$ with respect to an associated canonical stochastic basis $(\Omega,\mathscr{F},(\mathscr{F}_t),\P)$ and $Q$ is a positive Hilbert-Schmidt operator on $\Wbf$ which we assume can be diagonalized with respect to $\{e_m\}$ with eigenvalues $\{q_m\}\in \ell^2(\mathbb{K})$ defined by
\[
Qe_m = q_m e_m,\quad m = (k,i)\in \mathbb{K}.
\]
In this way $Q\dot{W}_t$ can be represented in terms of the basis$\{e_m\}$ by
\[
	Q\dot{W}_t = \sum_{m\in\mathbb{K}} q_m e_m \dot{W}^m_t 
\]
where $\{W^m_t\}_{m\in \mathbb{K}}$ are a collection of iid one-dimensional Wiener processes with respect to $(\Omega,\mathscr{F},(\mathscr{F}_t),\P)$.

We will assume that $Q$ satisfies the following regularity and non-degeneracy assumption:
\begin{assumption} \label{a:Highs}
  There exists $\alpha$ satisfying $\alpha > \frac{5d}{2}$ and a constant $C$ such that
\[
	\frac{1}{C}\|(-\Delta)^{-\alpha/2}u\|_{\Wbf} \leq \|Qu\|_{\Wbf} \leq C\|(-\Delta)^{-\alpha/2}u\|_{\Wbf}.
\]
Equivalently
\[
	|q_{m}| \approx |k|^{-\alpha}.
\]
\end{assumption}
\begin{remark}
Assumption \ref{a:Highs} essentially says that the forcing is $QW_t$ has high spatial regularity but cannot be $C^\infty$. The non-degeneracy requirement on $Q$ can be weakened to a more mild non-degeneracy at only high-frequencies (see \cite{BBPS18}), but fully non-degenerate noise simplifies some arguments. 
We are currently unable to treat noise which is degenerate at high frequencies as in \cite{HM06,HM11}.
See Remark \ref{rmk:Hypo} below for a more precise discussion of the two places we fundamentally depend on Assumption \ref{a:Highs}. 
\end{remark}

We will define our primary phase space of interest to be the following: 
\[
	\Hbf = \Hbf^\sigma , \quad \text{for some fixed}\quad  \sigma \in (\alpha-2(d-1), \alpha  - \tfrac{d}{2}). 
\]
Note that we have chosen $\alpha$ sufficiently large to ensure that $\sigma > \frac{d}{2} + 3$ so that we have the embedding $\Hbf \hookrightarrow C^3$. 
We will consider a stochastic evolution $(u_t)$ in $\Hbf$, which we refer to as the {\em velocity process},
 solving one of the two following stochastic PDEs: 
\begin{system}[2D Navier-Stokes equations]\label{sys:NSE}
\begin{equation}
\begin{dcases}
\,\partial_t u_t + u_t \cdot \grad u_t =- \grad p_t + \nu \Delta u_t + Q \dot W_t \\ 
\,\Div u_t = 0 \, , 
\end{dcases}
\end{equation}
where $u_0 = u \in \Hbf$ and the viscosity parameter $\nu >0$ is fixed.
\end{system}

\begin{system}[3D hyper-viscous Navier-Stokes]\label{sys:3DNSE}
\begin{equation}
\begin{dcases}
\,\partial_t u_t + u_t \cdot \grad u_t =- \grad p_t + \nu' \Delta u_t - \nu \Delta^{2} u_t + Q \dot W_t \\ 
\,\Div u_t = 0,
\end{dcases}
\end{equation}
where $u_0 = u \in \Hbf$. Here, the viscosity parameter $\nu' >0$, and hyperviscosity parameter $\nu>0$ are fixed.
\end{system}

The following well-posedness theorem is classical (See Section \ref{sec:Prelim}).   

\begin{proposition} \label{prop:WP}
For both Systems \ref{sys:NSE}, \ref{sys:3DNSE} and all initial data $u\in \Hbf$,
there exists a $\P$-a.s. unique, global-in-time, $\mathscr{F}_t$-adapted mild solution $(u_t)$ satisfying $u_0 = u$. 
Moreover, $(u_t)$ defines a Feller Markov process and the corresponding Markov semigroup has a unique stationary probability measure $\mu$ on $\Hbf$.
\end{proposition}

With the $(u_t)$ process on $\Hbf$ as in Proposition \ref{prop:WP}, we write $\phi^t = \phi^t_{u}$ for the stochastic
flow of diffeomorphisms on $\T^d$ solving \eqref{def:LagIntro} with initial velocity field $u$. 
This gives rise to an $\mathscr{F}_t$-adapted, Feller Markov process $(u_t, x_t)$ on
$\Hbf \times \T^d$ which we refer to as the \emph{Lagrangian process} defined by $x_t = \phi^t (x)$, where $x_0 = x \in \T^d$. 
Uniqueness of the stationary measure $\mu \times \Leb$ for the process $(u_t, x_t)$ for both of Systems \ref{sys:NSE} and \ref{sys:3DNSE} was proved in \cite{BBPS18}.

\subsection{Main results and discussion} \label{sec:mainres}
We are now ready to state our main results, for which we give further context and discussion afterwards. 
\begin{theorem}\label{thm:-sdecay}
Let $(u_t)$ be as in any of Systems \ref{sys:NSE} -- \ref{sys:3DNSE}, initiated at $\mu$-generic $u \in \Hbf$, with $\mu$ as in Proposition \ref{prop:WP}.
Fix $s > 0$ and $p \geq 1$. Then, there exists a (deterministic)
constant $\hat{\gamma} = \hat{\gamma}(s,p) > 0$, depending only on $s,p$ and the parameters of the system (e.g. $Q$, $\nu$, etc), and a
random constant $D = D(\omega, u) : \Omega \times {\bf H} \to [1,\infty)$ depending on the same parameters as $\hat{\gamma}$ and additionally on $(\omega,u)$, the sample path and the initial data, which satisfy the following properties:
\begin{itemize}
\item[(i)] For any $f, g \in H^s(\T^d)$, satisfying $\int f \,\dx = \int g\,\dx = 0$ the following holds for all $t \geq 0$:
\begin{align}
\left| \int f (x) g(\phi^t_{u} (x)) \, \dx\right| \leq D \| f \|_{H^s} \| g \|_{H^s} e^{- \hat{\gamma} t} \, . \label{ineq:sdecay}
\end{align}
\item[(ii)] For $u$ fixed, $D(\cdot, u)$ is $\PP$-a.e. finite and moreover satisfies the following: $\exists \beta \geq 2$ (independent of $u$, $p$, $s$) such that  $\forall \eta > 0$ there holds
\begin{align}
\EE D^p \lesssim_{p,\eta,s} \left(1 + \norm{u}^2_{\Hbf}\right)^{p\beta} \exp\left(\eta \norm{u}_{\mathbf{W}}^2 \right) < \infty \, . \label{ineq:Dtail}
\end{align}
For $s \leq 1$ and $p \geq 1$, one can take $\hat\gamma \gtrsim \frac{s}{p}$. 
\end{itemize}
\end{theorem}

As discussed above \eqref{ineq:quenched}, an immediate corollary of Theorem \ref{thm:-sdecay} is $H^{-s}$ decay for passively advected scalars in the absence of 
dissipation and with no sources as in \eqref{eq:gt} above.

\begin{corollary} \label{cor:mixscalar}
In the setting of Theorem \ref{thm:-sdecay}, fix $s > 0$, $p \geq 1$, and let $\hat\gamma =\hat \gamma(s,p)$ as in Theorem \ref{thm:-sdecay}. Then, 
for any mean-zero initial $g \in H^s$, the solution $g_t$ to the passive scalar problem \eqref{eq:gt} 
satisfies the estimate
\begin{align}
\| g_t \|_{H^{-s}} \leq D e^{-\hat \gamma t} \| g \|_{H^s}  \label{ineq:gdec}
\end{align}
where $D$ is the random constant appearing in Theorem \ref{thm:-sdecay}.
\end{corollary}

\begin{remark} 
By time reversibility of the advection equation, no point-wise in time decay, such as in \eqref{ineq:gdec}, can hold without some regularity assumption on $g$. On the other hand, by a density argument, for $\mu$-generic $u$ and all mean-zero $g \in L^2$, we have $g_t \rightharpoonup 0$ in $L^2$ as $t \to \infty$ $\PP$-a.e. 
at a rate which is uniform over compact sets in $L^2$. See \cite{TZ18} for more discussion.
\end{remark}

\begin{remark}
One always has $\norm{g_t}_{H^{-s}} \leq \norm{g}_{L^2}$, hence $D$ quantifies an upper bound on 
the random time-scale one will have to wait in order to see the exponential decay $\hat\gamma$. Decay occurs for times $t \gtrsim \tau = \frac{1}{\hat\gamma} \log \left(D \norm{g}_{H^s}\right)$  and estimates such as \eqref{ineq:Dtail} shows that $\exists \delta >0$ such that the following exponential moment holds: $\exists \beta'$ large so that $\forall \eta > 0$,
\begin{align*}
\EE e^{\delta \tau} \lesssim_{\eta} \left(1 + \norm{u}_{\Hbf}^2 \right)^{\beta'} \exp\left( \eta \norm{u}_{\mathbf W}^2\right) \norm{g}_{H^s}.  
\end{align*}
\end{remark}

It is well-known that sufficiently regular velocity fields can mix at most exponentially fast.
Refining exactly the relation between regularity and mixing rate is the content of Bressan's conjecture \cite{Bressan03} and has been studied in several works, for example \cite{S13,IyerXu14} and the references therein. In the other direction, the construction of exponential mixers has also proved challenging.
Roughly speaking, deterministic, time-autonomous flows possess a coherent flow direction along which 
no mixing can occur, presenting significant difficulties in the construction of such mixers. 
See, e.g., \cite{ruelle1983flows, liverani2004contact}.

Exponential mixers which are H\"older continuous and smooth away from a finite set of 
hyperplanes were constructed in \cite{TZ18} (the fields studied in Theorem \ref{thm:-sdecay} are not $C^\infty_x$ but for all $k < \infty$, can be chosen $C^k_x$ by choosing $\alpha$ large enough).
 See also the work \cite{GGM19} where the velocity is chosen depending on the scalar, and \cite{YZ14} where the decay is on long but finite time-intervals. 
For stochastic flows of diffeomorphisms on compact manifolds (in particular these velocity fields are white-in-time) having a positive Lyapunov exponent and satisfying certain non-degeneracy assumptions, almost-sure exponential mixing was proved in \cite{dolgopyat2004sample}.
As we will see, the unbounded phase-space (of $(u_t)$) and infinite-dimensional dynamics of the velocity field in Theorem \ref{thm:-sdecay} both present fundamental difficulties that require many new ideas to treat.

{ Informally, the infinitesimal mechanism responsible for mixing is hyperbolicity, or 
stretching and contracting in various directions in phase space.
In our previous work \cite{BBPS18}, we proved the presence of hyperbolicity
for Lagrangian flow by showing 
 that with probability 1, $D_x \phi^t$ (eventually) grows exponentially fast for all initial $x \in \T^d$ 
and velocity fields $u$. 
However, this statement is \emph{local} in $x$, and is therefore much weaker than the 
\emph{global} statement of Theorem \ref{thm:-sdecay}. Indeed, a
 positive Lyapunov exponent is basically equivalent to an (almost-sure, eventual) exponential growth of $\norm{\grad g_t}_{L^2}$ for any non-zero $g \in H^1$ (see the discussion in \cite{BBPS18}), which by Sobolev interpolation, is much weaker than \eqref{ineq:sdecay}. 
One can construct many examples of dynamical systems with a positive Lyapunov exponent but arbitrarily slow (e.g., polynomial or logarithmic) decay of correlations, for example, Pommeau-Manneville maps (see, e.g., \cite{sarig2002subexponential}).}

\begin{remark}[The role of nonlinearity]
If one drops the nonlinearity from Systems \ref{sys:NSE} or \ref{sys:3DNSE}, one is left with the Stokes equation, and the velocity field is simply
\begin{align*}
u_t(x) = \sum_{m \in \mathbb K} \beta^m_t e_m(x) \, , 
\end{align*}
where $\set{\beta^m_t}_{m \in \mathbb K}$ is a family of independent Ornstein-Uhlenbeck processes with variances $\frac{q_m^2}{2 \nu \abs{k}^2}$ (with $m = (k,i)$). 
Heuristically, one can expect the Lagrangian flow to behave similarly to successively applying i.i.d random shear flows of all orientations, which would resemble a hyperbolic toral automorphism.
In accordance with this intuition, the Stokes equations is by far the easiest case to treat. 

In our work here and in \cite{BBPS18}, the nonlinearity in the Navier-Stokes equations is (a priori) an \emph{enemy}.
It requires the adaptation or creation of more powerful and flexible tools.
Moreover, there are also physical reasons one may worry about introducing nonlinearity. 
For example, in two dimensions, the Navier-Stokes equations form coherent vortices inside of which hyperbolicity is halted.
See e.g. \cite{BBPV94,Provenzale1999,BabianoProvenzale07,PBBPW08} and the references therein for discussions on this and its implications for passive scalar dynamics. 
The Furstenberg criteria arguments in \cite{BBPS18} rule out that these vortices can form coherent structures that permanently entrain any part of the Lagrangian flow. 
Theorem \ref{thm:-sdecay} gives exponential tail control on any slow down of hyperbolicity caused by transient coherent vortices. 
\end{remark}

\subsection{Finite dimensional evolution and $C^{k}_t C^\infty_x$ almost-sure exponential mixers}\label{subsec:finiteDimIntro}
When considering finite dimensional evolutions for the velocity fields, our methods significantly simplify 
it suffices to impose much weaker nondegeneracy conditions on the noise.
\begin{assumption}[Low mode non-degeneracy] \label{a:lowms}
Define $\mathcal{K}_0 \subset \mathbb K$ to be the set of $m \in \mathbb K$ such that $q_m \neq 0$. Assume $m \in \mathcal{K}_0$ if $\abs{m}_{\infty} \leq 2$ (for $m = (k,i)$, $k = (k_i)_{i = 1}^d \in \Z^d$ we write $|m|_{\infty} = \max_{i} |k_i|$). 
\end{assumption}

We write $\Hbf_{\mathcal K_0} \subset \Hbf$ for the subspace spanned by the Fourier modes $m \in \mathcal K_0$ and $\Hbf_N \subset \Hbf$ for the subspace spanned by the Fourier modes satisfying $\abs{m}_\infty \leq N$.
Consider the Stokes system (with very degenerate forcing) and Galerkin-Navier-Stokes systems defined as the following. 

\begin{system} \label{sys:2DStokes} 
We refer to the \emph{Stokes system} in $\mathbb T^d$ ($d = 2,3$) as the following, for $u_0 = u \in \Hbf_{\mathcal K}$: 
\begin{equation} \label{eq:stokes-2d}
\begin{cases}
\,\partial_t u_t =- \grad p_t + \Delta u_t + Q \dot W_t \\ 
\,\Div u_t = 0 
\end{cases},
\end{equation}
where $Q$ satisfies Assumption \ref{a:lowms} and $\mathcal K_0$ is finite.
\end{system}
\begin{system} \label{sys:Galerkin}
We refer to the \emph{Galerkin-Navier-Stokes system} in $\mathbb T^d$ ($d = 2,3$) as the following, for $u_0 = u \in \Hbf_N$:
\begin{equation}
\begin{cases}
\,\partial_t u_t + \Pi_N\left(u_t\cdot \grad u_t + \grad p_t \right) = \nu\Delta u_t + Q \dot W_t \\ 
\,\Div u_t = 0
\end{cases}
\end{equation}
where $Q$ satisfies Assumption \ref{a:lowms}; $N \geq 3$ is an integer; $\Pi_N$ denotes the projection to 
Fourier modes with $| \cdot|_{\infty}$ norm $\leq N$; $\Hbf_N$ denotes the span of the first $N$ Fourier modes; and $\nu > 0$ is fixed and arbitrary.
\end{system}

In addition to Systems \ref{sys:2DStokes} and \ref{sys:Galerkin}, which feature fluid models subjected to white-in-time forcing, the methods easily extend to treat 
fluid systems subjected to certain types of forcing that we refer to as `OU tower noise'.
This is basically an external forcing that is a projection of an Ornstein-Uhlenbeck process on $\Real^M$. Note that the force can be $C^k_t, k \geq 1$ (see Remark \ref{rmk:CkCinfx} below).
In particular, we consider the following set of systems (stated in a general Navier-Stokes-like setting): 
\begin{system}\label{sys:Markov}
\begin{subequations} \label{eqn:finiteD}
We refer to the \emph{(generalized) Galerkin-Navier-Stokes system with OU tower noise} in $\mathbb T^d$ ($d = 2,3$) as the following stochastic ODE for $u_0 \in \Hbf_N$:
\begin{align}
  \partial_t u_t & + X(u,u) = \nu \Delta u_t + Q Z_t \\
  \partial_t Z_t & = -\mathcal{A}Z_t + \Gamma \dot{W}_t, 
\end{align}
\end{subequations}
where $Z_t \in \Hbf_M$, the operator $\mathcal{A} : \Hbf_M \to \Hbf_M$ is diagonalizable and has a strictly positive spectrum, and the bilinear term $X(u,u):\Hbf_N \times \Hbf_N \to \Hbf_N$ satisfies $u \cdot X(u,u) = 0$ and $\forall j$,  $X(e_j,e_j) = 0$. 
Note that $(u_t)$ is not Markov, but $(u_t,Z_t)$ is Markov.
\end{system}

Theorem \ref{thm:-sdecay} extends to all of Systems \ref{sys:2DStokes}, \ref{sys:Galerkin}, and \ref{sys:Markov}. 

\begin{theorem} \label{thm:FiniteDregs}
Consider any of Systems \ref{sys:2DStokes}--\ref{sys:Markov}.
Assume that $Q$ satisfies Assumption \ref{a:lowms} and that the parabolic H\"ormander condition is satisfied for $(u_t)$ or $(u_t,Z_t)$ (see e.g. \cite{Hairer11notes}).  
Then,
\begin{itemize}
\item[(i)] the Lagrangian flow \eqref{def:LagIntro} has a strictly positive Lyapunov exponent in the same sense as described in \cite{BBPS18};
\item[(ii)] all of the results of Theorem \ref{thm:-sdecay} hold. 
\end{itemize}
\end{theorem}

\begin{remark}
We have chosen to include Theorem \ref{thm:FiniteDregs} to emphasize that our methods do not fundamentally require non-spatially smooth velocity fields,
nor do they require velocity fields that are directly subjected to white-in-time forcing.
The difficulty in extending Theorem \ref{thm:-sdecay} to include $C^k_t C^\infty_x$ velocity fields is the lack of sufficiently strong hypoellipticity results in infinite-dimensions, i.e., the lack of a sufficiently strong replacement for H\"ormander's theorem. See Remark \ref{rmk:Hypo} below. 
\end{remark}

\begin{remark} \label{rmk:CkCinfx}
Theorem \ref{thm:FiniteDregs} contains many examples of $C^{k}_t C^\infty_x$ velocity fields for any $k \geq 0$.
For example, consider the following: 
\begin{align*}
u_t(x) & = \sum_{m \in \mathbb K: \abs{m}_{\infty} \leq 2} \hat{u}^m_t e_m(x) \, , 
\end{align*}
where the coefficients are given by the system 
\begin{align*}
\partial_t \hat{u}^m_t & = - \hat{u}^m_t + Z^{m,0}_t \\
\partial_t Z^{m,\ell} & = - (\ell+1) Z_t^{m,\ell} + Z_t^{j,\ell-1} \quad 1 \leq \ell \leq n \,\\
\partial_t Z^{m,n} & = - Z_t^{m,n} +\dot{W}^m_t
 .
\end{align*}
By indexing correctly, one can re-write this in the form stated in System \ref{sys:Markov} with $X \equiv 0$.
One can check that the parabolic H\"ormander condition is satisfied. This example also explains the terminology `OU tower'. 
\end{remark}

\section{Outline of the proof}\label{subsec:outline2PTDecay}

In this section, we give the main steps for the proof of Theorem \ref{thm:-sdecay}; details 
will be given in Sections \ref{sec:ExpProj}--\ref{sec:finishUp}. 
The exposition we give here focuses on the infinite-dimensional
Systems \ref{sys:NSE}, \ref{sys:3DNSE}, with the finite-dimensional systems
described in Section \ref{subsec:finiteDimIntro} addressed in a series of remarks as we go along. 

\subsection{Correlation decay via geometric ergodicity of the two point Lagrangian motion}\label{subsec:sufficesTwoPtMotion}

Our primary tool for investigating quenched correlation decay as in Theorem \ref{thm:-sdecay} is
to study the following Markov process.

\begin{definition} \label{def:2ptProc}
We define the {\it two-point Lagrangian process} $(u_t, x_t, y_t)$ on $ \Hbf \times \T^d \times \T^d$ 
by
\[
x_t = \phi^t_{u}(x) \, , \quad y_t = \phi^t_{u}(y)
\]
for fixed initial $(u, x, y) \in \Hbf \times \T^d \times \T^d$.
\end{definition}
The two-point process $(u_t, x_t, y_t)$ simultaneously tracks the velocity field $u_t$ as well as two separate
 Lagrangian flow trajectories $x_t, y_t$. Throughout, we assume $x \neq y$, hence $x_t \neq y_t$ for all $t >0$, since the diagonal 
 \[
 \Dc = \{ (x, x) : x \in \T^d \} \subset \T^d \times \T^d
 \]
 and its complement $\Dc^c$ are invariant.
 
We write $P^{(2)}_t$ for the corresponding Markov semigroup on $ \Hbf \times \Dc^c$.
Note that $\mu \times \Leb \times \Leb$ is automatically a stationary measure for the two-point process. 
We will deduce Theorem \ref{thm:-sdecay} from geometric ergodicity with respect to this measure.
\begin{theorem}\label{thm:2-pt-decay}
There exist $\hat\alpha >0$ and a measurable function $\mathcal V :  \Hbf \times \Dc^c \to [1,\infty)$, with $\mathcal V \in L^1(\mu \times \Leb \times \Leb)$ such that for each 
measurable bounded $\psi :  \Hbf \times \T^d \times \T^d$ with $\iiint \psi \dee\mu \dx\dy = 0$ 
and each $(u, x, y) \in  \Hbf \times \mathcal{D}^c$, we have that
\[
|P_t^{(2)}\psi(u, x, y) | \leq \mathcal V(u, x, y) e^{-\hat\alpha t}\left\|\psi\right\|_{L^\infty}.
\]
for all $t \geq 0$.
\end{theorem}

\begin{remark}
The idea of using geometric ergodicity of a two-point process to
deduce quenched correlation decay is known to experts in random dynamical systems, although it does
not appear to be generally well-known. To the best of our knowledge, this idea first appears in the literature
in \cite{dolgopyat2004sample} on quenched correlation decay for SDE on compact manifolds (see 
also \cite{ayyer2007quenched}).
\end{remark}

The bulk of the work in this paper is aimed at proving Theorem \ref{thm:2-pt-decay}.
Before proceeding to describe the proof, we first give an indication of how Theorem \ref{thm:2-pt-decay} will
be used to deduce Theorem \ref{thm:-sdecay}. Fix mean-zero $f, g \in L^\infty(\T^d)$
and $\hat\gamma \in (0, \frac{\hat\alpha}{2})$. For all $n \in \Z_{\geq 0}$, $u \in \Hbf$, 
\begin{equation}
\begin{aligned}
\P  \left\{\left| \int f\, (g\circ \phi^n_u)\,\dx  \right| > e^{- \hat\gamma n} \right\}
& \leq e^{2 \hat \gamma n} \E  \left( \int f\, (g\circ \phi^n_{u})\dx \right)^2   \\
& = e^{2 \hat \gamma n } \int \tilde f(x, y)\, (P_n^{(2)} \tilde g)(u, x,y) \, \dx\dy \, ,
\end{aligned}
\end{equation}
where $\tilde f(x,y) = f(x) f(y), \tilde g(x,y) = g(x) g(y)$. By Theorem \ref{thm:2-pt-decay}, 
the above RHS is bounded $\lesssim_{u,f,g} e^{(2 \hat \gamma - \hat\alpha) n}$.
We conclude by the Borel-Cantelli lemma that 
\[
\left| \int f (g \circ \phi^n_{u}) \dx \right| \leq \tilde D e^{- \hat \gamma n} \, , 
\]
for all $n \geq 1$, where $\tilde D = \tilde D(f, g, \omega, u)$ is a random constant depending on $f, g$ and the initial velocity field $u$. Additional work is needed to determine the dependence of the random constant $\tilde D$
on $f, g$ and $u$, as well as to pass from discrete to continuous time. These arguments are 
carried out in detail in Section \ref{sec:finishUp}.

\subsection{Conditions for geometric ergodicity} 

A prevailing strategy for proving correlation decay for Markov chains on noncompact spaces is to 
verify conditions guaranteeing that the Markov process 
visits a ``small'' subset of phase space 
with a positive asymptotic frequency.  
 Many criteria of this kind have been developed (sometimes called Harris theorems): for a detailed account, see, e.g., the reference \cite{meyn2012markov}. 

For obtaining geometric ergodicity of the 2-point process, we will implement a series of criteria
developed by Goldys and Maslowski \cite{goldys2005exponential}, particularly useful for semigroups generated by SPDE,
for checking the conditions of the abstract Harris-type theorems (c.f. \cite{meyn2012markov}; see also, e.g., 
\cite{bricmont2002exponential, kuksin2002coupling, kuksin2002coupling2, mattingly2002exponential, KNS18} for other approaches to geometric ergodicity for SPDE). 
A simplified version of their framework is as follows.
 Let $\mathcal Z$ be a Polish space with, and
 let $(Z_t)_{t \geq 0}$ be a continuous-time Markov process on $\Zc$ with transition kernels $P_t(z, K) = \P_z(Z_t \in K)$. 
As usual we define the Markov semigroup on observables $\psi:\Zc \to \R$ by
\begin{align*}
P_t \psi(z) = \EE_z\left(\psi(Z_t)\right) = \int_{\Zc} \psi(z') P_t(z,dz'). 
\end{align*}

\begin{condition}[Strong Feller] \label{defn:SF}
We say that a Markov process $(Z_t)$ is \emph{strong Feller} if for all $t > 0$	and bounded measurable $\psi : \Zc \to \R$, we have that $z \mapsto P_t \psi(z)$ is continuous
	on $\Zc$ for all $t > 0$. 
\end{condition}

\begin{condition}[Topological irreducibility] \label{defn:TopIrr}
	We say that a Markov process $(Z_t)$ is \emph{topologically irreducible} if for all open $U \subset \Zc$, 
	we have that $P_t(z, U) > 0$ for all $t > 0, z \in \Zc$.
\end{condition}

The next two conditions refer to a given measurable function $\Vc : \Zc \to [1,\infty)$.

\begin{condition}[Uniform lower bounds] \label{def:UnifBd}
For each $ r > 1$, there exists a compact set $K \subset \Zc$ and a time $t_0 = t_0(r) > 0$
such that
	\[
	\inf_{\{ \Vc(z) \leq r\}} P_{t_0} (z, K) > 0 \, .
	\]
\end{condition}

\begin{condition}[Drift condition] \label{defn:drift}
We say $\Vc: Z \to [1,\infty)$ satisfies a \emph{drift condition} (a.k.a. Lyapunov function) if there are constants $k, \varkappa, c > 0$ such that
	\[
	P_t \Vc \leq k e^{- \varkappa t} \Vc + c
	\]
	holds pointwise.
\end{condition}

Below, given $\Vc : Z \to [1,\infty)$, we write $C_\Vc$ for the Banach space of continuous observables $\psi : Z \to \R$ for which
the norm
\begin{align}
 \| \psi \|_{C_\Vc} := \sup_{z \in \mathcal Z} \frac{|\psi(z)|}{\Vc(z)} \label{def:CV}
\end{align}
 is finite.
\begin{theorem}[Follows from Theorem 3.1 and Lemma 3.2 in \cite{goldys2005exponential}] \label{thm:GM}
Suppose that the Markov process $(Z_t)$ and the function $\Vc: \Zc \to [1,\infty)$ satisfy Conditions \ref{defn:SF}, \ref{defn:TopIrr}, \ref{def:UnifBd} and \ref{defn:drift}.
Then, the Markov process $(Z_t)$ admits a unique stationary measure $m$, with respect
 to which $(Z_t)$ is geometrically ergodic in $C_{\Vc}$. That is, for all $\psi \in C_\Vc$, 
we have that
\[
\left| P_t \psi(z) - \int \psi \dee m \right| \leq C \Vc(z) e^{- \beta t} \| \psi \|_{C_\Vc}  \quad \text{ for all } t > 0 \, , 
\]
where $C > 0, \beta > 0$ are constants.
\end{theorem}

\begin{remark}
Harris-type theorems typically have two sets of assumptions: a \emph{minorization condition}
satisfied by a subset $S$ of phase space, which guarantees that orbits initiated from $S$ couple 
with some uniform probability $> 0$, and a \emph{drift condition} (Condition \ref{defn:drift}) which controls
excursions from $S$. 
Conditions \ref{defn:SF} -- \ref{def:UnifBd} are sufficient for verifying a suitable minorization condition
for the sublevel sets $\{ \mathcal V \leq r\}, r > 0$.
See \cite{meyn2012markov, hairer2011yet} for more discussion on abstract Harris theorems.
\end{remark}

We will apply Theorem \ref{thm:GM} to the two-point Markov process $(u_t, x_t, y_t)$
 on the state space $\Hbf \times \Dc^c$.  
Along the way, in Section \ref{sec:DefUniBd} we also apply Theorem \ref{thm:GM} to another related Markov process.

\subsection{Strong Feller and Irreducibility} \label{subsubsec:SFoutline}

Let us now begin sketching how to verify the conditions of Theorem \ref{thm:GM} for the two-point process $(u_t, x_t, y_t)$ on $\Hbf \times \Dc^c$.
First, we will prove the strong Feller property as in Condition \ref{defn:SF}
on a scale of Sobolev spaces-- this refinement is useful in verifying Condition \ref{def:UnifBd} later on.
Note that the evolution of $(x_t,y_t)$ in Definition \ref{def:2ptProc} is not subject to noise. 
After verifying the requisite uniform H\"ormander conditions (Section \ref{sec:UHC}), the proof follows by methods used previously in \cite{BBPS18}.
A brief sketch is included in Section \ref{sec:SF} for completeness. 
\begin{proposition}\label{prop:SFscaleIntro}
For any $\sigma' \in (\alpha - 2(d-1),\alpha - \frac{d}{2})$ the two-point Markov process $(u_t, x_t, y_t)$ on $\Hbf^{\sigma'} \times \Dc^c$ is strong Feller.
\end{proposition}

Next, we verify topological irreducibility as in Condition \ref{defn:TopIrr}. This follows by a relatively simple approximate controllability statement; see Section \ref{sec:Irr2pt} for the details. 
\begin{proposition}\label{prop:topIrredIntro} 
For any $\sigma' \in (\frac{d}{2} + 2, \alpha - \frac{d}{2})$,   the two-point process $(u_t, x_t, y_t)$, regarded as a process on $\Hbf^{\sigma'} \times \Dc^c$, is topologically irreducible.
\end{proposition}

\begin{remark}
	
	Checking Conditions \ref{defn:SF} -- \ref{defn:drift} for the finite-dimensional systems in
	Section \ref{subsec:finiteDimIntro} is considerably easier. 
	Conditions \ref{defn:SF} and
	\ref{defn:TopIrr} follow immediately from H\"ormander's theorem. 
	By similar arguments, Condition \ref{def:UnifBd} follows immediately as 
	long as the function $\Vc$ has the property that its sublevel sets $\{ \Vc \leq r\} \subset \Hbf \times \Dc^c$ are bounded (also away from $\mathcal{D}$), hence compact, for all $r > 1$. We return to the construction of $\Vc$ and Condition \ref{defn:drift} for these systems in Remark \ref{rmk:construftLFforFD} below.
\end{remark}

\subsection{Construction of the Lyapunov function $\Vc$}\label{subsec:outlineConstructV}

Let us now turn to the most difficult task: constructing a Lyapunov function $\Vc$ for the two-point process which satisfies Conditions \ref{def:UnifBd} and \ref{defn:drift}.

\subsubsection{Lyapunov functions for the $(u_t)$ process}

The first step is to find a suitable Lyapunov functional for the $(u_t)$ process. Unlike previous works we are aware of treating geometric ergodicity of stochastic Navier-Stokes (see, e.g., \cite{HM08, goldys2005exponential}), we need the Lyapunov functions to control $(u_t)$ in $\Hbf$ regularity.
To our knowledge, these have not previously appeared in the literature and a somewhat non-trivial additional effort is needed to deduce them (but see closely related results in Section 3.5.3 of \cite{KS}).
The resulting geometric ergodicity statements are of some independent interest. 

To simplify notation, we will denote each corresponding Lyapunov function by the same symbol $V_{\beta, \eta}$.
Lemma \ref{lem:Lyapu} is proved below in Section \ref{sec:ExpProj} (see Lemma \ref{lem:TwistBd}). 

\begin{lemma} \label{lem:Lyapu}
If $d=2$ define $\mathcal{Q} = \sup_{m = (k,i) \in \mathbb K} \abs{k}\abs{q_m}$ and if $d=3$ define $\mathcal{Q} = \sup_{m = (k,i) \in \mathbb K} \abs{q_m}$. 
Let $0 < \eta < \eta^* = \nu/ 64\mathcal{Q}$, $\beta \geq 0$, and define
\begin{align}
V_{\beta,\eta}(u) = (1 + \norm{u}_{\Hbf}^2)^{\beta}\exp\left(\eta \norm{u}_{\Wbf }^2 \right) \label{def:V}
\end{align}
where ${ \| \cdot \|_{\Wbf} }$ is as in \eqref{defineWbfNorm}. Then \eqref{def:V} satisfies Condition \ref{defn:drift} for the $(u_t)$ process.
\end{lemma}

\begin{remark} \label{rmk:FinDimLyapFunctions}
  Naturally, for Systems \ref{sys:2DStokes} and \ref{sys:Galerkin} we need only to consider the case $\beta = 0$ in Lemma \ref{lem:Lyapu}.
For the System \ref{sys:Markov}, $V$ must also depend on $Z$ ($(u_t,Z_t)$ is now the relevant Markov process). 
Let $\mathcal{A} = \Xi D \Xi^{-1}$ be the diagonalization of $\mathcal{A}$. Then, for $\eta_1,\eta_2$ chosen sufficiently small (depending on $D$, $\Xi$, $Q$, $\Gamma$, $\nu$) it suffices to take the following for $V$: 
\begin{align}
V(u,Z) = \exp\left(\eta_1 \norm{u}^2 + \eta_2 \norm{\Xi^{-1}Z}^2 \right). 
\end{align}
Indeed, note that $Y_t = \Xi^{-1}Z_t$ solves $dY_t = -D Y_t + \Xi^{-1} \Gamma \dot {W}_t$, which is compatible with the $\Xi^{-1}Z$ factor in $V$; the $QZ_t$ term in the $(u_t)$ evolution is absorbed by $-\norm{\grad  u_t}^2$ and $-\norm{D^{1/2} Y_t}^2$ (c.f. Section \ref{sec:SuperL}). 
\end{remark}

\subsubsection{Repulsion from the diagonal $\Dc$ for the two-point process}

The families of Lyapunov functions defined in Lemma \ref{lem:Lyapu} 
for the $(u_t)$ process capture repulsion from parts of phase space
where $\| u_t\|_\Hbf$ is unboundedly large. For the two-point process, however, 
we additionally require repulsion from the diagonal $\Dc$, which is considered part of ``infinity'' for the two point motion $(x_t,y_t)$. 
For this, we will crucially use the fact that near the diagonal the positive Lyapunov exponent for the Lagrangian flow $\phi^t$ (as established in \cite{BBPS18}) 
causes $x_t$ and $y_t$ to diverge from each other at an exponential rate with probability 1. 
  
To make this more precise, suppose $x,y$ are close together and 
consider the coordinate change $(x,y) \mapsto (x,w)$, where
$w = w(x,y)$ is the minimal displacement vector from $x$ to $y$. This induces a Markov process
$(u_t, x_t, w_t)$, with $w_t = w(x_t,y_t) \in \R^d$, which is continuous in time as long as $x_t, y_t$ remain close together. Note that $\T^d \times \{ 0 \}$ plays the same role for the linearized process $(x_t, w_t^*)$ as the diagonal $\Dc$ does for the two-point process $(x_t,y_t)$. Near the diagonal, we can approximate $w_t$ by the linearized process $(w_t^*)$ on $\R^d \setminus \{ 0 \}$, 
\[
w_t \approx w_t^* := D_{x} \phi^t w, \quad w = w(x, y) \, . 
\]
In \cite{BBPS18}, we proved that there exists $\lambda_1 > 0$ (deterministic) such that  
\begin{align}\label{eq:lyapgrowhtIntro}
\lim_{t \to \infty} \frac{1}{t} \log |w_t^*| = \lambda_1 > 0,
\end{align}
with probability $1$ for all initial $(u, x) \in \Hbf \times \T^d$ and $w \in \R^d \setminus \{ 0 \}$. 
Hence, one anticipates that on average, $|w_t^*|^{-p} \approx e^{- p \lambda t} |w|^{-p}$ for $p > 0$.
Thus, to capture repulsion from the diagonal it is natural to consider candidate Lyapunov functions of the form
\begin{align}\label{eq:hpHatDefnIntro}
f_p (u,x,w) = |w|^{-p} \psi_p(u,x, w / |w|) \, , 
\end{align}
where $\psi_p : \Hbf \times \T^d \times \mathbb S^{d-1} \to \R$ is nonnegative 
(actually, it is natural to enforce $\psi_p (u,x,-w) = \psi_p(u,x,w)$, so $\psi_p$ will be 
thought of as a function on the projective bundle\footnote{Here, $P \T^d \cong \T^d \times P^{d-1}$
is the projective bundle over $\T^d$, where $P^{d-1} = P (\R^d)$ is the projective space for
$\R^d$.} $\Hbf \times P \T^d$).

Let us write $T P_t$ for the Markov semigroup on $\Hbf \times \T^d \times \R^d$ corresponding to $(u_t, x_t, w_t^*)$. Repulsion as in \eqref{eq:lyapgrowhtIntro} suggests to look for an $f_p$ that is an eigenvector of $T P_t$ with eigenvalue $< 1$. Then, it is straight forward to see that we would have the following drift condition for $TP_t$:
\begin{align}\label{eq:linearDriftCondition}
T P_t f_p = e^{- \Lambda(p) t} f_p \, , 
\end{align}
for some $\Lambda(p) > 0$.  

More precisely let $V = V_{\beta, \eta}$ be as in Lemma \ref{lem:Lyapu}. We write $C_V$ for the space of continuous functions on $\Hbf\times P\T^d$ with bounded $\| \cdot \|_{C_V}$ norm and define $C_V^1$ to be the space of 
continuously differentiable functions on $\Hbf \times P \T^d$ for which
\[
\| \psi \|_{C_V^1} := \| \psi \|_{C_V} + \sup_{(u,x,v) \in \Hbf \times P \T^d}  \frac{\| D \psi(u,x,v)\|_{\Hbf^*}}{V(u)}
\]
is finite. 
As discussed in \cite{HM08}, since $C_V$ and $C^1_V$ are not separable, it is sometimes necessary to work with the spaces $\mathring{C}_V$ and $\mathring{C}_V^1$
which are the $C_V$-closure and $C_V^1$-closure of the space of smooth `cylinder functions':
\begin{align}\label{eq:defnC0inftyOutline}
	\mathring C^\infty_0(\Hbf\times P\T^d) = \{\psi \,| \, \psi(u,x,v) = \phi(\Pi_{\mathcal K} u,x,v) \,,\, \mathcal K \subset \mathbb K, \,\, \phi \in C^\infty_0\}, 
\end{align}
where $\Pi_{\mathcal K}$ denotes orthogonal projection onto $\Hbf_{\mathcal K} \cong \R^{|\mathcal K|}$. 
These spaces are separable, closed subspaces on which finite-dimensional approximation can be made. 
It can be shown that, for example, $C_{V}$ is strictly larger than $\mathring{C}_{V}$: indeed, $V_{\beta,\eta} \notin \mathring C_{V_{\beta,\eta}}$ (see Section 5.3 of \cite{HM08}).

\begin{proposition}\label{prop:psiP}
	For all $|p| \ll 1$, there exists $\psi_p \in \mathring{C}^1_V$, with the following properties:
		\begin{itemize}
			\item[(a)] $\psi_p$ is strictly positive and is bounded uniformly from below on bounded subsets of $\Hbf \times P \T^d$.
			\item[(b)] $f_p(u,x,w) = |w|^{-p} \psi_p(u,x,w / |w|)$ is an eigenfunction of $T P_t$ with eigenvalue $e^{- \Lambda(p) t}$.
			\item[(c)] As $p \to 0$ we have $\Lambda(p) = \lambda_1 p + o(p)$, where
			$\lambda_1$ the Lyapunov exponent \eqref{eq:lyapgrowhtIntro}.
			
	\end{itemize}
\end{proposition}

\begin{remark} \label{rmk:MomLyap}
The value $- \Lambda(p)$ is sometimes referred to as the \emph{moment Lyapunov exponent} \cite{arnold1984formula},
and arises naturally in the study of large deviations in the convergence of Lyapunov exponents-- see, e.g., 
\cite{ arnold1986lyapunov, arnold1987large}. As we show, it satisfies the formula
\[
- \Lambda(p) = \lim_{t \to \infty} \frac{1}{t} \log \E |D_x \phi^t_u v|^{-p}
\]
for all $(u,x,v) \in \Hbf \times P \T^d$. The connection with repulsion from the diagonal $\Dc$ for the
(nonlinear) two-point process originates in the work of Baxendale and Stroock \cite{baxendale1988large},
which obtains large deviations estimates from the diagonal for the two-point process associated to a hypoelliptic SDE on a compact Riemannian manifold (see also \cite{baxendale1993kinematic}). In contrast, the unbounded and infinite-dimensional phase space in our setting introduces numerous challenges to be overcome; moreover, our techniques in handling the two-point process are notably different and more direct. 
\end{remark}

\subsubsection{Verifying the drift condition}

Returning to the nonlinear two-point motion, in light of \eqref{eq:hpHatDefnIntro}, it is natural to look for a Lyapunov function of the form
\begin{align}\label{eq:fullLyapFunctIntro}
\Vc(u,x,y) = h_p(u,x,y) + V_{\beta + 1, \eta}(u) \, , 
\end{align}
where $\beta$ is sufficiently large, $\eta \in (0,\eta^*)$, and $h_p(u,x,y) := \chi(|w(x,y)|) f_p(u, x, w(x,y))$,
where $\chi$ is a smooth cutoff satisfying $\chi|_{[0,1/10]} \equiv 1, \chi|_{[1/5,\infty)} \equiv 0$.

For $\Vc$ as above, let us sketch how to satisfy Conditions \ref{def:UnifBd} and \ref{defn:drift}. 
For Condition \ref{def:UnifBd}, by parabolic regularity and the strong Feller property 
in Condition \ref{defn:SF}, it suffices (Lemma \ref{lem:unifBoundedCompact12312}) to show that for any $r > 0$, the sublevel set $\{\Vc \leq r\}$
satisfies
\begin{align}\label{eq:containIntro}
\{ \Vc \leq r \} \subset \{ \| u \|_{\Hbf} \leq R_1 \text{ and } d(x,y) \geq R_2\} \, , 
\end{align}
where $R_i = R_i(r), i = 1,2$. We are guaranteed this as long as $\psi_p$ is uniformly bounded
from below on bounded sets as in Proposition \ref{prop:psiP}(c). 

For Condition \ref{defn:drift}, 
the validity of the linear approximation depends on the size of the velocity field. Since the latter can be arbitrarily large, the time scale on which this approximation can be used is nonuniform in $u$. For this reason, it is convenient to work at the time-infinitesimal level, i.e., with infinitesimal generators. 
Temporarily neglecting technical issues regarding the domain, let $\mathcal L_{(2)} = \lim_{t \to 0} \frac{1}{t}\left(P^{(2)}_t - \operatorname{Id}\right)$ be the infinitesimal generator of the two-point motion. 
We will essentially show that 
\begin{align}\label{desiredDriftCond}
\mathcal L_{(2)} \Vc \leq - \Lambda(p) \Vc + C
\end{align}
for some $C > 0$. 
As one may expect, \eqref{desiredDriftCond} cannot be rigorously justified 
exactly as such in infinite dimensions due to the
fact that $V \notin \mathring C_{V}$, but standard arguments are used to obtain an 
almost equivalent analogue (see Section \ref{sec:2ptDrift}). 

Let us briefly sketch this argument, ignoring technical issues. 
When $\mathcal L_{(2)}$ hits $h_p$, we obtain
\begin{align}\label{ineq:LinGenIntro}
\mathcal L_{(2)} h_p \leq - \Lambda(p) h_p + C' V_{\beta + 1, \eta} \, ; 
\end{align}
the first term is good and comes from the fact that $TP_t f_p = e^{- \Lambda(p) t} f_p$
and that $T P_t$ well-approximates $P^{(2)}_t$ near the diagonal, while the second term is an 
error coming from the linearization approximation.

\newcommand{\Lc}{\mathcal L}

Unlike \eqref{desiredDriftCond}, we are able to show that $h_p$ is in the domain of $\mathcal L_{(2)}$ and rigorously justify \eqref{ineq:LinGenIntro} (see Lemma \ref{lem:approxhp}).
In order to make this perturbation argument, we crucially need that $\psi_p \in \mathring C_V^1$,
$V = V_{\beta, \eta}$.

For $\Vc$ we effectively deduce
\[
\Lc_{(2)} \Vc \leq - \Lambda(p) h_p + C' V_{\beta + 1, \eta} + \Lc V_{\beta + 1, \eta}
\]
where $\Lc$ denotes the (formal) infinitesimal generator of the $(u_t)$ process (see Definition \ref{def:formal-gen}).
From here, we will absorb the $C' V_{\beta + 1, \eta}$ linearization error
into $\Lc V_{\beta + 1, \eta}$ by showing a ``super drift condition'': formally, we can essentially view it as $\forall \kappa > 0$,
$\exists C_\kappa > 0$ such that
\[
	\Lc V_{\beta + 1, \eta} \leq - \kappa V_{\beta + 1, \eta} + C_\kappa.
\]
In other words, the drift condition for $V$ as in Condition \ref{defn:drift} can be taken with as strong an exponential decay rate as desired. This is much stronger than a standard drift condition -- see \eqref{eq:generator-ineq} for details (also Remark \ref{rmk:SuperL}). Taking $\kappa > 0$ sufficiently large
absorbs the $C' V_{\beta + 1, \eta}$ error term, resulting in \eqref{desiredDriftCond}.

The above sketch is far from complete and there are many details to fill in make it rigorous, including an appropriate
$C_0$ semigroup framework for $P^{(2)}_t$ and issues relating to the domain of $\mathcal{L}_{(2)}$. Much of the difficulty in carrying out this argument rigorously stems from infinite dimensionality of $\Hbf \times \Dc^c$ -- see Section \ref{sec:2ptGeoErg} for a detailed discussion.  Nevertheless, we ultimately verify the following. 

\begin{proposition} \label{prop:2ptDrift}
The Lyapunov function $\Vc$ satisfies Condition \ref{defn:drift}, i.e.,  there exists $K > 0$ such that  
\begin{align}\label{eq:driftCond222}
P^{(2)}_t \Vc \leq e^{-\Lambda(p) t}\Vc + K. 
\end{align}
\end{proposition}
With Condition \ref{defn:drift} in place, Theorem \ref{thm:2-pt-decay} now follows.

\subsubsection{Outline of the proof of Proposition \ref{prop:psiP}}

In our analysis, it is convenient to transform the eigenproblem $T P_t f_p = e^{- \Lambda t} f_p$, with $f_p$ as in \eqref{eq:hpHatDefnIntro}, into the equivalent problem $\hat P^p_t \psi_p = e^{- \Lambda t} \psi_p$, where $\hat P^p_t$ denotes the following `twisted' Markov semigroup acting on observables $\psi : \Hbf \times P \T^d \to \R$:
\begin{equation}
\begin{aligned}
\hat P^p_t \psi (u, x, v) = \E _{(u,x,v)}\left( |D_x\phi^t v|^{-p} \,\psi(u_t, x_t, v_t)\right) \, . \label{def:TwistMark}
\end{aligned}
\end{equation}
Above, we write $(u_t, x_t, v_t)$ for the \emph{projective process} on $\Hbf \times P \T^d$, where $v_t \in P^{d-1}$ denotes the projective class of the vector $w^*_t = D_{x} \phi^t w$. We will refer to the Markov semigroup associated with the projective process as $\hat P_t = \hat P_t^0$. The projective process is natural here and also plays a major role in the proof of the positive Lyapunov exponent in \cite{BBPS18}.

 Ultimately, we seek to define $\psi_p$ to be an eigenfunction corresponding to the dominant
 eigenvalue of the semigroup $\hat P_t^p$ in some function space. 
Since we require $\psi_p \in \mathring C_V^1$, it is natural to consider spectral theory 
for the semigroup $\hat P_t^p$ on $\mathring C_V^1$. 
 To work in this framework, however, entails significant technical problems. 
To start, it is already a challenge to prove that $\hat P_t$, let alone $\hat P_t^p$, 
is bounded on $C_V^1 \to C^1_V$ for any value of $t > 0$. Moreover, 
we are unable to show $C_0$ continuity for $\hat P_t^p$ in $\mathring C_V^1$. We note that
this problem does not arise when working on the lower regularity space $\Wbf$ as in
 \cite{HM08}.
 
We will show that $\hat P^p_{T_0}: C^1_V \to C^1_V$ is bounded for a sufficiently large $T_0$,
permitting us to define $\psi_p$ to be an eigenfunction for of the discrete-time operator $\hat P_{T_0}^p$.
Then, we show that $\hat P_t^p$ defines a $C_0$ semigroup on $\mathring{C}_V$ and the relation $\hat P_t^p \psi_p = e^{- \Lambda(p) t} \psi_p, t \geq 0$ as in Proposition 
\ref{prop:psiP}(b) is proved via semigroup theory.

Let us now make this more precise.
\begin{definition}[Spectral gap]
Let $A$ be a bounded linear operator on a Banach space $\mathfrak{B}$ with simple leading eigenvalue $r$. We say $A$ has a {\em spectral gap} if there exists an $\ep>0$ such that
\[
	\sigma(A)\backslash\{r\} \subseteq B_{\abs{r}-\ep}(0).
\]
\end{definition}

For $T_0$ sufficiently large and all $p$ sufficiently small, 
we will show that $\hat P_{T_0}^p$ is bounded on $C_V^1$
and maps $\mathring C_V^1$ into itself (Lemmas \ref{lem:specPicCV1} and \ref{lem:OneRingToRuleThemAll}). 
We then seek to prove that $\hat P_{T_0}^p$ has a 
spectral gap in $\mathring C_V^1$, and then construct $\psi_p$ as an eigenfunction
of $\hat P_{T_0}^p$ corresponding to its leading eigenvalue (see \eqref{eq:definePsiPOutline} below).
 This is deduced from a spectral gap for the `untwisted' operator $\hat P_{T_0}$ in $\mathring C_V^1$
 and a spectral perturbation argument (see Lemma \ref{lem:specPerturbLF}) using the fact that $\hat P_{T_0}^p
 \to \hat P_{T_0}$ in the $C_V^1$ norm as $p \to 0$ (Lemma \ref{lem:specPicCV1}).
 
In turn, a spectral gap for $\hat P_{T_0}$ in $C_V^1$ is deduced from a combination of 
a spectral gap for $\hat P_{T_0}$ in $C_V$ and a Lasota-Yorke type gradient estimate (Proposition \ref{prop:LY}) analogous to those used to prove asymptotic strong Feller in e.g. \cite{HM06,HM11}. The $C_V$ spectral gap is deduced from Theorem \ref{thm:GM}
and essentially follows from ingredients already needed elsewhere (see Section \ref{sec:DefUniBd}).
The  gradient estimate is proved via an adaptation of the the Malliavin calculus scheme in \cite{HM11} adapted to our more complicated nonlinearity, $\Hbf$ regularity, and the $\Hbf \times P \mathbb T^d$ geometry. A brief sketch is provided in Section \ref{sec:C1VSpecProj}.
 In all, these arguments imply the following.
 
\begin{proposition}\label{prop:specgapC1V}
Let $V = V_{\beta,\eta}$ be as in Lemma \ref{lem:Lyapu}, where 
$\eta \in (0, \eta^*)$ is arbitrary and $\beta$ is sufficiently large. Then, there exist $T_0 > 0, p_0 > 0$ such that for all $p\in [- p_0,p_0]$, $\hat P_{T_0}^p$ has a spectral gap on $\mathring{C}_V^1$ with  
real leading eigenvalue $e^{- T_0 \Lambda(p)}$ for some $\Lambda(p) \in \R$. 

Moreover, the limit
\begin{align} \label{eq:definePsiPOutline}
	\psi_p := \lim_{n\to \infty} e^{n T_0 \Lambda(p)}\hat{P}^p_{nT_0}{\bf 1} 
\end{align}
exists in $C_V^1$ and defines an eigenfunction corresponding to the leading eigenvalue.
\end{proposition}

As observed earlier, to prove that $\psi_p$ is an eigenfunction of $\hat P^p_t$ for all $t \geq 0$ as in 
Proposition \ref{prop:psiP}(b), we cannot work on $\mathring{C}^1_V$. Instead we will regard $\hat{P}^p_t$ as a semi-group on $\mathring{C}_V$: here, $\hat{P}^p_t$ is bounded for all $t \geq 0$ and gives rise to a $C_0$ semi-group on $\mathring{C}_V$ (Proposition \ref{prop:C0-property-twist}). Carrying out an analogous spectral perturbation argument to Proposition \ref{prop:specgapC1V} and using spectral theory for $C_0$ semigroups gives the following.

\begin{proposition}\label{prop:specgapC1V-allt} 

Let $V, p_0$, $\Lambda(p)$ and $\psi_p$ be as in Proposition \ref{prop:specgapC1V}. 
For all $p\in [- p_0,p_0]$, $\hat{P}^p_t$ has a spectral gap on $\mathring{C}_V$ for all $t > 0$
with leading eigenvalue $e^{- \Lambda(p) t}$, where $\Lambda(p)$ is as in Proposition \ref{prop:specgapC1V}.
Moreover, $\psi_p$ is an eigenfunction for $\hat P_t^p$ corresponding to this eigenvalue, and 
\begin{align}\label{eq:contTimeSemigroupOutline}
	\hat{P}_t^p \psi_p = e^{-\Lambda(p)t}\psi_p.
\end{align}
\end{proposition}

At this point, we have the information required to prove Proposition \ref{prop:psiP}. 
The formula \eqref{eq:definePsiPOutline} implies that $\psi_p$ is nonnegative; positivity of
$\hat P^p_{T_0}$ and irreducibility of $(u_t, x_t, v_t)$ (Proposition \ref{prop:topIrredProj}) imply 
that $\psi_p > 0$ pointwise. The uniform positivity condition in Proposition \ref{prop:psiP} (a)
follows from a compactness argument-- see Lemma \ref{lem:unifLowerBoundPsiP}.
     
To prove the asymptotic $\Lambda(p) = \lambda_1 p + o(p)$ as in Proposition \ref{prop:psiP}(c):
first, we check that $p \mapsto \Lambda(p)$ is differentiable for $p$ sufficiently small (Lemma \ref{lem:diffy}) 
by combining \eqref{eq:contTimeSemigroupOutline} with Fr\'echet differentiability of $p \mapsto \hat P_t^p$
in the $C_V$ operator norm. The formula $\Lambda'(0) = \lambda_1$ is then checked
using a convexity argument and the characterization given in Remark \ref{rmk:MomLyap}-- 
see Lemma \ref{lem:verifyItemBLF} for details. 

\begin{remark} \label{rmk:DefH}
Note that the twisted Markov semi-group in \eqref{def:TwistMark} can be written as
\[
	\hat{P}^p_t \psi(u,x,v) = \E_{(u,x,v)} \left(\exp\left(-p\int_0^t H(u_s,x_s,v_s)\ds\right)\psi(u_t,x_t,v_t)\right),
\]
where
\[
	H(u,x,v) := \langle v,D u(x) v\rangle,
\]
so that $\hat P^p_t$ is the Feynman-Kac semi-group with potential $H$ (see, e.g., \cite{varadhan1984large}).
This particular $H$ and the resulting moment Lyapunov exponent $\Lambda(p)$ are 
useful in studying large deviations in the convergence of Lyapunov exponents for random dynamics-- see, e.g., \cite{arnold1986lyapunov, arnold1984formula, arnold1987large}. 
\end{remark}

\begin{remark}\label{rmk:construftLFforFD}
For the finite-dimensional 
models defined in Section \ref{subsec:finiteDimIntro}, the construction of $\psi_p$
and the corresponding Lyapunov function $\mathcal V$ can be simplified (though the above proof certainly suffices).
Naturally, to treat System \ref{sys:Markov}, everywhere there is $(u_t)$ one must instead use $(u_t,Z_t)$. 
The conditions in Theorem \ref{thm:GM} for geometric ergodicity for the projective process $(u_t, x_t, v_t)$ (or $(u_t,Z_t,x_t,v_t)$) follow more-or-less immediately from hypoellipticity 
(see \cite{BBPS18}) and the existence of Lyapunov functions for the $(u_t)$ or $(u_t,Z_t)$ processes (Remark \ref{rmk:FinDimLyapFunctions}).
Moreover, $\hat P_t$ defines a $C_0$ semigroup on $C_V$.
The spectral picture for the twisted Markov kernels $\hat P^p_t$ follows as discussed above (though simpler to check). 

This provides, $0 < \abs{p} \ll 1$, the $\psi_p$ as in Proposition \ref{prop:psiP} (in the OU tower case, it also depends on $Z$).
It is not necessary to study a spectral gap in $C^1_V$, as $\psi_p$ is $C^\infty$ by H\"ormander's theorem and $D\psi_p \in C_V$ follows from a more refined use of hypoelliptic regularity using the scale of available $V$. The construction of $\Vc$ proceeds as above, as does the generator argument for
checking the drift condition \eqref{eq:driftCond222} (though these are much easier to justify).
This completes the proof of Theorem \ref{thm:2-pt-decay} for these models.

\end{remark}

\begin{remark} \label{rmk:Hypo} 
There are two fundamental places in our proof where we depend on the non-degenerate noise in Assumption \ref{a:Highs}. 
The first is in the Furstenberg criterion [Theorem 4.7, \cite{BBPS18}] for the positive Lyapunov exponent \eqref{eq:lyapgrowhtIntro}. 
The Lasota-Yorke-type gradient bounds used in the asymptotically strong Feller frameworks of \cite{HM06,HM11} (and Section \ref{sec:C1VSpecProj} below) do not seem to be the correct tool for obtaining this criterion. We are currently unsure what suitable hypoelliptic smoothing can be applied aside from strong Feller. 
The other place is in this work. The twisted Markov semigroup $\hat{P}^p_t$ does not seem to define a bounded semigroup on $C^1_V$ as currently required by the weak Harris theorem frameworks of \cite{HM08,hairer2011yet} due to the higher regularity needed for our methods. 
This forces a `strong' Harris theorem-type framework as in Theorem \ref{thm:GM}, which requires non-degenerate noise such as Assumption \ref{a:Highs} (at least at high frequencies). 
\emph{All of the other} uses of the strong Feller property/non-spatially smooth noise in this work and our previous \cite{BBPS18} are not of fundamental importance and could be easily eliminated with some well-understood methods at the expense of additional technical complexity (e.g. those discussed in \cite{HM08,GHHM18} and the references therein).
In these places, we use the non-degenerate noise only to reduce the length and complexity
of the present manuscript.
\end{remark}

\subsection{Notation} \label{sec:Note} 

We use the notation $f \lesssim g$ if there exists a constant $C > 0$ such that $f \leq Cg$ where $C$ is independent of the parameters of interest. Sometimes we use the notation $f \approx_{a,b,c,...} g$ to emphasize the dependence of the implicit constant on the parameters, e.g. $C = C(a,b,c,...)$. We denote $f \approx g$ if $f \lesssim g$ and $g \lesssim f$. Throughout, $\R^d$ is endowed with the standard Euclidean inner product $(\cdot, \cdot)$ and corresponding norm $|\cdot|$. We continue
to write $|\cdot|$ for the corresponding matrix norm. When the domain of the $L^p$ space is omitted it is understood to be $\T^d$: $\norm{f}_{L^p} = \norm{f}_{L^p(\T^d)}$. We use the notations $\EE X = \int_{\Omega} X(\omega) \PP(d\omega)$ and $\norm{X}_{L^p(\Omega)} = \left(\EE \abs{X}^p \right)^{1/p}$. When $(z_t)$ is a Markov process, we write $\EE_z, \P_z$ for the expectation and probability, respectively, conditioned on the event $z_0 = z$. We use the notation $\norm{f}_{H^s} = \sum_{k \in \Z^d} \abs{k}^{2s} \abs{\hat{f}(k)}^2$ (denoting $\hat{f}(k) = \frac{1}{(2\pi)^{d/2}} \int_{\T^d} e^{-ik\cdot x}f(x) dx$ the usual complex Fourier transform).  We occasionally use Fourier multiplier notation $\widehat{m(\grad) f}(\xi) := m(i\xi) \hat{f}(\xi)$. Additionally, we will often use $r_0$ to denote a number in $(\frac{d}{2}+1,3)$ such that the Sobolev embedding $H^{r_0} \hookrightarrow W^{1,\infty}$ holds.

We denote $P \T^d \cong \T^d \times P^{d-1}$  for projective bundle. 
We are often working with the Hilbert spaces $\Wbf \times T_v P \T^d$ and $\Hbf \times T_v P \T^d$.
For these spaces we denote the inner product $\brak{\cdot,\cdot}_{\Wbf}$ (respectively $\Hbf$) and correspondingly for the norms as the finite-dimensional contribution to the inner product is unambiguous. For linear operators $A: \Wbf \times T_v P \T^d \to \Wbf \times T_v P \T^d$ we similarly denote the operator norm $\norm{A}_{\Wbf}$ and for linear operators $A:\Wbf \times T_v P\T^d \to \Real$ we use the notation $\norm{A}_{\Wbf^\ast}$ (analogously for $\Hbf$).
For $\mathcal{K} \subset \mathbb K$, define $\Pi_{\mathcal{K}}: \Wbf \times P \T^d \to \mathcal{K} \times P\T^d$ to be the orthogonal projection onto the subset of modes in $\mathcal{K}$.
For $n \in \mathbb N$, $\Pi_n$ denotes the orthogonal projection onto the modes with $k \in \mathbb K$, $\abs{k} \leq n$. 

\section{Preliminaries, drift conditions and Jacobian estimates} \label{sec:ExpProj}

\subsection{Preliminaries} \label{sec:Prelim}
We will write the Navier-Stokes system as an abstract evolution equation on $\Hbf$ by 
\begin{equation}\label{eq:NS-Abstract}
	\partial_t u + B(u,u) + Au = Q\dot{W} = \sum_{m\in\mathbb{K}} q_m e_m \dot{W}^m, 
  \end{equation}
where 
\begin{align}
B(u,v) &= \left(\Id - \grad (-\Delta)^{-1} \grad \cdot \right)\grad \cdot (u \otimes v) \\
Au &= \begin{cases} -\nu \Delta u \quad&  \textup{ if } d=2 \\
-\nu' \Delta u + \nu\Delta^2 u \quad& \textup{ if } d=3. 
\end{cases}
\end{align}
The $(u_t)$ process with initial data $u$ is defined as the solution to \eqref{eq:NS-Abstract} in the mild sense \cite{KS,DPZ96}:
\begin{align}
u_t = e^{-tA}u - \int_0^t e^{-(t-s)A} B(u_s,u_s) ds + \int_0^t e^{-(t-s)A} Q \dee W(s) \, ,  \label{eq:Mild} 
\end{align}
where the above identity holds $\P$ almost surely for all $t>0$. For \eqref{eq:Mild}, we have the following well-posedness theorem.
\begin{proposition}[\cite{KS,DPZ96}] \label{prop:WPapp}
For each of Systems \ref{sys:NSE}--\ref{sys:3DNSE}, we have the following.
 For all initial $u \in \Hbf \cap \Hbf^\gamma$ with $\gamma < \alpha-\frac{d}{2}$ and all $T> 0, p \geq 1$, there exists a $\P$ almost surely unique solution $(u_t)$ to \eqref{eq:Mild}. Moreover, $(u_t)$ is $\mathscr{F}_t$-adapted, and belongs to $L^p(\Omega;C([0,T];\Hbf \cap \Hbf^\gamma)) \cap L^2(\Omega;L^2(0,T;\Hbf^{\gamma+(d-1)}))$. 

Additionally, for all $p \geq 1$ and $0 \leq \gamma < \gamma' < \alpha - \frac{d}{2}$,  
\begin{align}
\EE \sup_{t \in [0,T]} \norm{u_t}_{\Hbf^\gamma}^p & \lesssim_{T,p,\gamma} 1 + \norm{u_0}_{\Hbf \cap \Hbf^\gamma}^p \\
\EE \int_0^T \norm{u_s}_{\Hbf^{\gamma + (d-1)}}^2 ds & \lesssim_{T,\delta} 1 + \norm{u_0}^2_{\Hbf^\gamma} \\ 
\EE \sup_{t \in [0,T]} \left(t^{\frac{\gamma'-\gamma}{2(d-1)}} \norm{u_t}_{\Hbf^{\gamma'}}\right)^p &\lesssim_{p,T,\gamma,\gamma'} 1 + \norm{u_0}^p_{\Hbf^\gamma}. \label{ineq:locRegu}
\end{align}
\end{proposition}

Before proceeding, we need the following useful Sobolev interpolation inequalities. 
\begin{lemma} \label{lem:SobTrick}
For all $r \in (0,3)$, $\forall \sigma > 2$, and $\forall \ep > 0$, $\exists C =C(r,\ep,\sigma) > 0$ such that $\forall u \in \Hbf^{\sigma+1}$,
\begin{align}
\norm{u}_{\Hbf^r} \leq \ep \norm{\Delta u}_{L^2}^2 + \ep \frac{\norm{\grad u}_{\Hbf^\sigma}^2}{1 + \norm{u}^2_{\Hbf^\sigma}} + C. \label{ineq:SobTrick}
\end{align}

\end{lemma}
\begin{proof}
Without loss of generality, we can assume $r > 2$ as otherwise the estimate is immediate. 
Similarly, we assume that $\norm{u}_{\Hbf^r} > \ep \norm{\Delta u}_{L^2}^2$, as otherwise the estimate is automatic.  
By Sobolev interpolation, 
\begin{align}
\norm{u}_{\Hbf^r} \lesssim \norm{\Delta u}_{L^2}^{1-\theta} \norm{\grad u}_{\Hbf^\sigma}^\theta \leqc \norm{u}_{\Hbf^{r}}^{\frac{1-\theta}{2}}\norm{\grad u}_{\Hbf^\sigma}^\theta
\end{align}
where $\theta = \frac{r-2}{\sigma-1}$.
Therefore,
\begin{align}
\norm{u}_{\Hbf^r}  \lesssim \norm{\grad u}_{\Hbf^\sigma}^{\frac{2\theta}{1+ \theta}}. \label{ineq:HsHsigSobTrick1}
\end{align}
A similar argument also implies that
\begin{equation}
\norm{u}_{\Hbf^\sigma} \lesssim \norm{u}^{\frac{1-\gamma}{2}}_{\Hbf^r}\norm{\nabla u}_{\Hbf^\sigma}^{\gamma}\leqc\norm{\grad u}_{\Hbf^\sigma}^{\frac{\gamma + \theta}{1 + \theta}}, \label{ineq:HsHsigSobTrick2}
\end{equation}
where $\gamma = \frac{\sigma-2}{\sigma -1 }$. It follows from \eqref{ineq:HsHsigSobTrick1} and \eqref{ineq:HsHsigSobTrick2} that
\[
	\frac{\norm{\nabla u}_{\Hbf^\sigma}^2}{1+\|u\|_{\Hbf^\sigma}^2} \geqc \norm{\nabla u}_{\Hbf^\sigma}^{\frac{2(1-\gamma)}{1+\theta}}\geqc \norm{u}_{\Hbf^r}^{\frac{1-\gamma}{\theta}}.
\]
Inequality \eqref{ineq:SobTrick} follows from the fact that for all $\ep> 0$ there exists a $C = C(\ep,\sigma,r)$ such that
\[
	\norm{u}_{\Hbf^r} \leq \ep\norm{u}_{\Hbf^r}^{\frac{1-\gamma}{\theta}} + C
\]
provided $\frac{1-\gamma}{\theta} = \frac{1}{r-2} > 1$, which holds if and only if $r < 3$.

\end{proof}

\begin{lemma} \label{lem:SobLogTrick2}
For all $\ep>0$, $\sigma \geq 1$, there exists a $C_\ep> 0$ such that the following holds for all $u \in \Hbf^{\sigma+1}$, 
\begin{align}
  \log \left(1 + \norm{u}_{\Hbf^\sigma}^2\right) \leq  \ep \norm{\Delta u}_{L^2}^2 + \ep \frac{\norm{\grad u}_{\Hbf^\sigma}^2}{1 + \norm{u}^2_{\Hbf^\sigma}} + C_\ep. \label{ineq:SobTrick2}
\end{align}
\end{lemma}
\begin{proof}
For all $\ep,\delta > 0$, $\exists C_{\delta,\ep} > 0$ such that the following holds $\forall x \geq 0$, 
\begin{align}
\log \left(1 + x^2\right) \leq  C_{\delta,\ep} + \ep x^\delta. 
\end{align}
Similar to Lemma \ref{lem:SobTrick}, we may assume $\sigma >1$ and 
\begin{align}
\norm{u}_{\Hbf^\sigma}^\delta & > \ep \norm{\Delta u}_{L^2}^2. \label{ineq:HsigRestr}
\end{align}
By Sobolev interpolation (for suitable $\theta \in (0,1)$),  
\begin{align}
\norm{\Delta u}_{L^2} \leq \norm{u}_{\Hbf^\sigma} \lesssim \norm{\Delta u}_{L^2}^{1-\theta}\norm{\grad u}_{\Hbf^\sigma}^{\theta}, 
\end{align}
and hence by \eqref{ineq:HsigRestr}, there holds 
\begin{align}
\norm{u}_{\Hbf^\sigma}^{1 - \frac{1}{2}\delta(1-\theta)} \lesssim \norm{\grad u}_{\Hbf^\sigma}^{\theta},  
\end{align}
which implies
\begin{align}
\frac{\norm{\grad u}_{\Hbf^\sigma}^2}{1 + \norm{u}_{\Hbf^\sigma}^2} \geqc  \norm{\grad u}_{\Hbf^\sigma}^{2 - \frac{4\theta}{2 - \delta(1-\theta)}}.
\end{align}
The desired estimate follows taking $\delta$ sufficiently small. 
\end{proof}

\subsection{Fundamental estimates and super-Lyapunov property} \label{sec:SuperL}
Let $C^2(\Hbf)$ be the space of twice Frechet differentiable functions on $\Hbf$. We denote by $D_u G(u)$ the Frechet derivative at $u$, which is a linear functional on $\Hbf$ defined for each $v\in \Hbf$ by
\[
	D_uG(u)v= \lim_{h\to 0} h^{-1}(G(u+ hv) - G(u)).
\]
Likewise the second Frechet derivative $D^2_u G(u)$ can be identified is a bilinear form on $\Hbf$ defined for each $v,w\in \Hbf$ by
\[
	 D^2_uG(u)[v,w] = \lim_{h\to 0} h^{-1}(D_uG(u + hw)v - D_uG(u)v).
\]

We will define the formal generator of the process \eqref{eq:NS-Abstract} as follows:
\begin{definition}\label{def:formal-gen}
Suppose $G\in C^2(\Hbf)$ has the property that for each $u\in \Hbf^{\sigma+1}$, and $v\in \Hbf^{\sigma-1}$, the mapping $u\mapsto D_u G(u)v$ is continuous on $\Hbf^{\sigma + 1}$ and satisfies
\[
	|D_uG(u)v| \leq K(\|u\|_{\Hbf})\|u\|_{\Hbf^{\sigma+1}}\|v\|_{\Hbf^{\sigma-1}}
\]
for some continuous function $K(r)$, $r>0$. We define the formal generator $\mathcal{L}$ of the process \eqref{eq:NS-Abstract} to be the linear operator acting on $G$ such that for each $u\in \Hbf^{\sigma +1}$,

\begin{equation}\label{eq:formal-gen}
\mathcal{L}G(u) := -D_u G(u)(B(u,u) +A u)  + \frac{1}{2}\sum_{m\in \mathbb{K}}|q_m|^2 D^2_uG(u)[e_m,e_m].
\end{equation} 
\end{definition}

\begin{remark}
One cannot identify $\mathcal{L}$ as the infinitesimal generator of the Markov semi-group associated to the \eqref{eq:NS-Abstract}, but they coincide on a core of smooth cylinder functions (see \cite{HM08} Lemma 5.11). Note that one can also apply $\mathcal{L}$ to functions that do not necessarily belong to the domain of the generator (see \cite{HM08} Remark 5.7 for a discussion of this). 
\end{remark}

The next lemma provides the energy estimate that eventually implies the drift condition on $V_{\beta,\eta}$. 
\begin{lemma} \label{lem:GenV}
For all $V = V_{\beta,\eta}$ with $\beta \geq 2$ and $\eta > 0$, there exists a $C = C(\beta,\eta) > 0$ such that the following estimates hold for all $u\in \Hbf^{\sigma+1}$
\begin{align}
\mathcal{L} \log V(u) \leq -\nu \eta \norm{\Delta u}_{L^2}^2 - \nu \beta \frac{\norm{\grad u}_{\Hbf}^2}{1 + \norm{u}_{\Hbf}^2} + C, \label{ineq:GenLogV}
\end{align}
\end{lemma}
\begin{proof}
First consider \eqref{ineq:GenLogV}. For simplicity, consider the case $d=2$; the case $d =3$ follows similarly with only trivial modification.  
By direct calculation, we find
\[
	D_u \log V(u)v = 2\eta\langle u, v\rangle_{\Wbf} + \frac{2\beta \langle u, v\rangle_{\Hbf}}{1+\|u\|_{\Hbf}^2},
\]
and
\[
	D_u^2\log V(u)[v,v] = 2\eta\|v\|_{\Wbf}^2 + \frac{2\beta \|v\|_{\Hbf}^2}{1+\|u\|_{\Hbf}^2} - \frac{4\beta |\langle u,v\rangle_{\Hbf}|^2}{(1 + \|u\|_{\Hbf}^2)^2},
\]
where the $\langle\cdot,\cdot\rangle_{\Hbf}$ inner product is to be interpreted as the natural pairing between $\Hbf^{\sigma+1}$ and $\Hbf^{\sigma-1}$. Consequently, we have
\begin{equation}
\begin{aligned}
\mathcal{L}\log V & \leq  \underbrace{-2\eta \brak{u,B(u,u) + Au}_{\Wbf}}_{T_\omega}+ \underbrace{-\frac{2\beta \brak{u,B(u,u) + Au}_{\Hbf}}{1 + \norm{u}_{\Hbf}^2}}_{T_\sigma}\\
& \quad + \underbrace{\sum_{k \in \Z^d_0} |q_k|^2\left(\eta \|e_k\|_{\Wbf}^2 + \frac{\beta\|e_k\|_{\Hbf}^2}{1 + \norm{u}_{\Hbf}^2}\right)}_{T_{Q}}. 
\end{aligned}
\end{equation}
By the divergence-free structure and standard $\Hbf^\sigma$ estimates, $\exists C_\sigma > 0$ such that (see e.g. \cite{MajdaBertozzi}) the following holds for some $r \in (\frac{d}{2}+1,3)$, 
\begin{align}
T_{\omega} & = -2\nu\eta \norm{\Delta u}_{L^2}^2, \\ 
T_{\sigma} & \leq -2\nu \beta \frac{\norm{\grad u}_{\Hbf}^2}{1 + \norm{u}_{\Hbf}^2} + C_\sigma \norm{u}_{\Hbf^{r}}.
\end{align}
The term involving $\|u\|_{\Hbf^r}$ is then estimated by applying Lemma \ref{lem:SobTrick}. Similarly, we have by $\|e_k\|_{\Wbf} \leq |k|$, $\|e_k\|_{\Hbf} = |k|^{\sigma}$ and $|q_k| \leqc |k|^{-\alpha}$ and $\alpha - \sigma > d/2$ that
\[
	|T_Q| \leq \sum_{k\in \Z^d_0} |q_k|^2 \left(\eta \|e_k\|_{\Wbf}^2 + \beta \|e_k\|_{\Hbf}^2\right) \leqc_{\beta,\eta} \sum_{k\in \Z^d_0} |k|^{-2(\alpha - \sigma)} \leqc_{\beta,\eta} 1. 
\]
\end{proof}

Lemma \ref{lem:GenV} allows us to prove the following bound on $(u_t)$ a solution to either of Systems \ref{sys:NSE} or \ref{sys:3DNSE}. The following Lemma is fundamental to our analysis and is used repeatedly in what follows. It controls not only an improvement in the moments of $V(u_t)$ but also controls the inclusion of exponential factors of time-integrations of $\|u_s\|_{\Hbf^r}$ for $r\in(0,3)$ which arise very naturally in our analysis.

\begin{lemma} \label{lem:TwistBd} Let $(u_t)$ solve either Systems \ref{sys:NSE} or \ref{sys:3DNSE}. There exists a $\gamma_* >0$, such that for all $0\leq \gamma < \gamma_*$, $T\geq 0$, $r\in (0,3)$, $\kappa \geq 0$, and $V(u) = V_{\beta,\eta}$ where $\beta \geq 0$ and $0 < e^{\gamma T}\eta < \eta_*$, there exists a constant $C = C(\gamma,T,r,\kappa,\beta,\eta) >0$ such that the following estimate holds
\begin{equation}\label{eq:Twistbd}
\EE_u \exp\left(\kappa \int_0^T\norm{u_s}_{\Hbf^r}\ds\right)\sup_{0\leq t\leq T}V^{e^{\gamma t}}(u_t) \leq C V(u).
\end{equation}

\end{lemma}
\begin{remark}
In fact, it suffices to take $\gamma_\ast = \frac{\nu}{8}$. 
\end{remark}
\begin{proof}
By It\^o's Lemma applied to the functional $\log{V(u)}$ (see \cite{KS} Theorem 7.7.5), we know that
\[
	M_t := e^{\gamma t}\log V(u_t) - \log V(u) - \int_0^t e^{\gamma s}(\gamma\log V(u_s)  + \mathcal{L}\log V(u_s))\ds
\]
is a mean-zero, time-continuous, local martingale with quadratic variation satisfying
\begin{equation}\label{eq:quad-var-bound}
	\langle M\rangle_t = \sum_{m\in \mathbb{K}}\int_0^t e^{2\gamma s} |q_k|^2 |D_u\log V(u_s)e_m|^2\ds
	\leq 8\int_0^t e^{2\gamma s}\left(\beta^2 + \eta^2 \mathcal{Q}\|\Delta u_s\|^2_{L^2} \right)\ds,
\end{equation}
where $\mathcal{Q}$ is defined in Lemma \ref{lem:Lyapu}. 
Recall the exponential martingale estimate
\begin{equation}\label{eq:martingale-bound}
	\E_u\exp\left( \sup_{t\geq 0} (M_t - \langle M\rangle_t)\right) \leq 2.
\end{equation}
In light of this, note that
\begin{equation}
\begin{aligned}
 &M_t - \langle M\rangle_t
 \geq e^{\gamma t}\log V(u_t) - \log V(u)\\
 &\hspace{.5in}- \int_0^te^{\gamma s}\left(\gamma \log V(u_s) + \mathcal{L}\log V(u_s) + 8e^{\gamma s}(\beta^2 + \eta^2 \mathcal{Q}\|\Delta u_s\|^2_{L^2})\right)\ds. 
\end{aligned}
\end{equation}
Applying Lemma \ref{lem:SobLogTrick2} to $\gamma \log V$ and Lemma \ref{lem:GenV} to $\mathcal{L}\log V$, collecting like terms, and bounding $\|u\|_{\Wbf}$ by $\|\Delta u\|_{L^2}$ we find that for each $\ep >0$
\begin{equation}\label{eq:gammalogvbound}
\begin{aligned}
&\gamma \log V(u_s) + \mathcal{L} \log V(u_s) + 8 e^{\gamma s}(\beta^2 + \eta^2\mathcal{Q} \|\Delta u_s\|^2_{L^2})\\
& \hspace{1in}\leq - \eta (\nu - \gamma - 8e^{\gamma s}\eta\mathcal{Q}) \|\Delta u_s\|_{L^2}^2  -(\nu \beta - \ep) \frac{\|\nabla u_s\|_{\Hbf^\sigma}^2}{1+ \|u_s\|_{\Hbf^\sigma}^2} + Ce^{\gamma s}. 
\end{aligned}
\end{equation}
Using that $e^{\gamma T}\eta\leq \eta_*$, we can choose $\gamma$ and $\ep$ small enough so that
\[
	\gamma + 8\eta_*\mathcal{Q} \leq \frac{1}{2}\nu\quad\text{and}\quad  \ep < \frac{1}{2}\nu \beta.
\]
Therefore we can apply the interpolation Lemma \ref{lem:SobTrick} to deduce that for each $\kappa >0$ the right-hand side of inequality \eqref{eq:gammalogvbound} is bounded by $-\kappa \|u_s\|_{\Hbf^r} + Ce^{\gamma s}$ and therefore
\[
	M_t - \langle M\rangle_t \geq e^{\gamma t}\log V(u_t) - \log V(u)+ \kappa \int_0^Te^{\gamma s}\|u_s\|_{\Hbf^r}\ds - C
\]
Applying \eqref{eq:martingale-bound} and restricting the supremum to $0\leq t \leq T$ gives the estimate.
\end{proof} 

\begin{remark} \label{rmk:SuperL}
Note that Lemma \ref{lem:TwistBd} is strictly stronger than a drift condition. The improvement in the power of $V$ is sometimes called a \emph{super-Lyapunov property} and it provides an important strengthening of the notion of a drift condition.  To see that \eqref{eq:Twistbd} implies a drift condition, we write $P_1 \varphi (u) = \E_u \varphi(u_1)$ as the Markov semi-group for Navier-Stokes, set $\kappa = 0$ in \eqref{eq:Twistbd}, and then Jensen's inequality implies that for some $C_L > 0$, there holds 
\begin{align}\label{eq:P_1-super-Lyapunov}
P_1 V \leq  (e^{C_L}V)^{e^{-\gamma}},
\end{align}
which immediately implies that for all $\delta > 0$, there exists $C_\delta > 0$ such that
\begin{align*}
P_1V \leq \delta V + C_{\delta}. 
\end{align*}
Furthermore, the bound \ref{eq:P_1-super-Lyapunov} can be iterated with repeated applications of Jensen's inequality (c.f. [Proposition 5.11, \cite{HM11}]) to produce
\begin{align}
P_nV \leq e^{C_L\frac{e^{-\gamma}}{1-e^{-\gamma n}}} V^{e^{-\gamma n}}.  \label{ineq:SuperLyapIter}
\end{align}
The inequality \eqref{ineq:SuperLyapIter} gives a strong quantification of the tendency to return back to uniform sub-level sets of $V$ irrespective of the initial distribution as $n \to \infty$ (i.e. the initial data is forgotten exponentially fast).  
\end{remark}

\subsection{Estimates on the Jacobian of the projective process} \label{sec:Jacobian}
In this section we provide the necessary estimates on the Jacobian of the projective process. Recall the projective process $(\hat{z}_t) = (u_t,x_t,v_t)$ solves the abstract SDE in $\Hbf\times P\T^d$
\[
	\partial_t \hat{z}_t = F(\hat{z}_t) + Q\dot{W}_t,
\]
where we view $Q\dot{W}$ as extended to $\Hbf\times T_{v_t}P\T^d$ and for each $\hat{z} = (u,x,v)\in \Hbf\times P\T^d$
\[
	F(\hat{z}) = \begin{pmatrix}-B(u,u) - Au\\ u(x)\\ (I - v\tensor v) (Du(x)v)\end{pmatrix}.
\] 
The Jacobian process $J_{s,t}$ denotes the Fr\'echet derivative of the solution $\hat z_t$ with respect to the value at time $s<t$. Hence, $J_{s,t}$ solves the operator-valued equation
\begin{equation}\label{eq:Jacdef}
	\partial_t J_{s,t} = DF(\hat{z}_t)J_{s,t}, \quad J_{s,s} = \Id.
\end{equation}
We will prove that this is a bounded operator $J_{s,t} : \Wbf \times T_{v_s} P \mathbb T^d  \to \Wbf \times T_{v_t} P \mathbb T^d$ (this is not obvious due to the evolution on $P\mathbb T^d$ requiring pointwise evaluations $u$ and $Du$). Additionally we let $K_{s,t}:\Wbf\times T_{v_t}P\T^d \to \Wbf\times T_{v_s}P\T^d$ denote the adjoint of $J_{s,t}$, in the sense that
\begin{align}
\brak{f, J_{s,t} \xi }_{\Wbf} = \brak{K_{s,t} f,  \xi }_{\Wbf}.
\end{align}
A straightforward calculation (see \cite{HM11}) shows that $K_{s,t}$ solves the following backward-in-time equation
\begin{align}
\partial_s K_{s,t}  = -DF(\hat{z}_s)^\ast K_{s,t}, \quad K_{t,t} = I.
\end{align}
where $DF(\hat{z}_s)^*: \Wbf \times T_{v_s}P\T^d \to \Wbf \times T_{v_s}P\T^d$ is the adjoint to $DF(\hat{z}_s)$.

In what follows, we will find it convenient to let $\tilde{z} = (\tilde{u},\tilde{x},\tilde{v}) \in \Wbf \times T_{v_s} P \mathbb T^d$ an initial perturbation and denote 
\[
\tilde{z}_t := (\tilde{u}_t, \tilde{x}_t, \tilde{v}_t) = J_{s,t}\tilde{w} = (J_{s,t}^u \tilde{w}, J_{s,t}^x \tilde{w}, J_{s,t}^v\tilde{w}) \in \Wbf \times T_{v_t}P\T^d,
\]
which readily solves the linear evolution equation
\[
	\partial_t \tilde{z}_t = DF(\hat{z}_t)\tilde{z}_t, \quad \tilde{z}_s = \tilde{z}.
\]

\begin{lemma} \label{lem:PathJacobian} 
$\forall \sigma > \frac{d}{2}+1$, $\forall r \in (\frac{d}{2}+1,3)$,  $\exists C, q' > 0$ such that the following holds path-wise
\begin{subequations} 
\begin{align}
\norm{\tilde{u}_t}_{\Wbf} & \leq \norm{\tilde{u}}_{\Wbf} \exp\left(C \int_s^t \norm{u_\tau}_{\Hbf^{r}} \dee\tau \right) \label{def:JHbd} \\ 
\norm{J_{s,t}}_{\Hbf^\sigma \to \Hbf^\sigma} & \lesssim \exp\left(C \int_s^t \norm{u_\tau}_{\Hbf^{r}} \dee\tau \right) \left(1 + \brak{t-s}^3 \sup_{s < \tau < t}\norm{u_\tau}_{\Hbf^\sigma}^{q'} \right). \label{def:JHSigbd}
\end{align}
\end{subequations}

\end{lemma} 
\begin{proof}
We consider the case $d=2$; $d=3$ follows analogously. We start by estimating $\tilde{u}_t$, which solves, 
\[
	\partial_t \tilde u_t + B(u_t, \tilde{u}_t) + B(\tilde{u}_t,u_t) = A \tilde{u}_t.  
\]
Integration by parts, vector calculus, and the divergence free condition gives
\begin{align}
\frac{\dee}{\dt}\norm{\tilde{u}_t}_{\Wbf}^2 + \nu \norm{\grad \tilde u_t}_{\Wbf}^2 & = \brak{\curl \tilde{u}_t, \grad \cdot \left(u_t \curl \tilde{u}_t  + \tilde{u}_t \curl u_t \right)}_{L^2}\\  & \lesssim \norm{\grad u_t}_{L^\infty} \norm{\tilde u_t}_{\Wbf} \norm{\grad \tilde u_t}_{\Wbf}, 
\end{align}
which implies for some $C> 0$, (applying Sobolev embedding), 
\begin{align}
\norm{\tilde{u}_t}_{\Wbf}  \leq \norm{\tilde{u}}_{\Wbf} \exp \left(C \int_s^t \norm{u_\tau}_{\Hbf^{r}} \dee\tau \right). 
\end{align}
This completes the proof of \eqref{def:JHbd} (in the case $d=2$). 
Next, we turn to the $\Hbf^\sigma$ estimate on $\tilde{u}_t$ for the proof of \eqref{def:JHSigbd}.
By the divergence free constraint,  
\begin{align}
\frac{\dee}{\dt}\norm{\tilde{u}_t}_{\Hbf^\sigma}^2 + \nu \norm{\grad \tilde u_t}_{\Hbf^\sigma}^2 & = \underbrace{\brak{\tilde{u}_t, \grad \cdot (u_t \otimes  \tilde{u}_t)}_{\Hbf^\sigma}}_{T1} + \underbrace{\brak{\tilde{u}_t, \grad \cdot( \tilde{u}_t \otimes u_t) }_{\Hbf^\sigma}}_{T2}. \label{eq:HsigEst} 
\end{align}
Using the divergence free property to introduce a commutator (recall the Fourier multiplier from Section \ref{sec:Note}), and the triangle inequality $\brak{k}^\sigma \lesssim_\sigma \brak{k-\ell}^{\sigma} + \brak{\ell}^\sigma$, there holds 
\begin{align}
T1 & = \brak{\brak{\grad}^\sigma \tilde{u}_t , \brak{\grad}^\sigma \grad \cdot (u_t \otimes  \tilde{u}_t) - \grad \cdot (u_t \otimes \brak{\grad}^\sigma \tilde{u}_t) }_{L^2} \nonumber \\ 
& \leq \left( \sum_{\abs{k-\ell} > \abs{\ell}} + \sum_{\abs{k-\ell} \leq \abs{\ell}}\right)\brak{k}^\sigma \abs{\hat{\tilde{u}_t}(k)} \abs{\brak{k}^\sigma - \brak{\ell}^\sigma } \abs{k} \abs{\hat{u}_t (k-\ell) \hat{\tilde{u}_t}(\ell)} \nonumber \\ 
& \lesssim \norm{\grad \tilde{u}_t}_{\Hbf^\sigma} \norm{u_t}_{\Hbf^\sigma} \norm{\tilde{u}_t}_{\Hbf^{r-1}} + \norm{u_t}_{\Hbf^{r}} \norm{\tilde{u}_t}_{\Hbf^\sigma}^2, \label{ineq:T1}
\end{align}
where in the penultimate line we used that in the second summation there holds $\abs{\brak{k}^\sigma - \brak{\ell}^\sigma} \lesssim \brak{\ell}^{\sigma-1}\abs{k-\ell}$ and in the last line we used Cauchy-Schwarz and Young's convolution inequality. 
By a similar argument (but no commutator necessary) there holds 
\begin{align}
T2 & \lesssim \norm{\grad \tilde{u}_t}_{\Hbf^\sigma} \norm{u_t}_{\Hbf^\sigma} \norm{\tilde{u}_t}_{\Hbf^{r-1}} + \norm{u_t}_{\Hbf^{r}} \norm{\tilde{u}_t}_{\Hbf^\sigma}^2. \label{ineq:T2}
\end{align}
Note by interpolation, for some $\theta = \theta(\sigma) \in (0,1)$, there holds 
\begin{align*}
\norm{\tilde{u}_t}_{\Hbf^{r-1}} \leq \norm{\tilde{u}_t}_{\Wbf}^{1-\theta} \norm{\grad \tilde u_t}_{\Hbf^\sigma}^\theta, 
\end{align*}
and therefore putting \eqref{ineq:T1} and \eqref{ineq:T2} together with \eqref{eq:HsigEst}  gives (for suitable $C > 0$), 
\begin{align*}
\frac{\dee}{\dt}\norm{\tilde{u}_t}_{\Hbf^\sigma}^2 
& \leq - \frac{\nu}{2}\norm{\grad \tilde u_t}_{\Hbf^\sigma}^2 + C\norm{u_t}_{\Hbf^{r}} \norm{\tilde{u}_t}_{\Hbf^\sigma}^2 + C \norm{\tilde{u}_t}_{\Wbf}^{2} \norm{u_t}_{\Hbf^\sigma}^{\frac{2}{1-\theta}}. 
\end{align*}
Integrating and \eqref{def:JHbd} completes the estimate on the velocity field necessary for \eqref{def:JHSigbd}.   

Next, consider the remaining contributions of the projective process necessary for \eqref{def:JHSigbd}. Using that
\[
	\partial_t \tilde{x}_t  = \grad u_t(x_t) \tilde{x}_t + \tilde{u}_t(x_t)
\]
and applying Gr\"onwall's inequality followed by the estimate on the $\Hbf^\sigma$ norm of the velocity field gives 
\begin{align}
\abs{\tilde{x}_t} & \leq \exp\left( C \int_s^t \norm{\grad u_\tau}_{L^\infty} \dee\tau\right) \abs{\tilde{x}} + \int_s^t \exp\left( C\int_\tau^t \norm{\grad u_{\tau'}}_{L^\infty} \dee\tau' \right) \norm{u_\tau}_{L^\infty} \dee\tau \\
& \lesssim \exp\left(C \int_s^t \norm{u_\tau}_{\Hbf^{r}} \dee\tau  \right) \left( 1 + \brak{t-s}\sup_{s < \tau < t}  \norm{u_\tau}_{\Hbf^\sigma}^{q'}\right) \left(\abs{\tilde x} + \norm{\tilde{u}}_{\Hbf^{\sigma}}\right),  
\end{align}
which suffices for estimates $\tilde{x}_t$. Similarly, for the projective process we have
\begin{equation}
\begin{aligned}
	&\partial_t \tilde v_t  = - \tilde v_t \otimes v_t \grad u_t(x_t) v_t  - v_t \otimes \tilde v_t \grad u_t(x_t) v_t + \left(1 - v_t \otimes v_t\right)\grad u_t(x_t) \tilde v_t \\ 
& \quad + \left(1 - v_t \otimes v_t\right) \tilde{x}_t \cdot \grad^2 u_t(x_t) v_t + \left(1 - v_t \otimes v_t\right)\grad \tilde{u}_t(x_t) v_t.
\end{aligned}
\end{equation}
Hence, 
\begin{align}\label{eq:tildev-ineq}
\frac{\dee}{\dt}\abs{\tilde v_t} \lesssim \norm{\grad u}_{L^\infty} \abs{\tilde{v}_t}+ \abs{\tilde{x}_t} \norm{\grad^2 u}_{L^\infty} + \norm{\grad \tilde{u}}_{L^\infty},
\end{align}
and the required estimate on $\tilde{v}_t$ follows by Gr\"onwall's inequality as in the case of $\tilde{x}_t$. 
\end{proof}

The path-wise estimates in Lemma \ref{lem:PathJacobian} together with Lemma \ref{lem:TwistBd} implies the following: 
\begin{lemma}[Jacobian bounds in expectation] \label{lem:JacExps}
For all $\sigma$ and all $\eta > 0$, there is a constant $C_J$ such that the following holds for all $1 \leq p < \infty$, 
\begin{align}
\sup_{s \leq t \leq 1}\EE\norm{J_{s,t}}^p_{\Hbf^\sigma \to \Hbf^\sigma} & \leq V_{q', \eta}^p (u_s) \exp \left( p C_J \right). 
\end{align}
\end{lemma}

Next, we  deduce a parabolic smoothing estimate on the Jacobian.
One small subtle point: pointwise evaluations of $\tilde u$ and $D \tilde{u}$ are appear in the ODEs for $\tilde{x}_t$ and $\tilde{v}_t$ (respectively) and hence a little care must be taken to control low regularity data for $\tilde{u}_t$ using the local-in-time parabolic smoothing. 
\begin{lemma} \label{lem:JacSmooth}
Let $\gamma \in [0,\alpha - \frac{d}{2})$ and $r \in (\frac{d}{2}+1,3)$. Then, $\exists \varkappa'$the following holds path-wise for $0 \leq s \leq t \leq 1$: 
\begin{align}
(t-s)^{\frac{\gamma}{2(d-1)}}\norm{J_{s,t}}_{\Wbf \to \Hbf^\gamma} \lesssim  \exp\left(C \int_s^t \norm{u_\tau}_{\Hbf^{r}} \dee\tau \right) \left(1 +  \sup_{\tau \in (s,t)} \norm{u_\tau}_{\Hbf^\sigma}^{\varkappa'}\right).    
\end{align}
\end{lemma}
\begin{proof}
We consider only the case $d=2$; the case $d=3$ is a straightforward variation.  
First, the desired estimate on $\tilde{u}_t$ follows from standard semilinear PDE methods (see for instance \cite{HM11}) and is omitted for the sake of brevity.

Turn next to $\tilde{x}_t$ and $\tilde{v}_t$. 
Estimating the Lagrangian process as in Lemma \ref{lem:PathJacobian} gives 
\begin{align*}
\abs{\tilde{x}_t} \lesssim \exp\left(\int_0^t \norm{\grad u_\tau}_{L^\infty} \dee\tau \right)  \abs{\tilde{x}} + \int_0^t \exp\left(\int_\tau^t  \norm{\grad u_s}_{L^\infty} \ds \right) \norm{\tilde{u}_\tau}_{L^\infty} d\tau. 
\end{align*}
By Sobolev embedding and the smoothing estimates deduced on $\tilde{u}_t$,  $\forall \delta > 0$,
\begin{align*}
\norm{\tilde{u}_\tau}_{L^\infty} \lesssim \norm{\tilde u_\tau}_{\Hbf^{1+\delta}} \lesssim \frac{1}{\tau^{\delta/2}}\exp\left(C\int_0^\tau \norm{u_s}_{\Hbf^{r}} \ds \right) \left(1 + \sup_{0 < s < \tau} \norm{u_s}_{\Hbf^\sigma}^{q'} \right) \norm{\tilde{u}}_{\Wbf},
\end{align*}
which yields the desired estimate. 

The situation for $\tilde{v}_t$ is similar but higher regularity is required. 
Observe that $\forall s,\delta > 0$, 
\begin{align}
\norm{\grad \tilde{u}_\tau}_{L^\infty} \lesssim_\delta \norm{\tilde{u}_\tau}_{\Hbf^{2+\delta}} \lesssim \tau^{-\frac{1 + \delta}{2}} \exp\left( \int_0^\tau \norm{u_s}_{\Hbf^{r}} \ds \right) \left(1 + \sup_{0 < s < \tau} \norm{u_s}_{\Hbf^\sigma}^{q'} \right) \norm{\tilde{u}}_{\Wbf}. 
\end{align}
Then, the desired estimate then follows similarly to the estimate on $|\tilde{x}_t|$.  
\end{proof}

\begin{remark}\label{rmk:high-freq-smooth}
The above smoothing estimate implies the following: For some constant $C_{\Pi} \approx \log N$ there holds the following  $\forall p$ with $p\eta < \eta^\ast$ and a $\beta$ sufficiently large,  
	\begin{align}
	\EE\|J_{0,1}(\Id-\Pi_{N})\|^p_{\Wbf \to \Hbf^{\sigma}} \leq V_{\beta ,\eta}^p (u) \exp(-p C_{\Pi}). 
	\end{align}
\end{remark}

\begin{remark} \label{rmk:JacSmthH}
Lemma \ref{lem:JacSmooth} is already non-trivial for $\gamma = 0$, that is, $J_{s,t}$ is a bounded linear operator. 
\end{remark}

The analogue of Lemma \ref{lem:PathJacobian} holds also for $K_{s,t}$. The proof is actually easier, and is omitted for brevity.
\begin{lemma} \label{lem:PathAdjoint}
$\forall \sigma > \frac{d}{2}+1$, $\forall r \in (\frac{d}{2}+1,3)$, $\exists C, q' > 0$ such that the following hold path-wise 
\begin{align}
\norm{K_{s,t}}_{\Wbf \to \Wbf} & \lesssim \exp\left(C \int_s^t \norm{u_\tau}_{\Hbf^{r}} \dee\tau \right) \\ 
\norm{K_{s,t}}_{\Hbf \to \Hbf} & \lesssim \exp\left(C \int_s^t \norm{u_\tau}_{\Hbf^{r}} \dee\tau \right)\left(1 +  \brak{t-s}^3 \sup_{s < \tau < t}\norm{u_\tau}_{\Hbf^\sigma}^{q'}\right). 
\end{align}
\end{lemma}

\section{Lasota-Yorke bound and spectral gap for the projective process} \label{sec:C1VSpecProj}

\subsection{Geometric ergodicity in $C_V$} \label{sec:DefUniBd}
In this section, we apply Theorem \ref{thm:GM} to the projective process $(u_t,x_t,v_t)$.
In \cite{BBPS18} we already showed that there exists a unique stationary measure $\nu$ on $\Hbf \times P \T^d$.  
Strong Feller in a scale of Sobolev spaces is proved in \cite{BBPS18}. 
\begin{proposition}[From \cite{BBPS18}]\label{prop:SFscaleProj}
For any $\sigma' \in (\alpha - 2(d-1),\alpha - \frac{d}{2})$ the projective Markov process $(u_t, x_t, v_t)$ on $\Hbf^{\sigma'} \times P \mathbb T^d$ is strong Feller.
\end{proposition}

By an easy variation of the irreducibility arguments in Section \ref{sec:Irr2pt} together with those in \cite{BBPS18}, one deduces the following irreducibility property as well.
The straightforward details are omitted for brevity. 

\begin{proposition}[Essentially from \cite{BBPS18}]\label{prop:topIrredProj} 
For any $\sigma' \in (\frac{d}{2} + 2, \alpha - \frac{d}{2})$,   the projective process $(u_t, x_t, y_t)$, regarded as a process on $\Hbf^{\sigma'} \times P \mathbb T^d$, is topologically irreducible.
\end{proposition}

The previous two properties imply equivalence of transition kernels by standard arguments (Lemma 3.2 in \cite{goldys2005exponential}; see also Theorem 4.1 in \cite{FM95} for a special case). 
\begin{lemma}\label{lem:equivFamilyIntro}
 Propositions \ref{prop:SFscaleProj} and \ref{prop:topIrredProj} imply that the family of transition kernels $\{ \hat{P}_t((u, x, y), \cdot) : t > 0, (u, x, y) \in \Hbf \times \mathbb P \T^d \}$ are equivalent.
\end{lemma}

Equivalence of transition measures provides the main tool for proving Condition \ref{def:UnifBd}.

\begin{lemma}\label{lem:hypothA3}
 Let $V$ be any of the Lyapunov functions for the projective process defined in Lemma \ref{lem:Lyapu}.  
	Let $K \subset \Hbf \times P \mathbb T^d$ be any compact set with $\nu(K) > 0$. Fix arbitrary
	$t > 0$	and $r > 0$. Then, we have
	\begin{align}\label{eq:condA3}
	\inf_{(u, x, v) : V(u) \leq r} \hat{P}_t((u, x, v) , K) > 0 \, .
\end{align}
\end{lemma}

\begin{proof}
  First, all transition kernels $\hat{P}_t((u, x, v), \cdot)$ are equivalent, and so we conclude that they are all equivalent with the unique stationary measure $\nu$, which we regard as a measure on each $\Hbf^{\sigma'} \times P \mathbb T^d$.
In particular, with $K$ as in the hypothesis of Lemma \ref{lem:hypothA3}, we have 
$\hat{P}_t((u, x, v), K) > 0$ for all $(u, x, v) \in \Hbf^{\sigma'} \times P\T^d$, $\sigma' \in (\alpha - 2(d-1),\alpha - \frac{d}{2})$.

For the sake of contradiction, assume Condition \ref{def:UnifBd} fails. Then, there is a sequence $\{z^n = (u^n,x^n,v^n) : V(u^n) \leq r\}$
for which $\hat{P}_t(z^n, K) \to 0$ as $n \to \infty$. Fix $\sigma' \in (\alpha - 2(d-1), \sigma)$. It follows from compact embedding and coercivity that $\{ z^n\}$ admits a subsequence $\{ z^{n'}\}$ converging
in $\Hbf^{\sigma'} \times P \mathbb T^d$. For this sequence, we have
\[
\lim_{n' \to \infty} \hat{P}_t(z^{n'}, K) = 0
\]
by the strong Feller property on $\Hbf^{\sigma'} \times P \mathbb T^d$ (Proposition \ref{prop:SFscaleProj}). 
This contradicts Lemma \ref{lem:equivFamilyIntro} at $\sigma'$.
\end{proof}

Putting everything together with Theorem \ref{thm:GM} implies the geometric ergodicity of the projective process.
Boundedness in $C_V$ follows from Lemma \ref{lem:TwistBd}. 
Recall $\nu$ is the stationary measure for the $(u_t,x_t,v_t)$ process.  
\begin{proposition} \label{prop:CVspecGapProj}
For any $V$ satisfying the conditions of Lemma \ref{lem:Lyapu}, there exists a $\gamma > 0$ (depending on $V$) such that for any $\psi$ bounded measurable on $\Hbf \times P \mathbb T^d $, there holds, 
\begin{align}
\abs{\hat{P}_t \psi(z) - \int_{\Hbf \times P \mathbb T^d}  \psi \,\dee\nu} \lesssim V(u) e^{-\gamma t}.
\end{align}
\end{proposition}

\subsection{Spectral gap in $C^1_V$ from a Lasota-Yorke estimate}
Recall in the outline that our approach to obtaining a spectral gap on $\hat{P}_t$ in $\mathring{C}^1_V$ is the following Lasota-Yorke gradient estimate. The version stated below was introduced by Hairer and Mattingly in \cite{HM06} as a sufficient condition for the asymptotic strong Feller property as well as in \cite{HM08} to prove spectral gaps in $C^1_V$.
\begin{proposition}[Lasota-Yorke estimate] \label{prop:LY}
$\forall \beta' \geq 2$ sufficiently large and $\forall \eta' \in (0, \eta^\ast)$, $\exists C_1, \varkappa >0$ such that the following holds $\forall t > 0$, and $\hat{z}=(u,x,v)\in \Hbf\times P\T^d$
\begin{align}
\|D \hat{P}_t \psi(\hat{z})\|_{\Hbf^*} \leq C_1 V_{\beta',\eta'}(u) \left( \sqrt{\hat{P}_t\abs{\psi}^2(\hat{z})} + e^{-\varkappa t} \sqrt{\hat{P}_t \|D\psi\|_{\Hbf^*}^2(\hat{z})}\right).
\end{align}
\end{proposition}
Proposition \ref{prop:LY} when combined with Proposition \ref{prop:CVspecGapProj} is sufficient to deduce a spectral gap for $\hat{P}_{T0}^p$ on $\mathring{C}^1_V$ for $T_0$ sufficiently large.
\begin{proposition} \label{prop:C1VSpec}
For all $V = V_{\beta,\eta}$ with $\beta$ sufficiently large and $\eta \in (0,\eta^\ast)$, we have that $\hat{P}^p_{T_0}$ has a spectral gap on $C^1_V$ for all $T_0$ sufficiently large.
\end{proposition}
\begin{proof}  
For some $\delta> 0$ to be chosen below, define the equivalent norm on $C^1_V$
\begin{align}
\norm{\psi}_{C^1_{V,\delta}} := \sup_{w \in \Hbf \times P \mathbb T^d} \left(\frac{|\psi(\hat{z})|}{V(u)} + \frac{\delta \|D\psi(\hat{z})\|_{\Hbf^{*}}}{V(u)}\right). 
\end{align}
Choose $\beta' < \beta$ and $\eta' < \eta$. Then, the Lasota-Yorke bound (Proposition \ref{prop:LY}) together with the super-Lyapunov property \eqref{ineq:SuperLyapIter} implies
\begin{align}
\|D \hat{P}_n \psi(\hat{z})\|_{\Hbf^*} & \leq C_1 V_{\beta^\prime,\eta^\prime}\sqrt{ \hat{P}_n V^2(u)}\left(1  + e^{-\varkappa n}\delta^{-1}\right)\|\psi\|_{C^1_{V,\delta}}\\
&\leq \frac{C}{\delta} V_{\beta',\eta'}(u) V^{e^{-\gamma n}}(u)\left(\delta  + e^{-\varkappa n}\right) \norm{\psi}_{C^1_{V,\delta}}.  
\end{align}
Since $\beta > \beta'$ and $\eta > \eta^\prime$, there exists a $n_0$ sufficiently large so that
\begin{equation}
\beta' + e^{-\gamma n} \beta < \beta \quad\text{and}\quad \eta' + e^{-\gamma n} \eta  < \eta, 
\end{equation}
which implies $V_{\beta',\eta'}(u) V^{e^{-\gamma n_0}}(u) \leq V(u)$. 
Choosing $\delta>0$ such that $C\delta < 1/8$ and $n_1 > n_0$ such that  $Ce^{-\varkappa n_1} < 1/8$, we have
\begin{align}
\|D\hat{P}_{n_1} \psi(\hat{z})\|_{\Hbf^*} \leq \frac{1}{4} \frac{V(u)}{\delta} \norm{\psi}_{C^1_{V,\delta}}. 
\end{align}
Combining this with the $C_V$ spectral gap (Proposition \ref{prop:CVspecGapProj}) implies that $\hat{P}^p_{T_0}$ is a contraction on $C^1_V$ with respect to the $\|\cdot\|_{C^1_{V,\delta}}$ norm for all $T_0 > 0$ sufficiently large. This implies a spectral gap of $\hat{P}^p_{T_0}$ in $C^1_V$.
\end{proof}

\subsection{Lasota-Yorke estimate: Proof of Proposition \ref{prop:LY}}
Let us now turn to the proof of Proposition \ref{prop:LY}. The proof follows closely the discussion in \cite{HM11} with some adjustments, and so we only provide a brief sketch. Aside from the fact that the necessary estimates on the $J_{s,t}$ and $K_{s,t}$ processes are more complex than in \cite{HM11}, the Malliavin matrix non-degeneracy requires some adjustments to deal with the degrees of freedom associated with $P \T^d$. 
Additionally, we need to use Lemma \ref{lem:JacSmooth} and choose $\beta'$ sufficiently large to control
$\Hbf$ regularity. Other than this, only minor modifications are needed. 

\subsubsection{Malliavin calculus and preliminaries}
First, let us recall some basics of Malliavin calculus required to set up the framework of \cite{HM11}. 
We will mostly be dealing with random variables $X  \in \Wbf\times P \mathbb T^d$.  
The Malliavin derivative $\MalD_h X$ of $X$ in direction $h = (h_t) \in L^2(\R_+,\Wbf)$ is defined by
\[
	\MalD_h X := \frac{\dee}{\dee \eps}X(W + \eps H)|_{\eps=0}, \quad H= \int_0^\cdot h_s\ds,
\]
when the limit exists (in the Fr\'{e}chet sense). If the above limit exists, we say that $X$ is {\em Malliavin differentiable}. In practice $\MalD_h X$ admits a representation of the form
\begin{equation}\label{eq:D_sX-def}
\MalD_h X = \int_0^{\infty} \MalD_s X h_s\, \ds,
\end{equation}
where for a.e. $s\in \R_+$, $\MalD_h X$ is a random, bounded linear operator from $\Wbf$ to $\Wbf \times T_v P \T^d$ (see \cite{Nualart06} for more details).
It is standard that if $X_t$ is adapted to the filtration $\mathscr{F}_t$ generated by $W_t$, then $\MalD_s X_t= 0$ if $s \geq t$. 
The Malliavian derivative $\mathcal{D}_h w_t$ is given by (recall $J_{s,t}$ is the Jacobian; see \eqref{eq:Jacdef}), 
\begin{align}
\mathcal{D}_h w_t = \int_0^t J_{s,t} Q h_s \ds := \mathcal{A}_t h.
\end{align}
The adjoint is given by $A^\ast_t \xi(s) = Q^\ast K_{s,t} \xi$ for $s \leq t$ and $0$ for $s > t$ (recall{} $K_{s,t} = J^\ast_{s,t}$).
We similarly denote $\mathcal{A}_{s,t} h = \int_s^t J_{\tau,t} Q h_\tau d\tau$.
The \emph{Malliavin matrix} $\mathcal{M}_{s,t}$ is a \emph{symmetric, positive semi-definite} as a linear operator $\Wbf\times T_{v_t} P \mathbb T^d \to \Wbf \times T_{v_t} P \mathbb T^d$ defined by
\begin{align}
\mathcal{M}_{s,t} := \mathcal{A}_{s,t} \mathcal{A}^\ast_{s,t} = \int_s^t J_{r,t} Q Q^\ast K_{r,t} \dr.
\end{align}

For real-valued random variables, the Malliavin derivative can be realized as a Fr\'{e}chet differential operator $\MalD: L^2(\Omega)\to L^2(\Omega; L^2(\R_+;\Wbf))$. The adjoint operator $\MalD^*:L^2(\Omega; L^2(\R_+;\Wbf)) \to L^2(\Omega)$ is referred to as the {\em Skorohod integral}, whose action on $h \in L^2(\Omega; L^2(\R_+;\Wbf))$ we denote by
\[
	\int_0^\infty \langle h_t,\delta W_t\rangle_{\Wbf} := \MalD^*h.
\]
The Skorohod integral is an extension of the usual It\^{o} integral; see \cite{Nualart06,HM11}. 
One moreover has the following, 
\[
	\E \left(\int_0^\infty \langle h_t, \delta W_t\rangle_{\Wbf}\right)^2\leq \EE \int_0^\infty \norm{h_{t}}_{\Wbf}^2 + \E \int_0^\infty \int_0^\infty \norm{\MalD_sh_{t}}_{\Wbf\to \Wbf}^2\ds\dt. \label{ineq:SkorIto}
\]

A fundamental result in the theory of Malliavin calculus is the Malliavin integration by parts formula, stated below for the process $(\hat z_t)$ which takes values in $\Hbf \times P\T^d$ (see e.g. \cite{DaPrato14,Nualart06}). 
\begin{proposition}\label{lem:MalIBP}
Let $\psi$ be a bounded differentiable function on $\Hbf\times P\T^d$ with bounded derivatives
and $h_t$ be a process satisfying 
\begin{equation}\label{eq:control-bound}
\EE \int_0^T \norm{h_t}_{\Wbf}^2 \dt + \E \int_0^T \int_0^T \norm{\MalD_sh_t}_{\Wbf\to \Wbf}^2\ds\dt < \infty, 
\end{equation}
then the following relation holds
\[
	\E \MalD_h\psi(\hat{z}_T) = \E\left(\psi(\hat{z}_T)\int_0^T \langle h_s , \delta W_s\rangle_{\Wbf} \right).
\]
\end{proposition}

\subsubsection{Defining the control}
By Proposition \ref{lem:MalIBP}, for any choice of $g$ satisfying the hypotheses, there holds, 
\begin{align}
D\hat{P}_t \psi(\hat z)\xi = \EE D\psi(\hat z_t)J_t\xi = \EE D\psi(\hat z_t) \rho_t +  \EE\psi(\hat z_t) \int_0^t \brak{h_s,\delta W_s}_{\Wbf} \label{ineq:DerivSwap}
\end{align}
where $\rho_t\in \Hbf \times T_{v_t}P\T^d$ is the residual of the control, 
\begin{align}
\rho_t = J_t\xi - \mathcal{D}_{h} \hat{z}_t. 
\end{align}
From \eqref{ineq:DerivSwap}, Cauchy-Schwartz implies 
\begin{align}
\abs{D\hat{P}_t \psi(\hat z)\xi} \leq \sqrt{ \EE\abs{\int_0^t \brak{h_s,\delta W_s}_{\mathbf{W}}}^2 } \sqrt{\hat{P}_t |\psi|^2(z)} + \sqrt{ \EE\|\rho_t\|_{\Hbf}^2 }  \sqrt{\hat{P}_t \norm{D\psi}^2_{\Hbf^*}(z)}. \label{ineq:DxiPtw}
\end{align}
The goal in \cite{HM06,HM11} is to choose $h$ such that $\EE\|\rho_t\|_{\Hbf}^2$ decays exponentially fast while, at the same time, having a suitably controlled Skorohod integral.
The same abstract formula for the control used in \cite{HM11} will also work here (though obviously the control itself is different). 
For a random parameter $\beta_{k}>0$ chosen below, we iteratively define the control as follows:  
\begin{align} \label{def:ctrl}
h = h_s^\xi = \begin{cases} 
Q^\ast K_{s,2n+1} (\beta_{2n} + \mathcal{M}_{2n,2n+1})^{-1} J_{2n,2n+1} \rho_{2n} & \quad s \in [2n,2n+1) \\ 
0 & \quad s \in [2n+1,2n+2).
\end{cases}
\end{align}
With the control defined as above, one obtains the following recurrence relation for the $\rho_{2n}$ at even times,
\[
	\rho_{2n+2} = J_{2n+1,2n+2}\beta_{2n}(\beta_{2n} - \mathcal{M}_{2n,2n+1})^{-1}J_{2n,2n+1}\rho_{2n}
\]
By \eqref{ineq:DxiPtw}, Proposition \ref{prop:LY} follows from the following two lemmas. 
First, a decay estimate: 
\begin{lemma} \label{lem:rhodecay}
For all $a \geq 2$ sufficiently large,  $\forall \eta \in (0,\eta^\ast)$, and $\forall p \geq 2$ with $p\eta < \eta^\ast$, $\exists \kappa > 0$ such that the following holds, 
\begin{align}
\EE\norm{\rho_{2n}}_{\Hbf}^p \lesssim_{p,\eta,a} V^{p}_{a,\eta}(u) \exp(-p\kappa n) \norm{\xi}^p_{\Hbf}.
\end{align}
\end{lemma}
Second, a uniform estimate on the Skorohod integral of the control: 
\begin{lemma} \label{lem:skor}
Let $h$ be defined as in \eqref{def:ctrl}.
For all $a \geq 1$ sufficiently large, $\forall \eta \in (0,\eta^\ast)$, there holds 
\begin{align}
\EE\abs{\int_0^{2n} \brak{h,\delta W_s}_{\mathbf W}}^2 \lesssim  V_{a,\eta}(u) \norm{\xi}_{\Hbf \times T_v P \mathbb T^d}^2. 
\end{align}
\end{lemma}
It remains to sketch the proofs of Lemmas \ref{lem:rhodecay} and \ref{lem:skor}. Let $\Pi : \mathbf W \to \mathbf W$ be the projection to frequencies less than some high frequency $N_0$ chosen later and fixed through out the section.
In what follows, it is convenient to define for $k \in \mathbb K$: 
\begin{align}
g_k := (q_k e_k,0,0). 
\end{align}
In this notation, the Malliavin matrix takes the form for each $\xi\in \Wbf\times T_{v_t}P\T^d$
\begin{align}
\brak{\mathcal{M}_{s,t} \xi, \xi}_{\Wbf} = \sum_{k \in \mathbb K} \int_s^t \brak{g_k,K_{r,t}\xi}^2_{\Wbf} \dr.  \label{eq:MalNonDeg}
\end{align}
The precise statement of the Malliavin matrix non-degeneracy is as follows; we carry out the proof in Section \ref{sec:Malliavin}. 
\begin{proposition}[Malliavin matrix non-degeneracy]\label{prop:malnodeg} 
For all $a\geq 1$ sufficiently large, and all $\eta \in (0,\eta^\ast)$, $p \geq 2$ with $\eta p \leq \eta^\ast$,  $\forall \delta >0$ there exists a $C^\ast = C^*(p,\delta,a,\eta)$ such that
\[
\P\left(\inf_{\xi \in S_\delta} \langle \xi,\mathcal{M}_{0,1}\xi\rangle_{\Wbf} < \ep\right) \leq C^\ast V^{p}_{a,\eta}(u)\ep^p, 
\] 
where 
\begin{align}
	S_\delta = \left\{ \xi \in \mathbf{W} \times T_{v_1} P \mathbb T^d \,:\, \|\xi\|_{\Wbf} = 1 ,\quad  \|\Pi \xi\|_{\Wbf} > \delta \right\}. \label{def:Sbeta}
\end{align}
\end{proposition}
Once Proposition \ref{prop:malnodeg} is obtained, we  will set $\beta_k$ as in \cite{HM11} for suitably chosen parameters
\begin{align}
\beta_k = \frac{\delta^3}{V_{a,\eta}(u_k) (C^\ast)^{1/p}}.
\end{align}
From [Corollary 5.15, \cite{HM11}], Proposition \ref{prop:malnodeg} implies 
\begin{align}
\EE \norm{\Pi \beta_{2n} (\beta_{2n} + \mathcal{M}_{2n,2n+1})^{-1}}_{\Wbf \to \Wbf}^p \leq 2 \delta^p. \label{ineq:PiR}
\end{align}	
Once one has suitable estimates on the $J_{s,t}$ and $K_{s,t}$ (and by extension, $\mathcal{A}, \mathcal{A}^\ast, \mathcal{M}$) as well as the Malliavin derivatives of $J_{s,t}$ and $K_{s,t}$ (in order to estimate the Skorohod integral in Lemma \ref{lem:skor} via \eqref{ineq:SkorIto}) then Lemmas \ref{lem:rhodecay} and \ref{lem:skor} follow essentially as in \cite{HM11} (taking full advantage of Lemma \ref{lem:JacSmooth} and Remark \ref{rmk:high-freq-smooth} to switch to and from the $\Wbf$ and $\Hbf$ spaces). 
Most of the required estimates are immediate from Section \ref{sec:ExpProj}; we provide a brief sketch of the relevant Malliavin derivative estimates next.

\subsubsection{Malliavin derivatives of the Jacobian} \label{sec:MallPrelim}

Estimating the Malliavin derivatives of $J_{s,t}$ and similar quantities requires estimates on the second variation, which we denote $J_{s,t}^{(2)}$. 
This operator is given by (see \cite{HM11} or \cite{BBPS18}): 
\begin{align}
\partial_t J^{(2)}_{s,t}(\xi,\zeta) & = D^2F(\hat z_t)\left(J_{s,t}\xi, J_{s,t}\zeta \right) + DF(\hat z_t)J^{(2)}_{s,t}(\xi,\zeta) \quad\quad s < t. 
\end{align}
Observe that $D^2F$ is of the form: 
\begin{align}
D^2F(\hat{z}_t)(\xi,\zeta) =
\begin{pmatrix}
B(\xi,\zeta) + B(\zeta,\xi) \\
D \xi^u(x_t) \zeta^x + D \zeta^u(x_t) \xi^x + \xi^x \cdot D^2 u(x_t) \zeta^x \\
\mathfrak{V}_t(\xi,\eta),
\end{pmatrix} 
\end{align}
where $\mathfrak{V}$ is of the following general form: there exists a set of bounded operators $\mathcal{T}_j$ all multi-linear in all but the last components, 
\begin{align}
\mathfrak{V}_t & = \mathcal{T}_1(\xi^v,\zeta^v,\grad u_t(x_t),v_t) + \mathcal{T}_2(\xi^v, \zeta^{x}, \grad^2 u_t(x_t),v_t) + \mathcal{T}_3(\xi^v,\grad \zeta^u(x_t),v_t) \nonumber \\ & \quad + \mathcal{T}_4(\xi^x, \zeta^x,\grad^3 u_t(x_t),v_t) + \mathcal{T}_5(\xi^x,\grad^2 \zeta^u(x_t),v_t).  \label{def:J2v} 
\end{align}
The high numbers of derivatives appearing in $\mathfrak{V}$ makes obtaining estimates on $J^{(2)}_{s,t}$ non-obvious, indeed, it is bounded on high enough Sobolev spaces, but it is not clear that it defines a bounded operator on $(\Wbf\times T_{v_{s}} P\mathbb T^d)^{\otimes 2} \to \Wbf \times T_{v_t} P \mathbb T^d$. 
Nevertheless, the proof of Lemma \ref{lem:skor} uses bounds on Malliavin derivatives of $J_{s,t}$ in such low regularities (see \cite{HM11}). 
This requires taking advantage of the specific form of the Malliavin derivatives of $J_{s,t}$.
\begin{lemma}[Malliavin derivative bounds] \label{lem:MalliavinJA}
$\forall a \geq 1$ sufficiently large and all $p \geq 1$, $\eta \in (0,\eta^\ast)$ such that $p \eta \in (0,\eta^\ast)$, there exists a constant $C_M$ such that the following holds (where $C_L$ and $\gamma$ are as in \eqref{ineq:SuperLyapIter} above), 
\begin{align}
\EE \sup_{s,r \in [n,n+1]}\norm{\MalD_s J_{r,n+1}}^p_{\mathbf{W} \to \Wbf} & \lesssim e^{\frac{pC_L}{1-e^{-\gamma}} + pC_M} V_{a,\eta}^{p e^{-\gamma n}}(u) \label{ineq:MalDJ} \\
\EE \sup_{s \in [n,n+1]} \norm{\MalD_s \mathcal{A}_{n,n+1}}^p_{\mathbf{W} \to \Wbf} & \lesssim e^{\frac{pC_L}{1-e^{-\gamma}} + pC_M} V_{a,\eta}^{p e^{-\gamma n}}(u). 
\end{align}
\end{lemma}
\begin{remark}
Estimates on $\MalD_s \mathcal{A}^\ast_{n,n+1}$ follow by duality. 
\end{remark}
\begin{proof}
  Consider the Malliavin derivative with respect to the $k$-th Brownian motion (denoted by $\MalD_s^k$), which is given by (see e.g. [Lemma 5.13, \cite{HM11}]), 
\begin{align}
\MalD_s^k J_{r,n+1} \xi =
\begin{cases} 
J^{(2)}_{s,n+1}(J_{r,s}\xi,g_k) \quad r \leq s \\
J^{(2)}_{r,n+1}(J_{s,r}g_k,\xi) \quad s \leq r. 
\end{cases}
\end{align}
Next, we observe that
\begin{align}
J_{s,t}^{(2)}(\xi,\zeta) =  \int_s^t J_{r',t} \left(D^2F(\hat z_{r'})(J_{s,r'}\xi,J_{s,r'}\zeta) \right) \dr'. 
\end{align}
Therefore, for $r \leq s$
\begin{align}
J_{s,n+1}^{(2)}(J_{r,s}\xi,g_k) =  \int_s^{n+1} J_{r',n+1} \left(D^2F(\hat z_{r'})(J_{r,r'} \xi,J_{s,r'} g_k) \right) \dr'. 
\end{align}
Denote for suitably large parameters $C$, $q$, 
\begin{align}
\Gamma(r') = \exp\left(C \int_s^{r'} \norm{u_\tau}_{H^{r_0}} \dee\tau \right) \sup_{s < \tau < r'}\left(1 + \norm{u_\tau}_{\Hbf}^q\right). 
\end{align}
Due to the particular form of $J^{(2)}$ and in particular of $\mathfrak{V}$ in \eqref{def:J2v}, that $g_k$ is smooth and supported only in the velocity variables implies a significant simplification.
For example, $\abs{(J_{s,r'}g_k)^{x,v}} \lesssim \abs{k}^{-\alpha + \frac{d+3}{2}}\abs{s-r'} \Gamma(r')$.
Hence, where this appears, we may use the gain in time to balance loss of regularity through the smoothing estimate in Lemma \ref{lem:JacSmooth}.
For $r \leq s \leq r'$,  (using Remark \ref{rmk:JacSmthH} as well), 
\begin{align}
\norm{J_{r',n+1} \left(D^2F(\hat z_{r'})(J_{r,r'} \xi,J_{s,r'} g_k) \right)}_{\Wbf} & \lesssim \Gamma(r')\left(\frac{1}{ \left((n+1)-r'\right)^{\frac{1}{2(d-1)}}} +  \frac{(s-r')}{(r-r')^{\frac{1}{2(d-1)}}} \right) \norm{\xi}_{\Wbf} \\
  & \lesssim \Gamma(r')\left(\frac{1}{\left((n+1)-r'\right)^{\frac{1}{2(d-1)}}} +  1\right) \norm{\xi}_{\Wbf}. 
\end{align}
Applying Lemma \ref{lem:TwistBd} followed by Lemma \ref{ineq:SuperLyapIter} gives the desired result. 

The estimate for $s \leq r$ follows similarly, completing the required estimate \eqref{ineq:MalDJ}. 
The estimate on $\MalD_s \mathcal{A}_{n,n+1}$ then follows as in \cite{HM11}. 
\end{proof}

\subsubsection{Non-degeneracy of the Malliavin matrix} \label{sec:Malliavin}

In this section, we sketch the proof of Proposition \ref{prop:malnodeg}. The approach is similar to that of Hairer and Mattingly \cite{HM11}. However, the proof in \cite{HM11} does not exactly apply in our setting due to the fact that the nonlinearity in the $(x,v)$ dynamics is not polynomial. However, our situation is much nicer due to non-degenerate noise, and Lemma \ref{lem:lowerbound} tells us that we only need one Lie bracket to span the phase space (as opposed to infinitely many brackets in \cite{HM11}). As we saw in our previous paper \cite{BBPS18} this is enough to avoid any quantitative non-cancellation results like Norris' Lemma or estimates on Wiener polynomials. Since our proof is relatively simple and different enough from \cite{HM11} we include it below.

First, we record for the readers' convenience the following estimates proved in Section \ref{sec:Prelim}. 
\begin{lemma}\label{lem:dynamic-bounds}
For every $a \geq 2$, $\eta p \in (0,\eta^\ast)$, $p\geq 1$ we have
\[
	\E \sup_{0\leq t \leq 1} \|u_t\|^p_{\Hbf} \leqc V_{a,\eta}^p(u)
\]
\[
	\E \sup_{0< s <  1}\|K_{s,1}\|^p_{\Wbf \to \Wbf} \leqc V^p_{a,\eta}(u).
\]
\end{lemma}

Naturally, one of the main ingredients is a uniform spanning condition on Lie brackets to verify hypoellipticity.  
The following is a straightforward consequence of the spanning lemma in our previous work [Lemma 5.3, \cite{BBPS18}].
\begin{lemma}\label{lem:lowerbound}
For each initial $\hat{z} = (u,x,v) \in \Hbf\times P\T^d$, $\xi \in \Wbf\times T_vP\T^d$ with $\|\xi\|_{\Wbf} = 1$ and $s$ such that $(2\alpha +d)/4 < s < \sigma - 2(d-1)$ we have 
\[
	\sup_{k\in \mathbb{K}}\left\{|\langle g_k, \xi\rangle_{\Wbf}|,|\langle DF(\hat{z})g_k,\xi\rangle_{\Wbf}|^2\right\} \geqc_{s} \frac{\|\xi\|_{\Hbf^{-s}}^2}{(1+\|u\|_{\Hbf})^2}.
\]
\end{lemma}
\begin{proof}
Denote $\xi = (\xi_u,\xi_{x,v})$. First,  for $s > (2\alpha + d)/4$ it is not hard to deduce
\begin{equation}\label{eq:uniform-lower-bound1}
	\sup_{k\in\mathbb K}|\langle g_k, \xi\rangle_{\Wbf}| \geqc_s \|\xi^u\|_{\Hbf^{-s}}^2.
\end{equation}
Therefore \eqref{eq:uniform-lower-bound1} follows since $4s - 2\alpha > d$ and $\|\xi^u\|_{\Wbf} = 1$.
Next, we note that
\[
	\langle DF(\hat{z})g_k,\xi\rangle_{\Wbf} = q_k\langle B(e_k,u) + B(u,e_k),\xi_u\rangle_{\Wbf} + \langle Ae_k,\xi_u\rangle_{\Wbf} + \langle DF^{x,v}(\hat{z}) g_k,\xi_{x,v}\rangle_{T_v P \mathbb T^d},
\]
and that since $s< \sigma -2$ and $\sup_{k\in \mathbb{K}}q_k \|e_k\|_{\Hbf} < \infty$, 
\begin{align}
	&\|\xi^u\|_{\Hbf}^2+ \sup_{k\in \mathbb{K}}|\langle DF^{x,v}(\hat{z}) g_k,\xi^{x,v}\rangle_{\Wbf}|^2\\
	&\qquad \leqc \sup_{k\in \mathbb{K}}|\langle DF(\hat{z})g_k,\xi\rangle_{\Wbf}|^2 + (1+\|u\|_{\Hbf})^2\|\xi^u\|_{\Hbf}^2 \\ 
  &\qquad \leqc (1+\|u\|_{\Hbf})^2 \sup_{k\in \mathbb{K}}\left\{|\langle g_k, \xi\rangle_{\Wbf}|,|\langle DF(\hat{z})g_k,\xi\rangle_{\Wbf}|^2\right\}, 
\end{align}
where the last line followed from \eqref{eq:uniform-lower-bound1}. The proof is complete upon appealing to [Lemma 5.3, \cite{BBPS18}] and deducing 
\[
	\sup_{k\in \mathbb{K}}|\langle DF^{x,v}(\hat{z}) g_k,\xi^{x,v}\rangle_{T_vP\T^d}| \geqc \|\xi^{x,v}\|_{T_vP\T^d}.
\]
\end{proof}

Before continuing with the nondegeneracy calculation, we introduce the following useful short-hand notation from [Definition 6.9, \cite{HM11}].
\begin{definition}
Given a set $H = \set{H^\eps}_{\eps \leq 1}$, of measurable subsets $H^\eps \subset \Omega$, we will say ``$H$ is a family of negligible events'' if $\forall p \geq 1$, $\exists C_p$ such that $\PP(H^\eps) \leq C_p \eps^p$. Given such a family and a statement $\Phi_\eps$, we say ``$\Phi_\eps$ holds modulo $H$'' if $\forall \eps \leq 1$, the statement $\Phi_\eps$ holds on $(H^\eps)^c$. Finally we say the family $H$ is ``universal'' if it does not depend on the problem at hand. We say a set $H$ is ``$\Psi$-controlled'' for a function $\Psi$ if $C_p \lesssim_p \Psi^p(\hat z)$ (the initial condition of the projective process). 
\end{definition}

Recall the H\"older semi-norm of a function $f:[0,1]\to\R$ for $\alpha \in (0,1]$ is defined by
\begin{align}
[f]_{\alpha} = \sup_{t \neq s} \frac{|f(t) - f(s)|}{\abs{s-t}^{\alpha}}. 
\end{align}
Recall the following standard interpolation lemma (see e.g. \cite{HM11}). 
\begin{lemma}\label{lem:interpolation}
Let $f$ be a $C^{1,\alpha}$ function on $[0,1]$, $\alpha \in (0,1]$, then the following inequality holds
\[
	\|\partial_t f\|_{L^\infty} \leq 4 \|f\|_{L^\infty}\max\left\{1,\|f\|_{L^\infty}^{-\frac{1}{1+\alpha}}[\partial_t f]_{\alpha}^{\frac{1}{1+\alpha}}\right\}.
\]
\end{lemma}

Recall the following formula for the Malliavin matrix
\[
	\langle\xi,\mathcal{M}_{0,1}\xi\rangle_{\Wbf} = \sum_{k \in \mathbb K} \int_0^1 \langle g_k, K_{s,1}\xi\rangle_{\mathbf{W}}^2\ds.
\]
In light of this, we have the following implication:
\begin{lemma}\label{lem:first-implication}
For all $a\geq 2$ sufficiently large and $\forall \eta \in (0,\eta^\ast)$ the following implication holds
\[
	\langle \xi, \mathcal{M}_{0,1}\xi\rangle_{\Wbf} \leq \ep\quad \Rightarrow \quad \sup_{k}\sup_{0\leq t\leq 1}\left|\langle g_k, K_{t,1}\xi\rangle_{\Wbf}\right| \leq \ep^{1/8},
\]
modulo a $V_{a,\eta}$-controlled negligible set. 
\end{lemma}
\begin{proof}
Fix an arbitrary $\eta^\prime \in (0,\eta^\ast)$ and $a' \geq 2$. 
Define $f_k(t) = \int_0^t \langle g_k, K_{s,1}\xi\rangle_{\Wbf}\ds$ and note that $f_k(t)$ is $C^2$ and satisfies
\[
	\partial_t f_k = \langle g_k, K_{t,1}\xi\rangle_{\Wbf}\quad \text{and}\quad \partial^2_t f_k = \langle DF(\hat{z}_t)g_k, K_{t,1}\xi\rangle_{\Wbf}.
\]
Additionally, when $\langle \xi, \mathcal{M}_1\xi\rangle_{\Wbf} \leq \ep$, $\|\xi\|_{\Wbf} = 1$, we have $\sup_{k}\|f_k\|_{L^\infty}\leq  \ep^{1/2}$ by \eqref{eq:MalNonDeg} and Lemma \ref{lem:interpolation}. 

Furthermore, by Lemma \ref{lem:dynamic-bounds} there holds
\[
	\E \sup_k\|\partial^2_t f_k\|_{L^\infty}^p \leqc \left(\E \sup_{0\leq t \leq 1}(1 + \|u_t\|_{\Hbf^{2(d-1)}})^{2p}\right)^{1/2}\left(\E \sup_{0\leq t \leq 1}\|K_{t,1}\|_{\Wbf \to \Wbf}^{2p}\right)^{1/2}\leqc  V^{p}_{\beta,\eta'}(u). 
\]
and therefore, Chebyshev's inequality implies that modulo a $V_{4a',\eta'}$-controlled negligible set one has the bound $\sup_k\|\partial^2_tf_k\|_{L^\infty} \leq (1/16)\ep^{-1/4}$. Applying the interpolation Lemma \ref{lem:interpolation} to $f_k$ with $\alpha =1$ modulo this family of negligible sets gives the implication. 
\end{proof}

\begin{lemma}\label{lem:second-implication}
For all $a \geq 2$ sufficiently large and $\forall \eta \in (0,\eta^\ast)$ sufficiently small, the following implication holds
\[
	\sup_{k}\sup_{0\leq t\leq 1}\left|\langle g_k, K_{t,1}\xi\rangle_{\Wbf}\right| \leq \ep \quad \Rightarrow \quad \sup_k\sup_{0\leq t\leq 1}|\langle DF(\hat z_t)g_k, K_{t,1}\xi\rangle_{\Wbf}|\leq \ep^{1/8}
\]
modulo a $V_{a,\eta}$-controlled negligible set.
\end{lemma}
\begin{proof}
Fix $\eta^\prime \in (0,\eta^\ast)$ and $a' \geq 2$  and let $f_k = \int_0^t \langle g_k, K_{s,1}\xi\rangle_{\Wbf}\ds$. Our goal will be to apply the interpolation Lemma \ref{lem:interpolation} to $\partial_t f_k$. However, since $DF(\hat z_t)g_k$ is only H\"older continuous, $\partial^2_tf_k$ is not $C^1$ and we will need to obtain moment estimate on the $\alpha = 1/3$ H\"{o}lder semi-norm of $\partial^2_tf_k$. Indeed, applying the interpolation Lemma \ref{lem:interpolation} to $\partial_tf_k$ with $\alpha = 1/3$ gives
\begin{equation}\label{eq:interpbound-prelim}
	\sup_k\|\partial_t^2f_k\|_{L^\infty} \leq 4\ep^{1/4}\max\left\{\ep^{3/4},\sup_k[\partial^2_t f_k]_{1/3}^{3/4}\right\}. 
\end{equation}
The required H\"older estimate on $\partial_t^2 f_k$ then follows from the high regularity of $u\in \Hbf$, by standard time regularity estimates on Wiener processes and Lemma \ref{lem:dynamic-bounds}, 
\begin{align}
	\E \sup_k [\partial^2_t f_k]_{1/3}^{p} \leqc V^{2p}_{a',\eta'}(u). 
\end{align}
Therefore, modulo a $V_{12a',12\eta'}$-controlled negligible set one has the bound $\sup_k [\partial^2_t f_k]_{1/3} < 2^{-8/3}\ep^{-1/6}$. Substituting this into \eqref{eq:interpbound-prelim} gives the desired result.

\end{proof}

\begin{proof}[Proof of Proposition \ref{prop:malnodeg}]
To prove this, we note that Lemmas \ref{lem:first-implication} and \ref{lem:second-implication} imply that for each $\delta \in(0,1)$ and $\xi \in \mathcal{S}_\delta$ the following implication holds modulo a $V_{\beta,\eta}$-controlled family of negligible events
\[
	\langle \xi,\mathcal{M}_{0,1}\xi\rangle_{\Wbf} < \ep
	\quad \Rightarrow\quad
	\begin{cases}
	\sup_k\sup_{0\leq t\leq 1}\left|\langle g_k, K_{t,1}\xi\rangle_{\Wbf}\right| \leq \ep^{1/8}\\
	\sup_k\sup_{0\leq t\leq 1}|\langle DF(\hat z_t)g_k, K_{t,1}\xi\rangle_{\Wbf}|\leq \ep^{1/64}.
	\end{cases}
\]
By taking $t = 1$, this implies 
\[
	\sup_k \left\{|\langle g_k, \xi\rangle_{\Wbf}|,|\langle DF(\hat{z})g_k,\xi\rangle_{\Wbf}|^2\right\} \leq \ep^{1/32}.
\]
However, appealing to Lemma \ref{lem:lowerbound} and recalling definition \ref{def:Sbeta} implies 
\[
	\frac{\delta^2}{(1 + \|u\|_{\Hbf})^2} \leqc \frac{\|\xi\|_{\Hbf^{-s}}^2}{(1+\|u\|_{\Hbf})^2} \leqc \ep^{1/32}. 
\]
Therefore, choosing $\ep$ small enough like $\ep \leqc_{\delta} (1 + \|u\|_{\Hbf})^{-64}$, we deduce that
\[
	\P\left(\langle \xi,\mathcal{M}_{0,1}\xi\rangle_{L^2} < \ep\right) \leqc V^{p}_{a,\eta}(u)\ep^p.
\]
We can extend this estimate to all $\ep > 0$, by noting the other case is $1 \leqc (1 + \|u\|_{\Hbf})^{64p}\ep^{p}$, and hence the proof is complete by choosing $a > 64$.
\end{proof}

\section{Spectral theory for the twisted Markov semi-group}\label{sec:spectral-twist}

As discussed in Section \ref{subsec:outline2PTDecay}, our approach to proving a drift condition for the two-point process involves using spectral properties of the `twisted' Markov semi-group $\hat{P}_t^p$ defined for bounded measurable $\psi : \Hbf \times P \T^d \to \R$, by
\[
\hat P^p_t \psi(u,x,v) := \E_{(u,x,v)}|D\phi^t v|^{-p} \psi(u_t, x_t, v_t).
\]
To simplify notation, we will denote $\hat{z} = (u,x,v) \in \Hbf\times P\T^d$ and $\hat{z}_t = (u_t,x_t,v_t)$. It is important to note that the semi-group $\hat{P}^p_t$ can be written as a Feynman-Kac semi-group
\begin{align}
	\hat{P}^p_t\psi(\hat{z}) = \E_{\hat{z}} \exp\left(-p\int_0^t H(\hat{z}_s)\ds
\right) \psi(\hat{z}_t), \label{def:FeynSem}
\end{align}
where
\[
	H(u,x,v) := \langle v, \nabla u(x)v\rangle.
\]

As discussed in the outline, our analysis requires that we study $\hat{P}^p_t$ with respect to 
both the $C_V$ norm (as defined in \eqref{def:CV}) as well as the stronger $C_V^1$ norm
\[
	\|\psi\|_{C^1_V} = \sup_{\hat{z} \in \Hbf\times P\T^d} \left(\frac{|\psi(\hat{z})|}{V(u)} + \frac{\|D\psi(\hat{z})\|_{\Hbf^*}}{V(u)}\right) 
\]
where $V(u) = V_{\beta,\eta}(u)$ for some choice of $\beta >1$ and $\eta <\eta^*$. 

The plan is as follows: We start in Section \ref{subsec:CVframework} in the $C_V^1$ framework
by proving Proposition \ref{prop:specgapC1V}, establishing a 
spectral gap for $\hat P_{T_0}^p$ in $C_V^1$ for $T_0$ sufficiently large. In  
Section \ref{subsec:CV1frame} we obtain $C_0$ continuity for $\hat P_t^p$ on
the space $\mathring C_V$ as in Proposition \ref{prop:specgapC1V-allt}. Finally, in Section \ref{subsec:propsPsiP}
we pull this all together to prove the desired properties of the eigenfunction $\psi_p$
in Proposition \ref{prop:psiP}.

\subsection{Proof of Proposition \ref{prop:specgapC1V}}\label{subsec:CVframework}

As discussed in the outline, Proposition \ref{prop:specgapC1V} follows from a spectral perturbation argument. Let us first record the following standard spectral perturbation lemma, which is a consequence of analytic functional calculus (see, e.g., \cite{dunford1957linear}).

\begin{lemma}\label{lem:specPerturbLF}
Let $L$ be a bounded linear operator on a Banach space $\mathfrak B$ with norm $\| \cdot \|$. 

\begin{itemize}
	\item[(a)] For any $\epsilon > 0$, there exists $\delta > 0$ such that for any bounded linear $L'$
	on $\mathfrak B$ with $\| L - L' \| < \delta$, we have $\sigma(L') \subset B_\epsilon(\sigma(L))$.
	\item[(b)] Let $S \subset \sigma(L)$ be a closed, isolated subset, i.e., 
	there exists an open set $\mathcal U \subset \C$ such that $S \cap \mathcal U = \sigma(L) \cap \mathcal U$.
	Let $\pi_S$ denote the spectral projector corresponding to $S$. Then, for any $\epsilon > 0$, there exists $\delta > 0$
	such that if $\|L - L'\| < \delta$ for some bounded linear $L'$,
	then $S' := \mathcal U \cap \sigma(L')$ is a closed, isolated subset of $\sigma(L')$ with the property that
	the spectral projector $\pi_{S'}$ for $L'$ corresponding to $S'$ satisfies $\| \pi_S - \pi_{S'} \| < \epsilon$.
\end{itemize}
\end{lemma}

To apply Lemma \ref{lem:specPerturbLF}, we show that $\hat{P}_t^p$ is a bounded perturbation of $\hat{P}_t$ for large enough $t$ and small enough $p$. 

\begin{lemma}\label{lem:specPicCV1}
There exist $T_0 > 0, p_0 > 0$ such that $\forall \beta$ sufficiently large and all $\eta \in (0,\eta^\ast)$ there holds the following:
\begin{itemize}
\item[(a)] $\hat P_{T_0}^p : C_V^1 \to C_V^1$ is a bounded linear operator for all $p \in [-p_0,p_0]$. 

\item[(b)] We have $\hat P_{T_0}^p \to \hat P_{T_0}$ in norm on $C_V^1$ as $p \to 0$.

\end{itemize}
\end{lemma}

\begin{proof}
We first consider the proof of (a).
Note that $\abs{H(\hat z)} \leq \norm{\grad u}_{L^\infty}$, and hence $C_V$ boundedness follows from \eqref{def:FeynSem} and Lemma \ref{lem:TwistBd}. 
Next, compute $D \hat {P}_{t}^p \psi\xi$ for some $\xi \in \Hbf \times T_v P\T^d$:
\begin{align}
D \hat {P}_{t}^p\psi \xi = \EE \exp\left( -p\int_0^t H(z_s) ds\right) \left(D \psi(z_t) J_{0,t} \xi - p \psi(z_t) \int_0^t DH(z_s) J_{0,s}\xi ds \right). 
\end{align}
By Lemma \ref{lem:PathJacobian} there holds (for any $r \in (\frac{d}{2}+1,3)$), 
\begin{align}
\abs{D \hat {P}_{t}^p\psi \xi} & \leq \norm{\psi}_{C^1_V}\EE \exp\left(C \int_0^t \norm{u_s}_{H^r} ds\right) V_{\beta,\eta}(u_t) \left(1 + \sup_{0 < s <t} \norm{u_s}_{\Hbf}^{q'+1} \right) \norm{\xi}_{\Hbf \times T_v P\T^d} \\
& \leq \norm{\psi}_{C^1_V}\EE \exp\left(C\int_0^t \norm{u_s}_{H^r} ds\right) \sup_{0 < s < t}V_{\beta + q'+1,\eta}(u_s)  \norm{\xi}_{\Hbf \times T_v P\T^d}. 
\end{align}
Next, choosing $T_0$ sufficiently large such that
\begin{align*}
\left(\beta + q'+ 1\right) e^{-\gamma T_0} < \beta, 
\end{align*}
Lemma \ref{lem:TwistBd} implies $\forall t \geq T_0$, $\exists C = C(t,\beta,\eta)$,  
\begin{align}
\abs{D \hat {P}_{t}^p\psi \xi} \leq C(t,\beta,\eta)\norm{\psi}_{C^1_V} V(u) \norm{\xi}_{\Hbf \times T_v P \T^d}.  
\end{align}
This proves (a). 

Next, consider (b). First, observe that for $x \geq 0$, $\abs{e^{px} - 1} \leq |p| e^{(|p|+1)x}$, and hence 
\begin{align}
\abs{\hat P_{t}^p \psi - \hat P_{t} \psi} 
  & \leq \norm{\psi}_{C_V} \EE \abs{\exp\left( -p\int_0^t H(z_s) ds\right)- 1 } V(z_t) \\
  & \leq |p|\norm{\psi}_{C_V} \EE \exp\left( (1+|p|)\int_0^t \norm{\grad u_s}_{L^\infty} ds\right) V(z_t), 
\end{align}
and hence convergence in $C_V$ holds by Lemma \ref{lem:TwistBd}. 
For $\xi \in \Hbf \times T_v P \T^d$, 
\begin{align}
\abs{D\hat {P}_{t}^p\psi\xi - D\hat {P}_{t}\psi  \xi} & \leq \EE \abs{\exp\left( -p\int_0^t H(z_s) ds \right) - 1 } \abs{D \psi(z_t) J_{0,t} \xi} \\ & \quad + |p| \EE  \abs{\int_0^t DH(z_s) J_{0,s}\xi ds} \exp\left( -p\int_0^t H(z_s) ds\right) \abs{\psi(z_t)}. 
\end{align}
Convergence then holds by the same arguments used to prove boundedness for $t \geq T_0$ combined with that used to prove convergence in $C_V$. 
\end{proof}

As discussed in Section \ref{subsec:outline2PTDecay}, we need to work with the spaces $\mathring{C}_V$ and $\mathring{C}_V^1$
which are, respectively, the $C_V$-closure and $C_V^1$-closure of the space of smooth `cylinder functions' 
$\mathring{C}^\infty_0(\Hbf \times P \T^d)$ (see \eqref{eq:defnC0inftyOutline}).

\begin{lemma} \label{lem:OneRingToRuleThemAll}
For all $\beta$ sufficiently large and all $\eta \in (0,\eta^\ast)$ there holds the following. 
\item[(a)] For all $t > 0$ and $p \in \R$, $\hat{P}_t^p:C_V \to C_V$ is bounded and $\hat {P}_{t}^p(\mathring C_V) \subset \mathring C_V$.
\item[(b)] For $T_0, p_0$ as in Lemma \ref{lem:specPicCV1}, $\forall p \in [- p_0,p_0]$, we have that $\hat P_{T_0}^p (\mathring C_V^1) \subset \mathring C_V^1$.
\end{lemma}
\begin{proof}
Consider part (a) first. Boundedness in $C_V$ was proved at the beginning of Lemma \ref{lem:specPicCV1}.
To check $\hat P_t^p (\mathring C_V) \subset \mathring C_V$, note that 
 since $\hat P_t^p$ is a continuous linear operator on $C_V \to C_V$ and $\mathring C_V$ is a closed subspace of $C_V$
it suffices to prove that $\hat P_t^p$ maps a dense set of $\mathring C_V$ into $\mathring C_V$. 
To do so, we show that for all $\psi \in \mathring{C}^\infty_0(\Hbf \times P \T^d)$, 
\begin{align}\label{eq:suffPropagateMathring}
	\lim_{n\to\infty} \|\hat{P}^p_t\psi - (\hat{P}^p_t\psi)\circ \Pi_{n}\|_{C_V} = 0, 
\end{align}
where recall that $\Pi_n: \Hbf \times P\T^d \to \Hbf_{n} \times P\T^d$ is the projection onto Fourier modes satisfying $\abs{k} \leq n$
(note that for any $\varphi \in C_V$, $\varphi \circ \Pi_n \in \mathring{C}_V$). 
With this in mind, denote $\hat{z} = (u,x,v)\in \Hbf\times P\T^d$ and $\hat{z}^n= \Pi_{n} \hat{z}$,
 and let $\Phi_t$ denote the random flow for the projective process on $\Hbf\times P\T^{d}$ with $\hat{z}_t = \Phi_t(\hat{z})$, $\hat{z}_t^n = \Phi_t(\hat{z}^n)$.
 A direct calculation reveals, using that $\psi \in \mathring{C}^\infty_0$, 
\begin{equation}
\begin{aligned}
&|\hat P^p_t\psi(\hat{z}) - \hat{P}^p_t\psi(\hat{z}^n)|\\
&\leq \E\left|\exp\left(-p\int_0^t H(\hat{z}_s)\ds\right)\psi(\hat{z}_t) - \exp\left(-p\int_0^tH(\hat{z}^n_s)\ds\right)\psi(\hat{z}^n_t)\right|\\
&\leqc_\psi  \E\, \Gamma(t) \sup_{s\in [0,t]}d_{\Wbf\times P\T^d}(\hat{z}_s,\hat{z}^n_s), 
\end{aligned}
\end{equation}
where we have defined for $r \in (\frac{d}{2}+1,3)$, 
\[
	\Gamma(t) := \exp\left(\int_0^t \|u_s\|_{H^{r}} + \|u^n_s\|_{H^{r}}\ds\right)\sup_{s\in [0,t]}( 1+ \|u_s\|_{\Hbf} + \|u^n_s\|_{\Hbf})
\]
and $d_{\Wbf\times P\T^d}(\hat{z}_1,\hat{z}_2)$ is the metric on $\Wbf\times P\T^d$.
In particular, the fact that $\psi \in \mathring{C}^\infty_0$ allowed to exchange $\Hbf \times P\T^d$ for $\Wbf \times P\T^d$.
Following an analysis similar to the proof of the Jacobian estimate in Lemma \ref{lem:PathJacobian}, we find that $d_{H^{r}\times P\T^d}(\hat{z}_t,\hat{z}^n_t)$ satisfies the pathwise bound
\begin{align}
	d_{H^{r}\times P\T^d}(\hat{z}_t,\hat{z}^n_t) \leqc \Gamma(t)^{c_1} d_{H^{r}\times P\T^d}(\hat{z},\hat{z}^n), \label{ineq:zznEst}
\end{align}
for some constant $c_1 >1$.

A straightforward extension of Lemma \ref{lem:TwistBd} to the joint system $(\hat{z}_t,\hat{z}^n_t)$ gives for each $t>0$
\begin{align}\label{eq:boundGammaExpectationfff}
	\E \sup_{s\in [0,t]} \Gamma(t)^{1 + c_1}\leqc V_{\beta',\eta}(u),
\end{align}
for $\beta'$ sufficiently large and for al $0<\eta<\eta^\ast$. Therefore for $V = V_{\beta' + 1,\eta}$, we obtain
\[
	\|\hat{P}^p_t\psi - \hat{P}^p_t\psi\circ \Pi_n\|_{C_V} \leqc \sup_{\hat{z} \in \Hbf \times P\T^d}\frac{d_{H^{r}\times P\T^d}(\hat{z},\hat{z}^n)}{1 + \| u \|_{\Hbf}^2}
\lesssim \sup_{u \in \Hbf} \frac{\| (\Id - \Pi_{n})u \|_{H^{r}}}{1 + \| u \|_{\Hbf}^2}
\]
Note that $\norm{(\Id - \Pi_{n}) u}_{H^r} \lesssim n^{r-\sigma} \norm{u}_{\Hbf}$, which completes the proof of \eqref{eq:suffPropagateMathring}.

Turning to part (b): by part (a), it suffices to prove the following for all $\psi \in \mathring{C}^\infty_0(\Hbf \times P \T^d)$: 
\begin{align}
\lim_{n \to \infty} \norm{\hat P^p_{T_0} \psi - (\hat P^p_{T_0} \psi) \circ \Pi_{n}}_{C^1_V} = 0. 
\end{align}
As above, let $\hat{z}^n = \Pi_{n} \hat {z}$ and $\hat{z}_t^n = \Phi_t (\hat{z}^n)$.
It what follows we need to measure the difference between Jacobians of these different trajectories, namely
$D_\xi \hat{z}^n_t =: J^n_t\xi \in \Hbf \times T_{v^n_t} P \T^d$ and $D_\xi \hat{z}_t =: J_t\xi \in \Hbf \times T_{v_t} P\T^d$. These Jacobians map onto different tangent spaces, but on the event 
$E^n_t = \{d_{P^{d-1}}(v_t, v_t^n) \leq 1/50\}$ we can place $v_t, v_t^n$ in the same smooth chart
and identify their tangent spaces with a copy of $\R^{d-1}$. Thus, on $E^n_t$ we can always
make sense of expressions like $\| J_t \xi - J_t^n \xi\|_{H^s}$, $\| J_t - J_t^n\|_{H^s}$.

To estimate the distance between $J_t, J_t^n$, we use the following estimate on the second variation (itself a consequence of the estimates
in Section \ref{sec:Jacobian}): for some $C,q''>0$, 
\begin{align}
\norm{J^{(2)}_{0,t}[\xi,\zeta]}_{H^r} \lesssim \exp\left(C\int_0^t \norm{u}_{H^r}\right)(1 + \sup_{0< s< t}\norm{u}_{H^r}^{q''})\norm{\xi}_{H^r}\norm{\zeta}_{H^r}. 
\end{align}
On the event $E_t^n$, this implies the estimate
\begin{align}
\|J_t - J^n_t\|_{H^r} \lesssim \Gamma(t)^{c_2} d_{H^r \times P\T^d}(\hat{z},\hat{z}^n), \label{ineq:JJn}
\end{align}
for some $c_2 > 1$. 
Then, similar to part (a), for some $r \in (\frac{d}{2}+1,3)$, 
\begin{align}
&\abs{D \hat P^p_t \psi(\hat z) \xi - D \hat P^p_t \psi(\hat{z}^n) \xi}\\
& \leq \abs{\EE  \left(  \exp\left( -p\int_0^t H(\hat z_s) ds\right) D \psi(\hat z_t) J_{0,t} \xi
- \exp\left( -p\int_0^t H(\hat z^n_s) ds\right) D \psi(\hat z^n_t) J_{0,t}^n \xi \right)}  \\
&\qquad + \bigg| \EE    \bigg(  p \psi(\hat z_t) \exp\left( -p\int_0^t H(\hat z_s) ds\right) \int_0^t DH(\hat z_s) J_{0,s}\xi ds    \\
&\qquad -  p \psi(\hat z_t^n) \exp\left( -p\int_0^t H(\hat z_s^n) ds\right) \int_0^t DH(\hat z_s^n) J^n_{0,s}\xi ds \bigg) \bigg| 
\end{align}
We split the above into expectations on $E^n_t, (E^n_t)^c$. The integrand on $E^n_t$
can be bounded using \eqref{ineq:zznEst} and \eqref{ineq:JJn}, resulting in a 
bound $\lesssim \Gamma(t)^{c_3} d_{H^r \times P \T^d}(\hat z, \hat z^n)$ for some $c_3 > 0$. 
On $(E^n_t)^c$ we can bound the integrand as in Lemma \ref{lem:PathJacobian}, while the pathwise estimate
\eqref{ineq:zznEst} and the bound \eqref{eq:boundGammaExpectationfff} 
yields
\begin{align}
\abs{D \hat P^p_t \psi(\hat z) \xi - D \hat P^p_t \psi(\hat{z}^n) \xi} \lesssim \E_{\hat{z},\hat{z}^n}\, \Gamma^{c_4}(t) d_{H^r \times P\T^d}(\hat{z},\hat{z}^n) + 
V_{\beta, \eta}(u) d_{H^r \times P \T^d}(\hat z, \hat z^n)
\end{align}
for some sufficiently large $\beta > 0$. From here the proof proceeds as in part (a) (possibly after increasing $\beta$ further). 
\end{proof}

We are now ready to prove Proposition \ref{prop:specgapC1V}.

\begin{proof}[Proof of Proposition \ref{prop:specgapC1V}]
By Lemma \ref{lem:specPicCV1}, the operator $\hat P_{T_0}$
has a dominant, simple eigenvalue at $1$ (as an operator $C^1_V \to C^1_V$). Let $r_0 \in (0,1)$ such that $\sigma(\hat P_{T_0}) \setminus \{ 1 \} \subset B_{r_0}(0)$, where $r_0 \in (0,1)$. 
Fix $\epsilon \ll 1 - r_0$. By Lemma \ref{lem:specPerturbLF} and the fact that $\hat P_{T_0}^p \to \hat P_{T_0}$ in the operator
norm on $\mathring{C}^1_V$, it follows that for all $|p|$ sufficiently small, we have that $\sigma(\hat P^p_{T_0}) \subset B_{r_0 + \epsilon} (0) \cup B_\epsilon(1)$. Taking $|p|$ sufficiently small, Lemma \ref{lem:specPerturbLF}(b) implies 
that the spectral projector $\pi_p$ for $\hat P_{T_0}^p$ corresponding to
$\sigma(\hat P^p_{T_0}) \cap B_\epsilon(1)$ is close in operator norm to the rank-one spectral projector $\pi_0$
for $\hat P_{T_0}$ corresponding to $\{ 1 \}$. Since rank-one projection operators are an open set in the space of bounded linear operators on $\mathring{C}^1_V$, it follows that $\pi_p$ is rank-one when $|p| \ll 1$.

We conclude that for each such $p$, there is a unique, simple eigenvalue $\lambda_p \in B_\eps(1)$.
To show that this eigenvalue 
is real, note that the complexification of the operator $\hat P^p_{T_0}$ sends real parts of functions to real parts and imaginary parts to imaginary parts. Thus, if $\Im(\lambda_p) \neq 0$ then the complex conjugate $\overline{\lambda_p}$ would also be an eigenvalue. This contradicts the fact that the spectral projector $\pi_p$ is rank-one. We conclude that $\lambda_p$ is real, positive, and coincides with the spectral radius $\rho(\hat P^p_{T_0})$. The value $\Lambda(p)$ is defined so that $\lambda_p = e^{- T_0 \Lambda(p)}$.

Finally, convergence of the limit formula \eqref{eq:definePsiPOutline} follows from 
the standard Gelfand formula and that $\pi_p \mathbf{1} \neq 0$ (by Lemma \ref{lem:specPerturbLF} and $\pi_0 \mathbf{1} = \mathbf{1}$).  
\end{proof}

\subsection{Proof of Proposition \ref{prop:specgapC1V-allt}}\label{subsec:CV1frame}

We next prove Proposition \ref{prop:specgapC1V-allt}, namely that for all $t>0$, $\hat{P}^p_t$ has a spectral gap on $\mathring{C}_V$ and that $\psi_p$ is an eigenfunction for the dominant eigenvalue for 
all $t\geq 0$. As in the proof of Proposition \ref{prop:specgapC1V} we will make a spectral perturbation argument. We will need the following Lemma, which is a simple variation of Lemma \ref{lem:specPicCV1} and is proved in the same way.

\begin{lemma}\label{lemma:boundedCV}
There exists $p_0 > 0$ such that the following holds for all $p \in [p_0,p_0]$, all $\beta$ sufficiently large and all $\eta \in (0, \eta^\ast)$.  
\begin{itemize}
	\item[(a)] For all $t > 0$, we have that $\hat P_t^p$ is a bounded linear operator on $C_V$.
	\item[(b)] For each $t > 0$ fixed, we have $\lim_{p \to 0} \| \hat P_t^p - \hat P_t\|_{C_V} = 0$. 
\end{itemize}
\end{lemma}

Next, let us examine the problem of finding a suitable function-space framework for which
$t \mapsto \hat P_t^p$ is a $C_0$ semigroup. As it turns out, to check $C_0$ continuity in $C_V$
it does not suffice to check boundedness in $C_V$, since it does not admit the countable dense subset
of smooth functions we need to make a strong continuity argument. Instead, we will follow the approach carried out in \cite{HM08} and prove $C_0$ continuity on $\mathring{C}_V$,
which provides a natural, separable, closed subspace of $C_V$.

\begin{proposition}\label{prop:C0-property-twist} For all $p \in [-p_0,p_0]$ all $\beta \geq 1$ sufficiently large and all $0 <\eta <\eta^\ast$, $\hat{P}^p_t$ extends to a $C_0$-semigroup on $\mathring{C}_V$ for $V = V_{\beta,\eta}$. That is, $\hat P^p_t (\mathring C_V) \subset \mathring C_V$ for all $t > 0$, and $t \mapsto 
\hat P^p_t$ is a $C_0$ semigroup on $\mathring C_V$. 
\end{proposition}
\begin{proof}
  Following \cite{HM08} Theorem 5.10, it is sufficient to show that $\hat{P}^p_t$ maps $\mathring{C}_{V}$ into itself and that $t\mapsto \hat P^p_t$ is strongly continuous in $\mathring{C}_0^\infty(\Hbf\times P\T^d)$ in the $C_V$ topology.
  The first step,  $\hat P^p_t (\mathring C_V) \subset \mathring C_V$ for all $t > 0$, is proved in Lemma \ref{lem:OneRingToRuleThemAll}, hence it suffices to check the strong continuity.

To do this, we fix $\psi \in \mathring{C}^\infty_0(\Hbf\times P\T^d)$.
Let $\mathcal{K} \subset \mathbb K$ be the finite set such that we can write $\psi = \varphi \circ \Pi_{\mathcal{K}}$ for some function $\varphi \in C^\infty_0$. 
By It\^{o}'s formula, it follows that for each $r \in (1+d/2, 3)$ there holds for $t \in (0,1]$,
\begin{equation}
\begin{aligned}
	|\hat{P}^p_t\psi(\hat{z}) - \psi(\hat{z})| &\leqc_{\psi} t \,\E \exp\left(\int_0^t \|u_s\|_{H^{r}}\ds\right) \sup_{s\in [0,t]} \left(1 + \|\Pi_{\mathcal{K}} u_s\|_{\Hbf} + \|F(\Pi_{\mathcal{K}} \hat{z}_s)\|_{\Hbf\times P\T^d}\right)\\
	&\leqc_{\psi,n}t \, V_{1,\eta}(u),
\end{aligned}
\end{equation}
where we used the fact that $\|F(\Pi_{\mathcal{K}} \hat{z}_s)\|_{\Hbf\times P\T^d} \leqc_{\mathcal{K}} \|u_s\|_{\Hbf}^2$ and used Lemma \ref{lem:TwistBd}. 
Putting this together, we conclude that for $V = V_{\beta,\eta}, \beta \geq 3/2$, we have
\[
	\|\hat{P}^p_t\psi - \psi\|_{C_V} \leqc_{\psi,n} t \to 0 \quad \text{as}\quad t\to 0.
\]
By density of $\mathring{C}^\infty_0(\Hbf\times P\T^d)$ in $\mathring{C}_V$, we conclude strong continuity of $\hat{P}^p_t$ in $\mathring{C}_V$.
\end{proof}

We are now ready to prove Proposition \ref{prop:specgapC1V-allt}.
\begin{proof}[Proof of Proposition \ref{prop:specgapC1V-allt}]

Let $\Lambda(p)$ be as in Proposition \ref{prop:specgapC1V} and let $s(p) < e^{-T_0 \Lambda(p)}$ be such that $\sigma(P^p_{T_0}) \setminus \set{e^{-T_0 \Lambda(p)}} \subset B_{s(p)}(0)$ as an operator on $C^1_V$.

Lemmas \ref{lem:specPerturbLF} and \ref{lemma:boundedCV}, together with Proposition \ref{prop:CVspecGapProj}, imply
that for all $\eps > 0$, for all $p$ sufficiently small, there are $0 < \tilde{s}(p) < e^{T_0 \tilde{\Lambda}(p)}$ such that $e^{T_0 \tilde{\Lambda}(p)} \in B_\eps(1)$ and $\sigma(P^p_{T_0}) \setminus \set{e^{-T_0 \tilde{\Lambda}(p)}} \subset B_{\tilde{s}(p)}(0)$ as an operator on $C_V$.
Since $\psi_p \in \mathring C_V^1 \subset \mathring C_V$ is already an eigenfunction for $\hat P^p_{T_0}$,
we conclude that in fact $\tilde \Lambda(p) = \Lambda(p)$ 
for all $p$ sufficiently small. 

To complete the proof of Proposition \ref{prop:specgapC1V-allt}, we establish the spectral picture for $\hat P^p_t$ for all $t > 0$
using semigroup theory. To start, the Spectral mapping theorem for the point spectrum (\cite{arendt1986one} A-III Theorem 6.3) implies that $- \Lambda(p)$ is an eigenvalue of the infinitesimal generator $\mathcal A^p$ of $\hat P^p_t$. Corollary 6.4 in Chap. A-III of \cite{arendt1986one} implies that for all $\eta \in \C$,
\begin{align}\label{eq:infGenToSemigroup}
\ker(\eta \operatorname{Id}- \mathcal A^p) = \bigcap_{s \geq 0} \ker(e^{s \eta} \operatorname{Id} - \hat P^p_s) \, .
\end{align}
Applying \eqref{eq:infGenToSemigroup} to $\eta = - \Lambda(p)$, we have that $\psi_p$ is (up to rescaling)
the unique eigenvector for $\mathcal A^p$ for $- \Lambda(p)$; by another application of \eqref{eq:infGenToSemigroup},
we conclude that $\hat P^p_t \psi_p = e^{- \Lambda(p) t} \psi_p$ for all $t > 0$.
\end{proof}

\begin{remark}\label{rmk:compatibleEigenfunctions}
Recall that $\Hbf$ is defined to be the space of divergence free, mean-zero velocity fields
in $H^\sigma$, where $\sigma$ is drawn from the interval $(\alpha - 2 (d - 1), \alpha - \frac{d}{2} )$ (see Section \ref{subsubsec:noiseProcess} for notation; c.f. Propositions \ref{prop:SFscaleProj}, 
\ref{prop:topIrredProj}). In particular, the preceding arguments can be repeated with $\sigma$ replaced
by any other $\sigma' < \sigma$ in this interval: with $\hat P^p_{T_0}, |p| \ll 1$ regarded as an operator on the corresponding $C^1_V$-space, we obtain a corresponding dominant eigenfunction $\psi_p'$, defined on $H^{\sigma'}$ velocity fields and
 continuous in the $H^{\sigma'}$ topology. By parabolic regularization, $\psi_p'$ coincides with $\psi_p$ when restricted to the domain $\Hbf \times P \T^d$.
\end{remark}

\subsection{Properties of $\psi_p$: proof of Proposition \ref{prop:psiP}}\label{subsec:propsPsiP}

It remains to establish the desired properties of $\psi_p$ and $\Lambda(p)$, namely: 
\begin{itemize}
	\item[(A)] Uniform positivity of $\psi_p$ on bounded subsets of $\Hbf \times P \T^d$; and
	\item[(B)] The asymptotic $\Lambda(p) = p \lambda_1 + o(p), p \to 0$, where $\lambda_1 > 0$ is the 
		top Lyapunov exponent of the Lagrangian flow $\phi^t$.
\end{itemize}

First, we prove strict positivity. 

\begin{lemma} \label{lem:unifLowerBoundPsiP}
For any $R > 0$, we have that
		\begin{align}\label{eq:unifLowerBdEFunct} \inf_{\| u \|_{\Hbf} \leq R, (x, v) \in P\T^d} \psi_p(u, x, v) > 0. \end{align}
\end{lemma}
	\begin{proof}
The nonnegative cone $C^1_{V, +} = \{ \psi \in C^1_V : \psi \geq 0\}$
is closed in $C^1_V$; since $\hat P^p_{n T_0} C^1_{V, +} \subset C^1_{V, +}$ and $e^{n T_0 \Lambda(p)} \hat P^p_{n T_0}{\bf 1} \to \psi_p$ as $n \to \infty$ (Proposition \ref{prop:specgapC1V}), 
we conclude $\psi_p \geq 0$.  

To verify that $\psi_p > 0$ pointwise, we use the fact that the Markov semigroup $\hat P_t$ is topologically irreducible 
(Proposition \ref{prop:topIrredProj}). To wit, assume for the sake of contradiction that $\psi_p(z) = 0$ for some point $\hat z = (u, x, v) \in \Hbf \times P \T^d$. Since $\psi_p$ is continuous and not identically equal to zero, it holds that $U_p := \{ \psi_p > 0 \} \subset \Hbf \times P \T^d$ is nonempty and open. Therefore,
\[
\psi_p(\hat z) \geq \P_{\hat z} \big( \hat z_t \in U_p \big)  \E_{\hat z} \left( |D_x \phi^t_{u} v|^{-p} \psi_p (\hat z_t) |  \hat z_t \in U_p \right) > 0 \, .
\]
Note that the same arguments apply
to $\psi_p'$ corresponding to $\sigma' < \sigma$ as in Remark \ref{rmk:compatibleEigenfunctions}, hence
$\psi_p' > 0$ pointwise as well.

To conclude \eqref{eq:unifLowerBdEFunct}, assume for the sake of contradiction that there is a bounded 
sequence $\{\hat z^n = (u^n, x^n, v^n)\} \subset \Hbf \times P \T^d$ for which $\psi_p(\hat z^n) \to 0$. 
With $\sigma' < \sigma$ fixed as in Remark \ref{rmk:compatibleEigenfunctions}, 
let $\hat z^{n'}$ be a subsequence converging in $H^{\sigma'}$ to some $\hat z^*\in \Hbf^{\sigma'} \times P \T^d$. Since 
$\psi_p'$ is $H^{\sigma'}$-continuous and coincides with $\psi_p$ on $H^\sigma$, it follows that $\psi_p' (\hat z^*) = 0$, contradicting the pointwise positivity established earlier. 
\end{proof}

Item (B) is a version of classical results in the ergodic theory of stochastic differential equations
in finite dimensions.
 In that literature (see, e.g., the survey \cite{arnold1986lyapunov}), 
the value $- \Lambda(p)$ is referred to as the \emph{moment Lyapunov exponent}. This terminology
is justified by the following:
\begin{lemma}\label{lem:formulaLambdaP}
For all $(u, x, v) \in \Hbf \times P \T^d$, we have
\[
\Lambda(p) = -\lim_{t \to \infty} \frac{1}{t} \log \E | D_x \phi^t_{u} v|^{-p} \, .
\]
Moreover, the above limit is uniform on bounded subsets of $\Hbf \times P \T^d$.
\end{lemma}
\begin{proof}
The argument of the logarithm on the right-hand side is equal to $\hat P^p_t{\bf 1}$ evaluated at $(u, x, v)$. 
Uniform convergence on bounded subsets of $\Hbf \times P \T^d$ now follows from the $C_V$ limit
$\psi_p = \lim_{t \to \infty} e^{t \Lambda(p)} \hat P_t^p {\bf 1}$ (a consequence of the $C_V$ 
spectral gap) and \eqref{eq:unifLowerBdEFunct}.
\end{proof}

Next, we verify (B) by relating the value $\Lambda(p)$ with the Lyapunov exponent $\lambda_1$ of the 
Lagrangian flow $\phi^t$. The first step is to show that $\Lambda(p)$ is in fact differentiable. 

\begin{lemma} \label{lem:diffy}
The function $p \mapsto \Lambda(p)$ is differentiable in a neighborhood of $p = 0$. 
\end{lemma}
\begin{proof}
Let us first show that $p \mapsto \hat P^p_1$ is Fr\'echet differentiable as an operator-valued function 
in $C_V$. Formally, we expect the derivative $\frac{d}{dp} \hat P^p_1$ to be given by
\[
\frac{d}{dp} \hat P^p_1 \psi(u, x, v) = -\E_{(u,x,v)} \left( \log |D_x \phi^1_{u} v| \cdot |D_x \phi^1_{u} v|^{-p} \psi(u_1, x_1, v_1) \right) \, .
\]
A direct application of Lemma \ref{lem:TwistBd} implies that the RHS defines
a bounded linear operator on $C_V$.
Fr\'echet differentiability of $p \mapsto \hat P^p_1$ follows if
\[
\lim_{h \to 0} \left\| \E \left|  \left( \frac{1}{h} (|D_x \phi^1 v|^{-p - h} - |D_x \phi^1 v|^{-p}) + \log|D_x \phi^1 v| \cdot |D_x \phi^1 v|^{-p} \right)  V(u_1) \right| \right\|_{C_V} = 0 \, .
\]
By the Mean Value Theorem, this is justified as long as
\[
\left \| \E V(u_t)  |D_x \phi^1 v|^{-p} \log^2|D_x \phi^1 v|  \right\|_{C_V} < \infty \, ,
\]
which follows from Lemma \ref{lem:TwistBd}. 

Having established that $p \mapsto \hat P^p_1$ is Fr\'echet differentiable, it now follows from the spectral gap 
for $\hat P^p_1$ in $C_V$  and the standard
contour integral formula for spectral projectors that $p \mapsto \pi_p$ is likewise Fr\'echet differentiable, hence
$p \mapsto \psi_p = \pi_p({\bf 1})$ is also Fr\'echet differentiable. 
Fixing $\hat z^* \in \Hbf \times P \T^d$, we may now express
\[
e^{- \Lambda(p)} = \frac{1}{\psi_p(\hat z^*)} \hat P^p_1 \psi_p(\hat z^*) \, .
\]
Since the right-hand side is a ratio of differentiable functions and the denominator is non-vanishing, we conclude
$p \mapsto \Lambda(p)$ is differentiable.
\end{proof}

\begin{lemma}\label{lem:verifyItemBLF}
Let $p_0 > 0$ be as in Proposition \ref{prop:specgapC1V-allt}.
\begin{itemize}
\item[(a)] The mapping $p \mapsto \Lambda(p)$ is convex on $[- p_0, p_0]$. 
\item[(b)] We have
\[
\Lambda'(0) = \lambda_1 \, .
\]
In particular, since $\lambda_1 > 0$, we have that $\Lambda(p) > 0$ for all $p > 0$ sufficiently small. 
\end{itemize}
\end{lemma}
With our preparations in place, the proof of Lemma \ref{lem:verifyItemBLF} follows from straightforward
versions of standard arguments-- see, e.g., \cite{arnold1984formula, arnold1986lyapunov}. We sketch the proof below for the sake of completeness. 

\begin{proof}
For convexity, let $p, q \in (- p_0, p_0)$ and $\lambda \in [0,1]$ and fix an arbitrary $(u, x,  v) \in \Hbf \times P\T^d$. By Holder's inequality,
\[
\E | D_x\phi^t_u v|^{-\lambda p - (1 - \lambda) q} \leq \left( \E | D_x\phi^t_u v|^{-p} \right) ^\lambda \cdot \left( \E| D_x\phi^t_u v|^{-q} \right)^{1 - \lambda}
\]
Convexity follows on taking the $\log$ of both sides, dividing by $t$, taking $t \to \infty$ and 
applying Lemma \ref{lem:formulaLambdaP}.

Next, for all $p \in (-p_0, p_0)$, we have by Jensen's inequality that
\[
\E | D_x \phi^t_{u} v|^{-p} = \E e^{-p \log | D_x \phi^t_{u} v|} \geq \exp\left(-p \E  \log | D_x \phi^t_{u}| \right) \, .
\]
Taking the $\log$ of both sides, dividing by $t$ and taking $t \to \infty$ results in the inequality
\[
\Lambda(p) \leq p \lambda_1 \, .
\]
In particular, $\Lambda(p) / p \geq \lambda_1$ for $p \in (-p_0,0)$ and $\Lambda(p) / p \leq \lambda_1$ for $p \in (0,p_0)$. 
By convexity, the left- and right-hand derivatives $\Lambda'(0+)$ and $\Lambda'(0-)$ exist; the inequalities above imply that $\Lambda'(0-) \geq \lambda_1$ and $\Lambda'(0+) \leq \lambda_1$. By differentiability,
these values coincide. This completes the proof. 
\end{proof}

\section{Geometric ergodicity for the two point motion} \label{sec:2ptGeoErg}

We now turn to our study of the two point Lagrangian motion $(u_t,x_t,y_t)$. Recall that that given two initial points $(x,y) \in \mathcal{D}^c = \{(x,y) \in \T^d\times\T^d : x \neq y \}$, $(x_t,y_t)$ are defined by
\[
	x_t = \phi^t(x),\quad y_t = \phi^t(y).
\]
This induces a Feller Markov semi-group $P^{(2)}_t$ defined on bounded measurable $\varphi: \Hbf \times \mathcal{D}^c \to \R$ by
\[
	P^{(2)}_t\varphi(u,x,y) := \E_{(u,x,y)} \varphi(u_t,x_t,y_t).
\]

We eventually apply Theorem \ref{thm:GM} to $P^{(2)}_t$ to prove Theorem \ref{thm:2-pt-decay}. As discussed in Section \ref{subsec:outline2PTDecay}, Conditions \ref{defn:SF} and \ref{defn:TopIrr} follow from Propositions \ref{prop:SFscaleIntro} and \ref{prop:topIrredIntro} (strong Feller and topological irreducibility), which we prove below.
This shows that $\mu \times Leb \times Leb$ is the unique stationary measure for $P^{(2)}_t$ (see \cite{DPZ96}). 
Similarly, Propositions \ref{prop:SFscaleIntro} and \ref{prop:topIrredIntro} imply equivalence of transition kernels.
\begin{lemma}\label{lem:equivFamilyP^2}
The family of transition kernels $\{ P^{(2)}_t((u, x, y), \cdot) : t > 0, (u, x, y) \in \Hbf \times \Dc^c\}$ are equivalent measures.
\end{lemma}
As we saw in the proof of Lemma \ref{lem:hypothA3}, equivalence of the transition kernels implies the following. 
\begin{lemma}\label{lem:unifBoundedCompact12312}
Let $K \subset \Hbf \times \mathcal{D}^c$ be any compact set with $(\mu \times \Leb \times \Leb)(K) > 0$. Then for arbitrary $t > 0$ and $R_1,R_2 > 0$, we have
	\begin{align}
	\inf_{ \| u \| \leq R_1,\,d(x, y) \geq R_2} P^{(2)}_t((u, x, y) , K) > 0.
\end{align}
\end{lemma}
Therefore, Conditions \ref{def:UnifBd} and \ref{defn:drift} follow from Proposition \ref{prop:2ptDrift}, which we prove below, and the fact that the Lyapunov function $\Vc$ is \emph{coercive} in the sense that for each $r > 0$, we can always find $R_1,R_2 > 0$ such that
\[
\{(u,x,v) \in \Hbf \times \mathcal{D}^c : \mathcal{V}(u,x,y) \leq r\} \subset \{ (u, x, y) \in \Hbf \times \Dc^c : \| u \| \leq R_1 \, , \text{ and } d(x, y) \geq R_2 \}.
\]

The remainder of this section will be focused on proving Propositions \ref{prop:SFscaleIntro} and \ref{prop:topIrredIntro} and Theorem \ref{prop:2ptDrift}.

\subsection{Strong Feller property: proof of Proposition \ref{prop:SFscaleIntro}} \label{sec:SF}

\subsubsection{Uniform parabolic H\"ormander conditions} \label{sec:UHC}
Due to Assumption \ref{a:Highs}, we are only concerned with spreading the noise to the degrees of freedom on $\mathcal{D}^c$.
Accordingly, denote the vector field
\begin{align}
X_m(x,y) = \begin{pmatrix} e_m(x) \\ e_m(y)\end{pmatrix} \in T_{(x,y)}\mathcal{D}^c \simeq \Real^{2d}.  
\end{align}
In order to prove Proposition \ref{prop:SFscaleIntro}, we use the following uniform spanning. 
It is helpful to observe that vector fields in the set $\mathcal{A}_2$ in the statement below can be identified as incompressible velocity fields; the uniform spanning condition is simply a quantitative statement about the infinitesimal controllability, i.e. that the particles can be nudged in any direction by velocity fields accessible via the noise. 
\begin{lemma} \label{lem:UniSpan}
There holds the following \emph{uniform spanning condition}: for all $x,y \in \mathbb T^d$ with $x \neq y$, if we denote the set of unit vectors in the span of the available vector-fields:
\begin{align}
\mathcal{A}_2 = \set{\sum_{\substack{m=(k,i)\in \mathbb{K}\\\abs{k} \leq 2}} c_{m} X_m : \,\max_{m}\abs{c_{m}} \leq 1} \, , 
\end{align}
then, 
\begin{align}
\inf_{h \in \R^{2d} : \abs{h} = 1}\max_{\Gamma \in \mathcal{A}_2} \abs{ \brak{ \Gamma(x,y), h}_{\Real^{2d}}} \gtrsim d(x,y). \label{ineq:lwbd}
\end{align}
\end{lemma}
\begin{remark} \label{rmk:SharpAssump}
The proof shows that Assumption \ref{a:lowms} is not quite sharp, especially for Galerkin-Navier-Stokes, where hypoellipticity in $(u_t)$ will fill all available degrees of freedom. However, for simplicity of exposition it is easier simply to use the less complicated condition. 
\end{remark}
\begin{proof}
For definiteness, parameterize $\mathbb T^d$ as $(-\pi,\pi]^d$.
Note that by Assumption \ref{a:lowms} and trigonometric identities, $\mathcal{A}_2$ is translation invariant in the sense that if \eqref{ineq:lwbd} holds at some point $(x,y) $ then \eqref{ineq:lwbd} holds also at $(x',y') = (x + \beta \mod 2\pi \mathbb Z^d, y + \beta \mod 2\pi \mathbb Z^d)$ for any vector $\beta \in \Real^d$.

At any point $(x,y) \in \mathcal{D}^c$, we divide the tangent space into $\Real^d_x \oplus \Real^d_y$ where the first $\Real^d_x$ is associated with infinitesimal motions of $x$ and the second $\Real^d_y$ associated with the infinitesimal motions of $y$. 
We are able to restrict ourselves to vectors $X \in \mathcal{A}_2$ which vanish on either $\Real^d_x$ or $\Real^d_y$ as uniform spanning follows by linear combinations (after slightly adjusting the constant). 
By the above symmetry considerations, it suffices to consider the case $x=0$, $y \neq 0$ and show that we can uniformly span $\Real^d_y$ with vector fields in $\mathcal{A}_2$ that also vanish at zero. 

We first consider the case $d=2$.
In what follows we denote $y = (y^{(1)}, y^{(2)})$. 
Let $\delta \in (0,\frac{1}{10})$ be fixed and arbitrary.
Define the set of points where shear flows which vanish at $x=0$ cannot span $\Real^d_y$. 
\begin{align}
  \mathfrak{D} := \set{y \in \mathbb T^2: y = (0,a), \textup{or, } y = (a,0), \textup{ for some } a \in (-\pi,\pi] }.
\end{align}
There are essentially three cases.

\textbf{Case 1 (2D): $y$ is at least $\delta d(x,y)$ away from $\mathfrak{D}$:} \\
\noindent
By trigonometric interpolation, $\exists h(\zeta)$ a linear combination of $\cos 2\zeta$, $\sin 2\zeta$, $\cos \zeta$, $\sin \zeta$ with coefficients having absolute value less than one such that satisfies $h(0) = 0$ and $h(y^{(2)}) \gtrsim \delta \abs{y}^{(2)}$. 
Hence, the vector field, $Y(x,y)|_{\Real^2_x} = (h(x^{(2)}),0)$, $Y(x,y)|_{\Real^2_y} = (h(y^{(2)}),0)$, 
\begin{align}
\abs{ \brak{ Y(x,y)|_{\Real^2_y}, \begin{pmatrix} 1 \\ 0 \end{pmatrix}}_{\Real^2_y}} \gtrsim_\delta d(x,y). \label{ineq:lwbdHor}
\end{align}
The analogous transverse shear flows span the vertical direction in $\Real^2_y$.

\textbf{Case 2 (2D): $y$ is less than $\delta d(x,y)$ from $\mathfrak{D}$ but more than  $\delta \pi$ away from the points $(0,\pi)$, $(\pi,0)$} \\ 
\noindent
Suppose without loss of generality that $y$ is close to the horizontal line  $y^{(2)} = 0$. 
Shear flows span the vertical direction of $\Real^2_y$ as in Case 1.
To span the horizontal direction, we use the flow
\begin{align}
Y(x,y)|_{\Real^2_{y'}} = \begin{pmatrix} -\sin y^{(1)} \cos y^{(2)} \\ \cos y^{(1)} \sin y^{(2)} \end{pmatrix}, \label{eq:Mode1Cell} 
\end{align}
which gives \eqref{ineq:lwbdHor}. 

\textbf{Case 3 (2D): $y$ is less than $\delta \pi$ away from one of $(0,\pi)$, $(\pi,0)$:} 
\noindent
This case is the most difficult.
Suppose that $y$ is close to $(\pi,0)$; the case $(0,\pi)$ is treated analogously. As in Case 2, shear flows span the vertical direction.
However, there is a new degeneracy at $(\pi,0)$ in the horizontal direction, as the cellular flow in \eqref{eq:Mode1Cell} vanishes. 
To rectify this, we choose the flow
\begin{align}
Y(x,y)|_{\Real^d_{y'}} = \frac{1}{2}\begin{pmatrix} \cos 2y^{(1)} \cos 2y^{(2)} \\ \sin 2y^{(1)} \sin 2y^{(2)} \end{pmatrix} - \frac{1}{2}\begin{pmatrix} \cos y^{(1)} \cos y^{(2)} \\ \sin y^{(1)} \sin y^{(2)} \end{pmatrix}, 
\end{align}
which satisfies \eqref{ineq:lwbdHor}. 

Next, we discuss the extension to $d=3$.\\
In this case, we re-define $\mathfrak{D}$ in the analogous way:
\begin{align}
  \mathfrak{D} := \set{y \in \mathbb T^2: y = (a,0,0), \textup{or, } y = (0,a,0), \textup{or, } y = (0,0,a),  \textup{ for some } a \in (-\pi,\pi] }. 
\end{align}

\textbf{Case 1 (3D): $y$ is at least $\delta d(x,y)$ away from $\mathfrak{D}$:} \\
This case is analogous to the Case 1 above simply by using shear flows in each of the coordinate directions separately. 

\textbf{Case 2 (3D): $y$ is less than $\delta d(x,y)$ away from $\mathfrak{D}$} \\ 
\noindent
Suppose without loss of generality that $y$ is close to the horizontal plane.
The vertical direction (normal to the plane) is spanned by shear flows.
To span the horizontal plane, we use flows that are independent of the normal direction and the problem reduces to the 2D case treated above.

This completes the uniform spanning. 
\end{proof}

\subsubsection{Sketch of proof of Proposition \ref{prop:SFscaleIntro}}
We can essentially apply the same proof as we did in \cite{BBPS18} (which draws heavily from \cite{EH01} combined with some simplifications in the Malliavin calculus and a more sophisticated choice of control to deal with the more complicated nonlinearity). 

The strategy here is to regularize the process for large values of  $\|u\|_{\Hbf}$. This will be done through the use of an auxiliary Wiener process $Z_t \in \R^{2d}$ which will couple to the dynamics through the vector field $L$ on $\R^{2d}$ defined for each $Z\in \R^{2d}$ by
\[
	L(Z) := \sum_{j=1}^{2d} \hat e_j \frac{Z^j}{\sqrt{1+|Z^j|^2}},
\]
where $\{\hat e_j\}_{j=1}^{2d}$ are the canonical basis elements for $\R^{2d}$. The cut-off process $\bar{z}_t^\rho \in \Hbf\times \T^{2d}\times \R^{2d}$ is then defined by augmenting by $Z_t$ so that $\bar{z}^{\rho}_t = (u_t^\rho,x_t^\rho,y_t^\rho,Z_t)$ satisfies the cut-off equation
\[
	\partial_t \bar{z}^{\rho} = F^\rho(\bar{z}^{\rho}_t) - A\bar{z}^{\rho}_t + \bar{Q}\dot{\bar{W}}_t,\quad \tilde{z}^\rho_0 = \tilde{z}\in \Hbf\times \mathcal{D}^c\times\R^{2d},
\]
where $\tilde{{Q}}\tilde{W}_t = (QW_t,0,0,W^Z_t)$ for $W^Z_t$ a Wiener process on $\R^{2d}$ independent from $W_t$, and (for $F$, $A$, and $L$ suitably extended to vector fields on $\Hbf\times \mathcal{D}^c\times \R^{2d}$), for each $\bar{z} = (z,Z) \in \Hbf\times\mathcal{D}^c\times\R^{2d}$
\[
	F^{\rho}(\bar{z}) := (1-\chi_{2\rho}(\|u\|_{\Hbf})) F(z)+ \chi_{\rho}(\|u\|_{\Hbf}) L(Z),
\]
where $\chi_\rho(r) = \chi(r/\rho)$ with $\chi \in C^\infty(\Real_+)$ monotone increasing, non-negative, with $\chi(r) = 0$ for $r \leq 1$ and $\chi(r) = 1$ for $r > 2$. 

Let $P^{(2),\,\rho}_t$ denote the Markov semi-group associated with the cut-off process $(\bar z^\rho_t)$.
The main step in \cite{BBPS18} (and \cite{EH01}) is to prove the following gradient bound. This is done via Malliavin calculus with a low frequency approximation and short-time perturbation argument  to obtain the control. The main ingredient specific to the two-point motion is the uniform parabolic H\"ormander condition proved above in Lemma \ref{lem:UniSpan}. 
\begin{proposition}\label{prop:cut-off-SF}
There exists constants $a, b>0$ such that for each $\varphi \in C^2(\Hbf \times \T^{2d}\times \R^{2d})$ and each $\bar{z} = (u,x,y,Z) \in \Hbf \times \mathcal{D}^c\times \R^{2d}$, the derivative $DP^{(2),\,\rho}_t \varphi(\bar{z})$ exists and satisfies for each $\xi\in \Hbf \times \R^{4d}$
\begin{equation}
	|DP^{(2),\,\rho}_t\varphi(\bar z) \xi| \leqc_\rho t^{-a}d(x,y)^{-b} (1+ \|u\|_{\Hbf} + |Z|)^b\|\varphi\|_{L^\infty}\|\xi\|_{\Hbf\times \R^{4d}}.
\end{equation}
\end{proposition}

Using Proposition \ref{prop:cut-off-SF}, one can prove the strong Feller property for the non cut-off process $z_t$ using the following metric on $\Hbf\times\mathcal{D}^c$
\[
	d_{\,b}(z^1,z^2) := \inf_{\gamma : z^1\to z^2}\int_0^1 d(x_s,y_s)^{-b}(1+\|u_s\|_{\Hbf})^b\|\dot{\gamma}_s\|_{\Hbf \times \R^{2d}}\ds, 
\]
where the infimum is taken over all differentiable curves $[0,1]\ni t \mapsto \gamma_t = (u_t,x_t,y_t)$ in $\Hbf\times\mathcal{D}^c$ connecting $z^1$ and $z^2$. It is not hard to see that the metric $d_b(\cdot,\cdot)$ generates the $\Hbf\times \mathcal{D}^c$ topology since the extremal trajectories avoid the diagonal. 

\begin{proof}[Sketch of proof of Proposition \ref{prop:SFscaleIntro}]
Fix $t>0$ and $\ep >0$ and let $z^1,z^2 \in \Hbf\times \mathcal{D}^c$ and take the initial $Z = 0$. From the moment estimates on $(u_t)$ in Proposition \ref{prop:WPapp}, we can choose the cut-off large enough (depending on $\|\varphi\|_{L^\infty}$ and $\|u_1\|_{\Hbf}$ and $\|u_2\|_{\Hbf}$) such that (see \cite{BBPS18} for more detail), 
\[
	|P^{(2)}_t \varphi(z^1) - P^{(2)}_t\varphi(z^2)| \leq |P^{(2),\,\rho}_t\varphi(z^1,0) - P^{(2),\,\rho}_t\varphi(z^2,0)| + 2\ep.
\]
By Proposition \ref{prop:cut-off-SF} and minimizing along all curves connecting $(z^1,0)$ and $(z^2,0)$, one derives  
\[
	|P^{(2)}_t\varphi(z_1) - P^{(2)}_t\varphi(z_2)| \leqc_{\rho} t^{-a} d_b(z_1,z_2) + 2\ep.
\]
The proof is completed by taking $d_b(z_1,z_2)$ sufficiently small.
\end{proof}

\subsection{Irreducibility: proof of Proposition \ref{prop:topIrredIntro}} \label{sec:Irr2pt}
In this section we prove Proposition \ref{prop:topIrredIntro}, that is, we show that the transition kernel $P^{(2)}_t(z,\cdot)$ is locally positive on $\Hbf \times \mathcal{D}^c$ for $t > 0$ and $z \in \Hbf \times \mathcal{D}^c$.
Proposition \ref{prop:topIrredIntro} is an immediate consequence of the following lemma. 
\begin{lemma} \label{lem:2ptPos}
Let $z,z' \in \Hbf \times \mathcal{D}^c$ be arbitrary. Then, $\forall \eps>0$ and $\forall t > 0$, 
\begin{align}
P^{(2)}_t(z,B_\eps(z')) > 0,
\end{align}
where we denote $B_\eps(z')$ the $\eps$-ball in $\Hbf \times \mathcal{D}^c$.
\end{lemma}

As usual, Lemma \ref{lem:2ptPos} is proved via an approximate control argument. 
Consider the following for $g_t$ a deterministic control
\begin{align}
\partial_t u_t + B(u_t,u_t) + Au_t & = Q g_t \\ 
\partial_t x_t & = u_t(x_t) \\ 
\partial_t y_t & = u_t(y_t). 
\end{align}
We prove that for any $z,z'$ as in the statement of Lemma \ref{lem:2ptPos}, we construct a $g_t \in \mathbf{W}$ such that at time $t$, 
\begin{align}
\norm{u' - u_t}_{H^\sigma}  + d(x',x_t) + d(y',y_t) < \frac{\eps}{2}. 
\end{align}
Moreover, $Qg_t \in C^\infty$ and the size of $\norm{Qg}_{L^\infty_t H^\sigma}$ will depend only $d(x,y)$, $d(x',y')$, $u,u'$ (where we denote $z=(u,x,y)$, $z' = (u',x',y')$) and can be chosen uniformly over compact sets in $\Hbf \times \mathcal{D}^c$.
Local positivity of the Wiener measure together with a stability argument then implies Lemma \ref{lem:2ptPos}; see e.g. [Lemma 7.3 \cite{BBPS18}] for how to carry out such details.

Constructing $g_t$ is a three step procedure: use a `scaling' (see e.g. the discussion in \cite{GHHM18} and the references therein) to force $u_t$ to (approximately) zero in an arbitrarily short time-window. 
Then we use arguments involving well-chosen sequences of shear flows to exactly control the two particles to the desired locations (here the proof is vaguely reminiscent of that of Lemma \ref{lem:UniSpan}). 
Then, we again use a `scaling' to force the $(u_t)$ to (approximately) $u'$ while simultaneously not disturbing the particles by more than $O(\eps)$.

\begin{lemma} \label{lem:scl1}
Let $u \in H^\sigma$ be arbitrary. Then $\forall \eps > 0$, $\exists \delta < \eps$ and a control $g:[0,\delta] \rightarrow \mathbf{W}$ such that  $\norm{u_{\delta}}_{H^\sigma} \leq \frac{\eps}{4}$ and $\sup_{0 \leq t \leq \delta} \norm{u_t}_{H^\sigma} \leq 3 \norm{u}_{H^\sigma}$. 
\end{lemma}
\begin{remark}
This lemma is simplified by the use of fully non-degenerate noise, however, using the methods in \cite{GHHM18} and the references therein, one can obtain essentially the same lemma from any noise that satisfies the H\"ormander bracket conditions for $(u_t)$ discussed in \cite{E2001-lg,Romito2004-rc}. 
\end{remark}
\begin{proof}
The lemma follows by choosing $g_t$ as the following for suitably chosen $\delta$, $N$:
\begin{align}
g_t = -Q^{-1} \delta^{-1} \Pi_{N}u; 
\end{align}
see e.g. discussions in \cite{GHHM18} for more information.
\end{proof}

The next lemma constructs a control to move $x$ to $x'$ and $y$ to $y'$, assuming that the velocity is initially zero. 
\begin{lemma} \label{lem:gctr}
Let $a \in (0,\frac{1}{2})$ and suppose $u_a = 0$, $(x_a,y_a)= (x,y)$. 
For all $x,x,y,y' \in \mathcal{D}^c$, $\exists C_g$,  (depending only on $d(x,y)$ and $d(x',y')$) and a control $g =: g^{ctr,a}$ satisfying $\sup_{t \in (a,1-a)} \norm{g^{ctr,a}_t} \leq C_g$ such that $u_{1-a} = 0$ and $(x_{1-a},y_{1-a}) = (x',y')$. 
\end{lemma}
\begin{proof}	
First observe that, by first moving $x$ and then moving $y$, (and the vice-versa) it suffices to fix one particle and move the other. 
Parameterize $\mathbb T^d$ as $(-\pi,\pi]^d$ and suppose without loss of generality that $x=0$, $y \neq 0$. We will choose the velocity field to fix $x_t= 0$ and satisfy $y_{\frac{1}{2}} = y'$. 

We carry out the proof in $d=2$ for notational simplicity; the $d=3$ case follows similarly. 
Let $\delta > 0$ be fixed arbitrary. 
There are essentially two cases.\\
\noindent
\textbf{Case 1: } $y,y'$ do not lie within $\delta$ of the same coordinate axis. \\ 
\noindent
\textbf{Case 1a:} Neither $y$ nor $y'$ lie within a $\delta$ of any coordinate axis. \\ 
Denote $y = (y^{(1)},y^{(2)})$ and $y' = (y^{'(1)},y^{'(2)})$.
By Assumption \ref{a:lowms}, trigonometric interpolation, and that all shears are stationary solutions of the 2D Euler equations, for any smooth $f_t$ with $f_a = 0$, $\exists g_t$ such that 
\begin{align}
u_{t}(\zeta) = f_t 
\begin{pmatrix} 
h(\zeta^{(2)})\\ 0
\end{pmatrix},
\end{align}
with $h$ a linear combination of $\cos y,\sin y,\cos 2y, \sin 2y$ satisfying $h(0) = 0$ and $h(y^{(2)}) = 1$. 
It is clear that we can choose $f_t$ (and hence $g_t$) such that $y_{\frac{a}{2} + \frac{1}{4}} = (y^{'(1)},y^{(2)})$.
Over the time interval $t \in (\frac{a}{2} + \frac{1}{4},\frac{1}{2})$ we then similarly move the second component of $y$ using a shear flow of the form $(0,b(\zeta^{(2)}))$. 
  
\noindent
\textbf{Case 1b:} One or both $y$ of $y'$ lie within a $\delta$ of a coordinate axis. \\ 
Unlike the previous case, the order in which we apply the shear flows  (i.e. in the horizontal or in the vertical direction) matters. 

Suppose that $y'$ lies on the vertical axis. Then either $y$ lies on the horizontal axis or lies away from either axis. 
We apply the same procedure of two repeated shear flows as the previous case, however we \emph{first} adjust $y^{(2)}$ to $y^{'(2)}$ and \emph{then} adjust $y^{(1)}$ to $y^{'(1)}$.  
If $y'$ lies on the horizontal axis, we proceed similarly, but this time first adjusting $y^{(1)}$ to $y^{'(1)}$ and then adjusting $y^{(2)}$ to $y^{'(2)}$. 

\noindent
\textbf{Case 2: } $y,y'$ lie within a $\delta$ of the same coordinate axis. \\ 
 By symmetry, without loss of generality we can assume that the coordinate axis is the horizontal.
Since $y \neq 0$, by trigonometric interpolation, there exists a shear flow $(0,h(y))$ where $h$ is a trigonometric polynomial supported only in the first two harmonics satisfying 
\begin{align}
h(0) = 0, \quad h(y^{(1)}) = 1.   
\end{align}
Hence, in any time window we can move $y$ back into Case 1b, at which point we proceed as above. 
\end{proof} 

The next lemma simply the reverse of Lemma \ref{lem:scl1}. 

\begin{lemma} \label{lem:scl2}
Let $u' \in \Hbf$ be arbitrary. Then $\forall \eps > 0$, $\exists \delta \ll 1$ and a control $g:[1-\delta,1] \rightarrow \mathbf{W}$ such that if $\norm{u_{1-\delta}}_{\Hbf} \leq \frac{\eps}{4}$, then there holds $\norm{u_1-u'}_{\Hbf} < \frac{\eps}{4}$, $\sup_{1-\delta \leq t \leq 1} \norm{u_t}_{\Hbf} \leq 3 \norm{u'}_{\Hbf}$, and $d(x_{1-\delta},x_1) + d(y_{1-\delta},y_1) \lesssim \delta \norm{u'}_{\Hbf}$. 
\end{lemma}

Finally, we briefly sketch how to assemble the control and prove the necessary stability for the problem. 
 
\begin{proof}[\textbf{Proof of Lemma \ref{lem:2ptPos}}]
	We use the control (where $g^{ctr,\delta}$ is chosen to send $x_{\delta},y_{\delta}$ exactly to $x',y'$ at time $1-\delta$), 
	\begin{align}
	g_t = 
	\begin{cases}
	-\delta^{-1} Q^{-1}u_{\leq N}  & \quad t \in (0,\delta) \\ 
	 g^{ctr,\delta}_{t} & \quad t \in (\delta,1-\delta) \\
	\delta^{-1} Q^{-1} u'_{\leq N}. & \quad t \in (1-\delta,1). 
	\end{cases}
\end{align}
For suitable choices of $\delta$, $N$, the lemma now follows from Lemmas \ref{lem:scl1}, \ref{lem:scl2}, and  \ref{lem:gctr} and local positivity of the Wiener measure as in e.g. [Lemma 7.3, \cite{BBPS18}]. 
\end{proof}

\subsection{Verification of the drift condition: proof of Proposition \ref{prop:2ptDrift}} \label{sec:2ptDrift}

We now commence the analysis of $\Vc$ defined in Section \ref{subsec:outlineConstructV} using the eigenfunction $\psi_p$ constructed in Section \ref{sec:spectral-twist}. We will assume for the rest of this section that $p\in (0,p_0)$ is fixed, ensuring the existence $\psi_p$ by Proposition \ref{prop:psiP}. Recall, that $\psi_p$ belongs to $\mathring{C}^1_V$, where $V= V_{\beta,\eta}$ for all $\beta$ large enough and $\eta \in (0,\eta^*)$. In what follows, we will increase this lower bound on $\beta$ finitely many time without explicitly keeping track of the value. 

For $(x,y)\in \mathcal{D}^c$ we define $w(x,y) = y-x \,\mathrm{mod}\, 2\pi \mathbb Z^d \in \R^d$ (i.e., 
$(x,y)$ is the shortest displacement vector from $x$ to $y$). Then, for each $(u,x,w)\in \Hbf\times \T^d \times \R^d$,  $\Vc$ is of the form
\[
	\Vc(u,x,y) = \hat{h}_p(u,x,w(x,y)) + V(u)
\]
where
\begin{align*}
\hat{h}_p(u,x,w) := \abs{w}^{-p}\psi_p\left(u,x,\frac{w}{\abs{w}}\right) \chi(\abs{w}), 
\end{align*}
and $\chi$ is a smooth, strictly positive cutoff supported in $B(0,\frac{1}{50})$. The cut-off function ensures continuity, since  $w(x,y)$ is continuous on $\mathcal{D}_{1/10}$. Consequently this cut-off allows us to pull $h_p$ back to a continuous function $\hat{h}_p$ on $\Hbf\times \mathcal{D}^c$ by
\[
	h_p(u,x,y) := \hat{h}_p(u,x,w(x,y)).
\]

Our strategy for verifying the drift condition (Condition \ref{defn:drift}) is to show that $\hat h_p$ is an approximate eigenfunction for the two point Markov semi-group $P^{(2)}_t$. To do this, it is convenient to work with the infinitesimal generator $\mathcal{L}_{(2)}$ of $P^{(2)}_t$. Therefore we will need to show that it is a legitimate $C_0$ semi-group on an appropriate Banach space. Moreover, we will need to deal with observables that are unbounded both for large $u$ and as $(x,y)$ approach the diagonal $\mathcal{D}$. To do this we introduce the following weight 
\[
	\widehat{V}(u,x,y) := \widehat{V}_{p,\beta,\eta}(u,x,y) = d(x,y)^{-p}V_{\beta,\eta}(u),
\]
where $p\in (0,p_0)$ and $V_{\beta,\eta}$ is defined by \eqref{def:V}, with $\eta \in (0,\eta^*)$, $\beta \geq 1$ and $d(x,y)$ denotes the natural distance metric on $\T^d$. Treating $p,\eta, \beta$ as fixed for now we then define the following weighted supremum norm
\[
	\|\varphi\|_{C_{\widehat{V}}} := \sup_{z\in \Hbf\times \mathcal{D}^c} \frac{|\varphi(z)|}{\widehat V(z)},
\]
and denote $C_{\hat{V}}$ to be space of continuous functions on $\Hbf\times\mathcal{D}^c$ whose $\|\cdot\|_{C_{\hat{V}}}$ norm is finite. Note that since $\Vc \leqc \widehat{V}$ (for appropriate $p,\beta,\eta$), we have that $\Vc \in C_{\hat{V}}$.

First we show that $P^{(2)}_t$ is a bounded linear operator on $C_{\widehat{V}}$ (although it lacks strong continuity in this space).

\begin{lemma}\label{lem:ChatV-boundedness} For all $p\in (0,p_0)$ ,$\beta \geq 1$ and $\eta \in (0,\eta^\ast)$, $P^{(2)}_t$ extends to a bounded linear operator on $C_{\hat{V}}$. Specifically there exists a constant $C$ such that for each $\varphi \in C_{\widehat{V}}$, 
\[
	\|P^{(2)}_t\varphi\|_{C_{\widehat{V}}} \leq e^{C t}\|\varphi\|_{C_{\widehat{V}}}.
\]
\end{lemma}
\begin{proof}
To prove this, first note that
\[
	|P_t^{(2)}\varphi| \leq \|\varphi\|_{C_{\widehat{V}}} \E\left(d(x_t,y_t)^{-p}V(u_t)\right).
\]
Our first step will be to deduce a lower bound for $d(x_t,y_t)$. To this end, we note that when $d(x_t,y_t) < 1/10$ we find a local chart and represent $x_t$ and $x_t$ as vectors in $\R^d$ so that $d(x_t,y_t) = |x_t - y_t|$ and therefore deduce the differential inequality
\begin{equation}\label{eq:diff-log-eq}
	\frac{\dee}{\dt} \log(|x_t - y_t|) \geq  - \|u_t\|_{H^{r_0}}.
\end{equation}
In order to integrate the above inequality, we must be careful that we integrate over time intervals where the local chart we used remains valid. With this in mind, suppose that $t > 0$ is such that $d(x_{t},y_{t}) < d(x,y)/100$ and define
\[
	\tau_t = \sup \{s: 0\leq s< t, \, d(x_s,y_s) \geq d(x,y)/10\},
\]
to be the last time before $t$ that $d(x_s,y_s)$ was outside the chart. Note that $\tau_t$ is well defined, since $d(x_0,y_0) \geq d(x,y)/10$ and is strictly less than $t$ by continuity of $(x_t,y_t)$. Consequently integrating \eqref{eq:diff-log-eq} from $\tau_t$ to $t$ yields
\[
	|x_{t} - y_{t}| \geq |x_{\tau_t}-y_{\tau_t}|\exp\left(- \int_{\tau_t}^{t} \|u_s\|_{H^{r_0}}\ds\right) \geq \frac{d(x,y)}{10}\exp\left(- \int_0^{t} \|u_s\|_{H^{r_0}}\right).
\]
Of course when $d(x_t,y_t) > d(x,y)/100$, a lower bound is automatic and therefore we obtain
\[
	d(x_t,y_t) \geqc d(x,y) \exp\left(- \int_0^{t} \|u_s\|_{H^{r_0}}\right).
\]
It then follows from Lemma \ref{lem:TwistBd} that we can bound
\begin{equation}
\begin{aligned}
	\E_{z}\left(d(x_t,y_t)^{-p}V(u_t)\right) &\leqc d(x,y)^{-p}\E_u\exp\left(p\int_0^{t} \|u_s\|_{H^{r_0}}\right)V(u_t)\\
	 &\leqc e^{C t}d(x,y)^{-p}V(u).
\end{aligned}
\end{equation}
\end{proof}

As we saw for the twisted Markov semi-group, boundedness in a Banach space is not enough to ensure that $P^{(2)}_t$ gives rise to a $C_0$-semigroup on that space. Indeed, we must define the space $\mathring{C}_{\hat V}$ obtained as the closure of the space of smooth cylinder functions
\[
	\mathring{C}^\infty_0(\Hbf\times \mathcal{D}^c) = \{\varphi \,| \, \varphi(u,x,y) = \phi(\Pi_{\mathcal{K}} u,x,y) \,,\, |\mathcal{K}| <\infty\,,\,  \phi \in C^\infty_0(\R^{|\mathcal{K}|})\}
\]
with respect to the norm $\|\cdot\|_{C_{\widehat{V}}}$. An analogous argument to the proof of Proposition \ref{prop:C0-property-twist} for the twisted Markov semi-group $\hat{P}_t^p$ (in fact a strictly simpler one since it does not involve derivatives) gives the $C_0$ semi-group property of $P^{(2)}_t$. We omit the proof for brevity.

\begin{proposition}\label{prop:C0-twopoint}
Let $\widehat{V} = \widehat{V}_{p,\beta,\eta}$, where $\eta\in(0,\eta^*)$, $p\in (0,p_0)$, and $\beta >1$ is taken large enough. Then the Markov semi-group $P^{(2)}_t$ extends to a $C_0$ semigroup on $\mathring{C}_{\widehat{V}}$. 
\end{proposition}

 Consequently Proposition \ref{prop:C0-twopoint} implies that there is a well-defined generator $\mathcal{L}_{(2)}$ for $P^{(2)}_t$ on $\mathring{C}_{\widehat{V}}$ with dense domain $\mathrm{Dom}(\mathcal{L}_{(2)}) \subseteq \mathring{C}_{\hat{V}}$. The key estimate of this section is the following approximate drift condition for $\mathcal{L}_{(2)}$.
\begin{lemma} \label{lem:approxhp}
For all $p\in(0,p_0)$, $\eta \in (0,\eta^*)$ and $\beta\geq 1$ taken large enough, $h_p$ belongs to $\mathrm{Dom}(\mathcal{L}_{(2)})$ on $\mathring{C}_{\widehat{V}_{p,\beta,\eta}}$ and there exists a constant $C^\prime$ such that 
\begin{align}
\mathcal{L}_{(2)} h_p \leq -\Lambda(p) h_p + C' V_{\beta + 1,\eta},
\end{align}
\end{lemma}

A priori, it is not clear that $h_p$ actually belongs to the domain of the generator $\mathcal{L}_{(2)}$ since it involves $\psi_p$, which belongs to the domain of the generator of $\hat{P}^p_t$, and therefore more readily belongs to the domain of the generator of the semi-group associated to the linearized motion
\[
	TP_t \hat{h}(u,x,w) := \E_{(u,x,w)}\hat{h}(u_t,x_t,D\phi^tw).
\]
Therefore in order to prove Lemma \ref{lem:approxhp}, we will need to approximate $P_t^{(2)}\hat{h}_p$ by 
\[
	T^* P_t\hat{h}_p(u,x,y) := TP_t \hat{h}_p(u,x,w(x,y)).
\]
To this end, we write
\begin{equation}\label{eq:P2-split}
\frac{P^{(2)}_t h_p - h_p}{t}  = \frac{T^*P_t \hat h_p  - h_p}{t} + \frac{P^{(2)}_t h_p - T^*P_t \hat h_p}{t},
\end{equation}
and show that each limit on the right-hand side of \eqref{eq:P2-split} converges separately. It is important to remark that this approximation is only effective $t \to 0$, indeed $T^*P_t\hat h_p$ is not even continuous on $\mathcal{D}^c$ due to the discontinuous nature of $w(x,y)$. 

For the first term in \eqref{eq:P2-split}:
\begin{lemma}\label{lem:linearized-approx-eig} For all $p\in (0,p_0)$, $\eta\in(0,\eta^*)$ and $\beta \geq 1$ large enough, the following limit holds in $C_{\hat{V}_{p,\beta,\eta}}$
\begin{align}
\lim_{t \to 0} \frac{T^*P_t \hat h_p - h_p}{t} = 
-\Lambda(p) h_p +  \mathcal{E}_p,  \label{eq:TPhpGen}
\end{align}
where $\mathcal{E}_p(u,x,y) = \hat{\mathcal{E}}_p(u,x,w(x,y))$ and (recall the definition of $H$ in Remark \ref{rmk:DefH}) 
\[
 \hat{\mathcal{E}}_p(u,x,w)= H(u,x,w/|w|)|w|^{1-p}\psi_p(u,x,w/|w|)\chi^\prime(|w|).
\]
\end{lemma}
\begin{proof}
To begin, note that $|w|^{-p}\psi_p$ is an eigenfunction of $TP_t$ with eigenvalue $e^{-\Lambda(p)t}$. Denote the linearized process $w_t^* = D_x\phi^t w$ and note that $w^*_t$ is a solution to $\partial_t w_t^* = Du_t(x_t) w^*_t$. Using this, we find
\[
	TP_t \hat h_p = e^{-\Lambda(p)t}h_p + \E |w_t|^{-p}\psi_p(u_t,x_t,w_t^*)(\chi(|w_t^*|) - \chi(|w|)).
\]
Noting that
\[
	\chi(|w_t^*|) - \chi(|w|) = \int_0^t |w_s^*| \,H(u_s,x_s,v_s)\chi^{\prime}(|w_s^*|)\ds,
\]
where $v_s = w_s^*/|w_s^*|$ is the projective process with initial data $v = w/|w|$, allows us to write
\begin{equation}\label{eq:remainder-form}
	\frac{TP_t \hat h_p - \hat h_p}{t} = \frac{e^{-\Lambda(p)t} - 1}{t}\hat h_p +  \hat{\mathcal{E}}_p + \hat R_t,
\end{equation}
where
\[
	\hat R_t:= \E\left(|w_t^*|^{-p}\psi_p(u_t,x_t,v_s)\frac{1}{t}\int_0^t|w_s^*|H(u_s,x_s,v_s)\,\chi^\prime(|w^*_s|) \ds - \mathcal{E}_p\right).
\] 
The fact that $\psi_{p}\in \mathring{C}_V$, where $V= V_{\beta_0,\eta}$, for some $\beta_0\geq 1$ means that we can find $\psi_p^{(n)} \in \mathring C^\infty_0$ where $\psi^{(n)}_p$ only depends on finitely many Fourier modes $\Pi_n u$ of $u$, such that $\psi^{(n)}_p \to \psi_p$ in $C_V$. Using the fact that
\[
 |w_t^*| = |w|\exp\left(\int_0^t H(u_s,x_s,v_s)\ds\right)
\]
and that $\Pi_n u_t$ is finite dimensional, a direct calculation yields for each $n$
\begin{equation}\label{eq:remainder-bound}
	|\hat R_t| \leqc |w|^{1-p}\exp\left(C_p\int_0^t\|u_s\|_{H^{r_0}}\right)\sup_{s\in(0,t)}V_{\beta_0+1,\eta}(u_s) \left(C_n\rho_t + \|\psi_p - \psi^{(n)}_p\|_{C_{V}}\right),
\end{equation}
where $C_n$ is a constant depending (badly) on $n$ and $D\psi^{(n)}_p$ and
\[
	\rho_t := \sup_{s\in(0,t)}\left(\|u_s - u\|_{H^{r_0}} + d_{\T^d}(x_s,x) + d_{P^{d-1}}(v_s,v)\right).
\]
A straight forward consequence the evolution equation for the projective process $(u_t,x_t,v_t)$ directly implies
\[
	\rho_t \leq t \sup_{s\in(0,t)} (1+\|u_s\|_{\Hbf}^2) + \sup_{s\in(0,t)}\|QW_s\|_{H^{r_0}}.
\]
By the BDG inequality, we can bound $H^{r_0}$ norm of the Wiener process using
\[
\E \sup_{s\in(0,t)}\|QW_s\|_{H^{r_0}}^2  \leqc_{Q} t.
\]
Therefore taking expectation of \eqref{eq:remainder-bound} applying Cauchy-Schwartz and the exponential estimates from Lemma \eqref{lem:TwistBd}, we conclude that for some $\beta_1\geq 1$ large enough,
\[
	\E |\hat R_t| \leqc_{Q} |w|^{1-p}V_{\beta_1,\eta}(u) (C_n t^{1/2} + \|\psi_p - \psi^{(n)}_p\|_{C_{V}}).
\]
Denote $R_t(u,x,y) = \hat R_t(u,x,w(x,y))$ and note that while $R_t$ is not necessarily continuous in $(x,y)$ due to the discontinuity in $w(x,y)$ away from the diagonal, the quantity $|w(x,y)| = d_{\T^d}(x,y)$ is continuous on $\mathcal{D}^c$ and therefore $\E | R_t|$ is bounded above by a continuous function. The corresponding estimate on $R_t$ implies after first sending $t\to 0$ and then $n\to \infty$ that
\[
	\lim_{t \to 0}\|\E | R_t|\|_{C_{\hat{V}_{p,\beta_1,\eta}}} = 0.
\]
This, coupled with the fact that $\frac{e^{-\Lambda(p)t} - 1}{t} \to - \Lambda(p)$ as $t \to 0$ and $h_p \in C_{\hat{V}_{p,\beta_1,\eta}}$, is sufficient to conclude the proof in light of equation \eqref{eq:remainder-form}.
\end{proof}

The second term in \eqref{eq:P2-split} involves controlling the error involved in approximating $P^{(2)}_t h_p$ by $T^*P_t\hat h_p$. As discussed in Section \ref{subsec:outlineConstructV}, this is one of the main difficulties in proving a valid drift condition and is the only reason we need $\mathring{C}^1_V$ estimates on $\psi_p$, i.e., so that we can differentiate $\psi_p$ with respect to the projective coordinate and bound it by $V(u)$. 
\begin{lemma}\label{lem:error-approx} 
For all $p\in(0,p_0)$, $\eta \in (0,\eta^*)$, and $\beta \geq 1$ large enough, the following limit holds in $C_{\hat{V}_{p,\beta,\eta}}$:
\begin{equation}\label{eq:linear-approx-lim}
\lim_{t \to 0} \frac{P^{(2)}_t h_p - T^*P_t \hat{h}_p}{t} = (\grad_y h_p- \grad_x) \cdot \Sigma, 
\end{equation}
where $\Sigma(u,x,y) = u(y) - u(x) - Du(x)w(x,y)$.
\end{lemma}

\begin{proof}
Define $w_t = w(x_t,y_t)$ and $w_t^* = D\phi^t w$, where $w = w(x,y)$, and note that
\[
\frac{P^{(2)}_th_p - T^*P_t \hat{h}_p}{t} = \frac{1}{t}\E (\hat h_p(u_t,x_t,w_t) - \hat{h}_p(u_t,x_t,w_t^*)) = \E \int_0^1 \nabla_w \hat h_p(u_t,x_t,w^\theta_t)\dee \theta \cdot \frac{w_t - w^*_t}{t},
\]
where $w^{\theta}_t = \theta w_t + (1-\theta)w_t^*$. To continue, for each $t\geq 0$, we define the events
\[
	A_t := \left\{t \sup_{s\in(0,t)}\|\nabla u_s\|_{\infty}\leq \frac{1}{100}\right\},
\quad 
	B_t := \left\{t \sup_{s \in(0,t)}\left(\|\nabla u_s \|_{\infty}(|w_s|+|w_s^*|)\right)\leq \frac{|w_t^*|}{2}\right\}.
\]
Note that both the events $A_t$, $B_t$ implicitly depend on the initial data $(u,x,y)$ and we have that
\[
\lim_{t\to \infty} \P(A_t\cap B_t) = 1.
\]

The event $A_t\subset \Omega$ is chosen so that on it $w_t$ has not moved far from its starting point $w$,
\[
	|w_t - w| \leq t \sup_{s\in(0,t)}\|\nabla u_s\|_{L^\infty} \leq \frac{1}{100},
\]
and therefore on $A_t$ we can write 
\[
	w_t - w_t^*  = \int_0^t u_s(y_s) - u_s(x_s) - Du_s(x_s)w_s^*\ds.
\]
It follows that 
\begin{equation}
\frac{w_t - w^*_t}{t}\1_{A_t} = \Sigma\1_{A_t}  + R_t\
\end{equation}
where
\[
	R_t := \frac{\1_{A_t}}{t}\int_0^t (u_s(y_s) - u(y))  - (u_s(x_s) - u(x)) - (D u_s(x_s) w^*_s - Du(x)w) \dee \tau.
\]
Directly computing all the differences it is straightforward to see that $|R_t|$ can be bounded by
\begin{equation}\label{eq:R_t-bound}
	|R_t| \leqc \sup_{s\in(0,t)} \left(|u(y_s) - u(x_s) - u(y) + u(x)| + \|u_s - u\|_{H^{r_0}}|w| + \|u_s\|_{H^{r_0}}|w_s^* -w|\right)
\end{equation}
Using the evolution equation for $u_t(x_t)$, and denoting $F^u(u) = - B(u,u) - Au$ we find
\begin{equation}
\begin{aligned}
	&u_t(y_t) - u_t(x_t) = u(y) - u(x) +  \int_0^t F^u(u_s)(y_s) - F^u(u_s)(x_s)\ds\\
	&\hspace{1in} + \int_0^tu_s(y_s)\cdot \nabla u_s(y_s) - u_s(x_s)\cdot \nabla u_s(x_s)\ds + \int_0^t (Q(y_s) - Q(x_s))\dee W_s.
\end{aligned}
\end{equation}
Using this, similar to the proof of Lemma \ref{lem:linearized-approx-eig}, an application of the BDG inequality gives
\begin{equation}\label{eq:rho_t-bound}
	\sqrt{\E (R_t)^2}  \leqc |w|t^{1/2}\,\E \exp\left(C\int_0^t \|u\|_{H^{r_0}}\right)\sup_{s \in [0,t]} (1+\|u_s\|_{\Hbf}^2).
\end{equation}

The event $B_t \subset \Omega$ is chosen such that on $B_t$, $w_t^*$ satisfies the following geometric constraint:
\[
	|w_t -w_t^*|\leq |w_t -w| + |w_t^* - w| \leq t \sup_{s \in (0,t)}\left(\|\nabla u_s\|_{\infty}(|w_s|+|w_s^*|) \right) \leq \frac{1}{2}|w^*_t|
\]
and therefore on $B_t$ we have a lower bound for $|w^\theta_t|$
\[
	|w^{\theta}_t| \geq |w_t^*| - |w_t - w_t^*| \geq \frac{1}{2}|w_t^*|.
\]
Using the fact that $|\nabla_w \hat{h}_p| \leq |w|^{-p-1}\|\psi_p\|_{C^1_V} V_{\beta_0,\eta}(u)$ for some $\beta_0\geq 1$
\begin{equation}\label{eq:A-restriction-est}
\begin{aligned}
	\1_{B_t}\left|\int_0^1 \grad_w\hat h_p(u_t,x_t,w^{\theta}_t) d\theta \right| &\leqc\|\psi_p\|_{C^1_V} V_{\beta_0,\eta}(u_t)\, \1_{B_t}\int_0^1 |w^\theta_t|^{-p-1}\dee\theta\\
	&\leqc \|\psi_p\|_{C^1_V} V_{\beta_0,\eta}(u_t) |w^*_t|^{-p-1}\\
	&\leqc \|\psi_p\|_{C^1_V} \exp\left(C_p\int_0^t \|u_s\|_{H^{r_0}}\ds\right)|w|^{-p-1} V_{\beta_0,\eta}(u_t) .
\end{aligned}
\end{equation}
Combining \eqref{eq:R_t-bound} with equation \eqref{eq:A-restriction-est}, using Cauchy-Schwartz, estimate \eqref{eq:rho_t-bound} and Lemma \ref{lem:TwistBd} gives
\begin{equation}\label{eq:wtheta-remainder-bound}
\E \1_{A_t\cap B_t}\left|\int_0^1 \grad_w \hat{h}_p(u_t,x_t,w^{\theta}_t) d\theta \cdot R_t\right| \leqc t^{1/2}\|\psi_p\|_{C^1_V} |w|^{-p} \,V_{\beta_1,\eta}(u)
\end{equation}
for some $\beta_1>\beta_0$ and all $\eta \in (0,\eta^*)$. Again similarly to the proof of Lemma \ref{lem:linearized-approx-eig}, by density of $\mathring{\cC}^\infty_{0}(\Hbf\times \T^d\times P\T^d)$ in $\mathring{C}^1_V$ we can take an approximating sequence of $\psi^{(n)}_p$ converging to $\psi_p$ in $C^1_V$ such that $\psi_p^{(n)}$ only depends on finitely many Fourier modes $\Pi_n u$. Using the fact that $\Pi_n u_t$ is finite dimensional we conclude that
\begin{equation}
\begin{aligned}
	&\1_{A_{t}\cap B_t}\left|\int_0^{1}\left(\nabla_w \hat h_p(u_t,x_t, w^{\theta}_t)\dee\theta - (\nabla_y - \nabla_x) h_p\right)\cdot \frac{w_t - w_t^*}{t}\right|\\
	 &\hspace{.5in}\leq |w|^{-p}\exp\left(C\int_0^t\|u_s\|_{H^{r_0}}\right)\sup_{s\in(0,t)}V_{\beta_1,\eta}(u_s) \left(C_n\tilde{\rho}_t + \|D_v\psi_p - D_v\psi^{(n)}_p\|_{C_{V}}\right),
\end{aligned}
\end{equation}
where
\[
	\tilde{\rho}_t = \sup_{s\in(0,t)}\left(\|u_s - u\|_{H_{r_0}} + d_{\T^d}(x_s,x) + \1_{A_s}|w_s - w| + \1_{A_s}|w_s^*- w|\right).
\]
Using the evolution equation for $(u_t,x_t,w_t)$ and the BDG inequality to deal with $\tilde{\rho}_t$, an analogous argument to the one in the proof of Lemma \ref{lem:linearized-approx-eig} implies that
\begin{equation}
	\lim_{t\to 0}\E\1_{A_{t}\cap B_t}\left|\int_0^{1}\left(\nabla_w \hat h_p(u_t,x_t, w^{\theta}_t)\dee\theta - (\nabla_y - \nabla_x) h_p\right)\cdot \frac{w_t - w_t^*}{t}\right| = 0
\end{equation}
where the limit holds in $C_{\hat{V}_{p,\beta_2,\eta}}$ for some $\beta_2$ large enough. Combining this with \eqref{eq:wtheta-remainder-bound} and the fact that $R_t = \1_{A_t}\frac{w_t - w_t^*}{t} - \1_{A_t}\Sigma$ yields
\[
	\lim_{t\to 0}\E \1_{A_t\cap B_t} \int_0^1 \nabla_w \hat h_p (u_t,x_t,w^\theta_t)\dee \theta \cdot \frac{w_t - w^*_t}{t} = (\grad_y - \grad_x) h_p  \cdot \Sigma, 
\]
where the limit holds in $C_{\hat{V}_{p,\beta_3,\eta}}$ for $\beta_3$ large enough. 

On the complement $A^c_t\cup B_t^c$, we use the fact that for each $\delta >0$
\begin{equation}
\begin{aligned}
	 \1_{A^c_t\cup B^c_t}&\leqc  t^{1+\delta}\left(\sup_{s \in(0,t)}\|\nabla u_s\|^{1+\delta}_{L^\infty} + |w_t^*|^{-1-\delta}\,\sup_{s \in(0,t)}(\|\nabla u_s\|^{1+\delta}_{L^\infty}(|w_s|+|w^*_s|))^{1+\delta}\right) \\
	 &\leqc t^{1+\delta} \exp\left(2(1+\delta)\int_0^t\|u_s\|_{H^{r_0}}\ds\right)\sup_{s\in(0,t)}\|u_s\|_{H^{r_0}}^{1+\delta},
\end{aligned}
\end{equation}
and therefore
\begin{equation}
\begin{aligned}
	\frac{1}{t} \1_{A^c_t\cup B^c_t}\left( \hat h_p(u_t,x_t,w_t) -  \hat h_p(u_t,x_t,w^*_t) \right) &\leq \frac{1}{t} \1_{A^c_t\cup B^c_t}(|w_t|^{-p} + |w^*_t|^{-p}) V_{\beta_0,\eta}(u_t)\\
	&\hspace{-.5in}\leqc t^{\delta}|w|^{-p}\exp\left(C_{p,\delta}\int_0^t \|u_s\|_{H^{r_0}}\ds\right)\sup_{s\in(0,t)} V_{\beta_0+1,\eta}(u_s)\\
\end{aligned}
\end{equation}
which implies by the exponential estimates of Lemma \ref{lem:TwistBd} that the following limit holds in $C_{\hat{V}_{p,\beta_0+1,\eta}}$
\[
\lim_{t\to 0}\frac{1}{t} \EE \1_{A^c_t\cup B_t^c}\left( \hat h_p(u_t,x_t,w_t) -  \hat h_p(u_t,x_t, w^*_t) \right) = 0.
\]

Putting all the limits together completes the proof.
\end{proof}

Lemma \ref{lem:approxhp} is now a simple consequence of the previous two lemmas.

\begin{proof}[\textbf{Proof of Lemma \ref{lem:approxhp}}]

Applying Lemmas \ref{lem:linearized-approx-eig}, and \ref{lem:error-approx} to the splitting \eqref{eq:P2-split} we can deduce
\begin{align*}
\mathcal{L}_{(2)} h_p = -\Lambda(p)h_p + \mathcal{E}_p + (\nabla_y - \nabla_x)h_p \cdot \Sigma.
\end{align*}
Note that
\[
	\mathcal{E}_p \leq |w|^{1-p}\left\|\nabla u\right\|_{L^\infty} \|\psi_p\|_{C_V} V_{\beta,\eta}(u)
\]
Similarly, since $|(\nabla_y - \nabla_x)h_p| \leqc |d(x,y)|^{-p-1}\|\psi_p\|_{C^1_V} $ and by Taylor's theorem $|\Sigma| \leq |d(x,y)|^2 \|\nabla^2 u\|_{L^\infty}$, we deduce that since $p < 1$,
\[
|\mathcal{E}_p + (\nabla_y - \nabla_x )h_p \cdot \Sigma| \lesssim \abs{d(x,y)}^{1-p} \norm{u}_{W^{2,\infty}}\|\psi_p\|_{C^1_V} V_{\beta,\eta}(u) \lesssim V_{\beta+1,\eta}(u). 
\]

\end{proof}

We are now ready to complete the proof of Proposition \ref{prop:2ptDrift}. 
\begin{proof}[\textbf{Proof of Proposition \ref{prop:2ptDrift}}]
Let $\mathcal{L}$ denote the formal generator of the Navier-Stokes equations defined by equation \eqref{eq:formal-gen} of Section \ref{sec:ExpProj}. First, observe that for any $\beta > 0, 
\eta \in (0, \eta^*)$ we have for $V = V_{\beta, \eta}$ and all $u \in \Hbf^{\sigma +d-1}$ that
\[
\mathcal{L} V(u)  =\left( \mathcal{L} \log V(u) + \sum_{m\in \mathbb{K}} \abs{q_m}^2 \abs{D_{u}\log V(u)e_m}^2\right) V(u) .
\]
and therefore using the fact that
\[
\sum \abs{q_k}^2 \abs{D_{u}\log V(u)e_m}^2 \leq  8\beta^2 + 8\eta^2 \mathcal{Q} \norm{\Delta u}^2_{L^2},
\]
and applying inequality \eqref{ineq:GenLogV} of Lemma \ref{lem:GenV} we deduce that
\[
	\mathcal{L}V(u) \leq \left(- \eta\left(\nu - 8 \eta \mathcal{Q}\right)\norm{\Delta u}_{L^2}^2 - \nu\beta\frac{\|\nabla u\|_{H^\sigma}^2}{1+\|u\|_{H^\sigma}^2} + C\right)V(u).
\]
Note that $\eta \leq \eta^*$ ensures that $\frac{\eta}{2}\left(\nu - 16 \eta \mathcal{Q}\right)$ is positive. Applying Lemma \ref{lem:SobLogTrick2} implies that $\forall \delta > 0$, $\exists C_\delta> 0$ such that
\[
	\mathcal{L}V \leq - \delta \log(1+\|u\|_{H^\sigma}^2) V.
\]
We treat the right-hand side above by dividing into regions where $\|u\|_{H^\sigma}\leq 1$ and $\|u\|_{H^\sigma} > 1$, in the former case everything is bounded by a constant, in the latter case we can bound the logarithmic factor below by $\log(2)$. This gives for all $0 < \kappa <\log(2)\delta$, and some constant $C>0$
\[
	- \delta \log(1+\|u\|_{H^\sigma}^2) V \leq - \kappa V  + C.
\]
As $\delta$ was arbitrary, it follows that for each $\kappa >0$, we have the bound (for a suitable $C_\kappa > 0$) 
\begin{equation}\label{eq:generator-ineq}
	\mathcal{L}V(u) \leq - \kappa V(u) + C_\kappa.
\end{equation}
Note that inequality \eqref{eq:generator-ineq} is an infinitesimal version of a drift condition for Navier-Stokes. However, $V$ does not belong to the domain of $\mathcal{L}_{(2)}$ in $\mathring{C}_{\hat{V}}$, and so we must proceed with more care to deduce a corresponding drift condition on the semi-group $P_t$.
Define $\tau_n =\inf \{ t>0 : \|u_t\|_{\Hbf} > n\}$  (note $\tau_n \to \infty$ as $n\to \infty$ with probability one). 
Applying It\^{o}`s formula (Theorem 7.7.5 \cite{KS}) to $e^{\Lambda(p) t}V(u_t)$ implies that
\begin{equation}
\begin{aligned}
	\E_u e^{\Lambda(p) t\wedge \tau_n}V(u_{t\wedge \tau_n}) - V(u) &= \E_u\int_0^{t\wedge \tau_n} e^{\Lambda(p) s}\left(\Lambda(p) V(u_s) + \mathcal{L}V(u_s)\right)\ds,\\
	&\leq \E_u\int_0^{t\wedge \tau_n} e^{\Lambda(p) s}\left((\Lambda(p)- \kappa) V(u_s) + C\right)\ds.
\end{aligned}
\end{equation}
As we know that $\E \sup_{s\in(0,t)}V(u_s) <\infty$ (e.g. from Lemma \ref{lem:TwistBd}), we can use dominated convergence to pass the $n\to \infty$ limit on both sides of the above inequality to deduce
\begin{equation}\label{eq:EEV}
	e^{\Lambda(p) t}P_t V - V \leq \int_0^t e^{\Lambda(p)s}P_s\left((\Lambda(p)- \kappa) V(u_s) + C\right)\ds,
\end{equation}
where $P_t$ denotes the Markov semi-group for the Navier-Stokes equations. 

Recall that $\Vc$ takes the form
\[
	\Vc = h_p + V_{\beta + 1, \eta}
\]
where $h_p \in C_{\hat V_{p, \beta, \eta}}$. Naturally, using the $C_0$ semi-group property of $P^{(2)}_t$ on functions in $\mathring{C}_{\hat{V}_{p, \beta, \eta}}$ and Lemma \ref{lem:approxhp} we also find that
	\begin{equation}\label{ineq:EEhp}
	\begin{aligned}
e^{\Lambda(p) t} P^{(2)}_t h_p - h_p  &= \int_0^t e^{\Lambda(p) s}P^{(2)}_s(\Lambda(p) h_p + \mathcal{L}_{(2)}h_p)\ds \leq \int_0^t e^{\Lambda(p) s}C' P_s V_{\beta + 1, \eta}\ds.
\end{aligned}
\end{equation}
Using the fact that
\[
	P^{(2)}_t \Vc = P^{(2)}_t h_p + P_t V_{\beta + 1, \eta},
\]
we complete the proof by adding \eqref{ineq:EEhp} and $\eqref{eq:EEV}$ and taking $\kappa$ large enough so that $\kappa -\Lambda(p) \geq C'$ to conclude that there is a constant $K$ such that
\begin{equation}
\begin{aligned}
e^{\Lambda(p) t} P^{(2)}_t\Vc - \Vc &\leq C\int_0^t e^{\Lambda(p) s} \ds \leq Ke^{\Lambda(p) t}.
\end{aligned}
\end{equation}
\end{proof}

\section{Correlation decay: proof of Theorem \ref{thm:-sdecay}} \label{sec:finishUp}

\begin{proof}[Proof of Theorem \ref{thm:2-pt-decay}]
Geometric ergodicity for the two-point process $(u_t, x_t, y_t)$ as in Theorem \ref{thm:2-pt-decay}
follows from the general framework given in Theorem \ref{thm:GM} and Conditions \ref{defn:SF}, \ref{defn:TopIrr}, \ref{def:UnifBd} and \ref{defn:drift} listed there. A proof of the strong Feller property in Condition \ref{defn:SF} 
was sketched in Section \ref{sec:SF}, while topological irreducibility as in Condition \ref{defn:TopIrr} was 
obtained in Section \ref{sec:Irr2pt}. Our desired Lypaunov function $\Vc$ for the two-point process
was identified in the previous section the proof of Proposition \ref{prop:2ptDrift}, validating Condition
\ref{defn:drift}, while Lemma \ref{lem:hypothA3} affirms Condition \ref{def:UnifBd} holds for this choice
of $\Vc$.
\end{proof}

It remains to complete the proof of Theorem \ref{thm:-sdecay}, which occupies the remainder of the paper.

Our goal is to prove the following.  Let $s > 0$ and $p \geq 1$ be fixed. Then, there exists
a deterministic constant $\hat{\gamma} = \hat{\gamma}(s, p)$ and a measurable function
$D: \Omega \times \Hbf \to [1,\infty)$, with the following property: for $\P \times \mu$-almost all $(\omega, u) \in \Omega \times \Hbf$ and for all
mean-zero $f, g \in H^s$, we have
\begin{align}\label{eq:correlationEstTarget}
\left| \int f (x) g(\phi^t_{\omega, u}(x)) \, \dx \right| \leq D(\omega, u) e^{- \hat{\gamma} t} \| f \|_{H^s} \| g \|_{H^s} \, .
\end{align}

The proof proceeds in several steps: (1) Establish correlation bounds in discrete time; and (2) 
extending these correlation bounds to cover all real times $t \geq 0$. At the end, 
we will (3) estimate $p$-th moments for the `random constant' $D$ in terms of $u$.

\medskip

\noindent {\bf Notation. } 
Let $\Vc$ be the Lyapunov function appearing in Theorem \ref{thm:2-pt-decay}
and note (see Sections \ref{subsec:outlineConstructV} and \ref{sec:2ptGeoErg}) 
that $\Vc(u,x,y) \leq (V(u))^2 \hat W(x,y)$, where $V = V_{ \beta,  \eta}$ for some $ \beta,  \eta > 0$ (see \eqref{def:V} in Section \ref{subsec:outlineConstructV} for notation) and $\hat W \in L^1(\Leb)$.

Hereafter, $s > 0, p \geq 1$ are fixed. Without loss of generality, we will assume $s \in (0,1)$. 
Fix as well $\hat{\gamma} \in (0, \frac{\hat{\alpha}}{2})$, where
$\hat{alpha} > 0$ is the mixing rate as in Proposition \ref{prop:2ptDrift} for the two-point process; further constraints will be placed on $\hat{\gamma}$ as we proceed. Throughout, $ \hat{\gamma}' \in (\hat{\gamma}, \frac{\hat{\alpha}}{2})$ will be a dummy parameter, also chosen appropriately. Given generic
$\omega, u$, we write $\phi^t = \phi^t_{\omega, u}$ for short. Let $f, g$ be fixed, smooth, mean-zero observables. 
Finally, for $k = (k_1, \cdots, k_d) \in \Z^d$, 
we write $|k| = |k|_\infty = \max\{ |k_i|, 1 \leq i \leq d\}$.

\medskip

\subsection{Correlation bounds in discrete time}

Let $\{ \hat e_k\}_{k \in \Z^d_0}$ be the real Fourier basis of mean-zero functions on $\T^d$ 
(see, for instance, the notation in \cite{HM06}).

Fix $u \in \Hbf$, regarded as a fixed initial condition for the velocity field process $(u_t)$. 
By a variant of the Borel Cantelli argument given in Section 
\ref{subsec:sufficesTwoPtMotion}, we have that for each $(k, k') \in \Z^d_0 \times \Z^d_0$, 
the random variable
\[
N_{k, k'}(u) = \max\left\{ n \geq 0 : \left| \int \hat e_k (x) \hat e_{k'}(\phi^n(x)) \dx \right| > e^{-\hat{\gamma}' n}  V(u) \right\}
\]
is finite with probability 1, where the tail estimate
\begin{equation} \label{eq:N-tail-est}
\P \{ N_{k, k'}(u) > n\} \lesssim
e^{-(\hat{\alpha} - 2\hat{\gamma}')n}  
\end{equation}
holds uniformly in $u, k, k'$. In particular, it holds
that $|\int e_k (x) e_{k'}(\phi^n(x)) \dx| \leq e^{\hat{\gamma}' (N_{k, k'}(u) - n)}  V(u)$
for all $n \geq 0$. Hereafter, let us suppress the `$u$' (which remains fixed) in $N_{k, k'}$.

Expand
\[
	f = \sum_{k\in \Z^d_0}f_k e_k,\quad g = \sum_{k\in \Z^d_0} g_k e_k \, ,
\]
so that
\begin{align}
\left| \int f(x)g(\phi^n(x))\dx\right|
\leq  V(u) e^{- \hat{\gamma}' n} \sum_{k, k' \in \Z^d_0} |f_k| |g_{k'}|   e^{ \hat{\gamma}' N_{k, k'}} \label{eq:preHSestimate}
\end{align}
To form a comparison of the RHS of \eqref{eq:preHSestimate} with a Sobolev norm, observe that due to the uniformity of the estimate \eqref{eq:N-tail-est} in $k, k'$, we have
\begin{align}\label{eq:Nkkprimeboundtail}
	\P\left\{e^{\hat{\gamma}' N_{k,k^\prime}} > |k||k^\prime| \right\} \lesssim (|k||k^\prime|)^{-\frac{\hat{\alpha} - 2 \hat{\gamma}'}{\hat{\gamma}'}} \, ,
\end{align}
again uniformly in $k, k'$.
The right-hand side is summable over $\Z^d_0\times \Z^d_0$ when $\hat{\gamma}$ is sufficiently small. 
Applying Borel-Cantelli to the countable collection of events $\mathfrak S_{k, k'} = \{ e^{\hat{\gamma}' N_{k, k'}} > |k| |k'|\}$, 
we conclude that the measurable function\footnote{Given real numbers $a, b > 0$, we write $a \vee b = \max \{ a, b \}$.}
\[
K = \max\{  |k| \vee |k'| : k, k' \in \Z^d_0 \text{ and } e^{\hat{\gamma}' N_{k,k^\prime}} > |k||k^\prime| \}
\]
satisfies the tail estimate $\P \{ K > \ell\} \lesssim \ell^{2 d - \frac{\hat{\alpha} - 2 \hat{\gamma}'}{\hat{\gamma}'}}$ uniformly in $u$. Observe that by definition, we have
$	e^{\hat{\gamma}' N_{k,k^\prime}} \leq |k||k^\prime| $
for all $(k, k') \in \Z^d_0 \times \Z^d_0$ with $|k| \vee |k'| > K$. 

Define \[\hat D = \max_{|\hat k|,|\hat k^\prime| \leq K} e^{\hat{\gamma}' N_{\hat k,\hat k^\prime}} \, . \]
 Plugging the simple bound $e^{\hat{\gamma}' N_{k,k^\prime}} \leq \hat D (1 + |k||k'|)$
into \eqref{eq:preHSestimate}, we conclude that for all $n \geq 0$,
\begin{equation}
\begin{aligned}
	\left|\int_{\T^d} f(x) g(\phi^n(x))\dx  \right| & \leq \hat D  V(u) \left( \sum_k |k|  |f_k| \right) \left( \sum_k |k|  |g_k| \right)  e^{-\hat{\gamma}' n}  \\
	& \lesssim \hat D  V(u) \| f \|_{H^{\frac{d}{2} + 2}} \| g \|_{H^{\frac{d}{2} + 2}} e^{- \hat{\gamma}' n}
\end{aligned}
\end{equation}
To control the RHS in terms of $H^s$ norms, we can compensate for reduced regularity of $f, g$ by reducing the 
exponential decay rate. For this, we will use the following standard approximation Lemma
(c.f. Lemma 4.2 in \cite{CZDT18}, where similar ideas are used to control correlation 
decay in terms of an $H^1$ norm).

\begin{lemma}
Let $0 < s < s'$ and let $h \in H^s$ have zero mean. Then, for any $\epsilon > 0$ there exists a mean-zero 
$h_\epsilon \in H^{s'}$
such that (i) $\| h_\epsilon \|_{L^2} \lesssim \| h \|_{L^2}$, 
(ii) $\| h_\epsilon - h\|_{L^2} \lesssim \epsilon \| h \|_{H^s}$, and 
 (iii) $\| h_\epsilon\|_{H^{s'}} \lesssim_{s, s'} \epsilon^{- \frac{s' - s}{s} } \| h\|_{H^s}$.
\end{lemma}

We apply the above with $s' = \frac{d}{2} + 2$ and for the value of $s$ specified at the beginning of the section. 
Fixing $\epsilon > 0$, we estimate
\begin{align*}
\left| \int f(x)  g(\phi^n(x)) \dx \right|
& \leq \left| \int f_\epsilon(x)  g_\epsilon(\phi^n(x)) \dx \right| + \| g\|_{L^2} \| f - f_\epsilon\|_{L^2} + \| f_\epsilon \|_{L^2} \| g - g_\epsilon\|_{L^2} \\
& \lesssim \hat D  V \left(  e^{- \hat{\gamma}' n} \epsilon^{- \frac{d + 4 - 2 s}{2 s}} + \epsilon \right) \| f \|_{H^s} \| g \|_{H^s}. 
\end{align*}
Optimizing in $\epsilon$ on the RHS yields $\epsilon= e^{- \frac{2 s \hat{\gamma}'}{d + 4} n}$, resulting in the
estimate
\begin{align}\label{eq:discTimeMixingEstimate}
\left| \int f(x)  g(\phi^n(x)) \dx \right| \leq \hat D  V e^{- \hat{\gamma}'' n} \| f \|_{H^s} \| g \|_{H^s},
\end{align}
where $\hat{\gamma}'' := \frac{2 s \hat{\gamma}'}{d + 4}$ (having absorbed an $s$-dependent constant into $\hat D$).

\subsection{Correlation bounds in continuous time}

To estimate \eqref{eq:correlationEstTarget} at continuous times $t \in \R_{\geq 0}$, let $t = n + \hat t, \hat t \in [0,1), n \in \Z_{\geq 0}$.  Applying \eqref{eq:discTimeMixingEstimate}, we have
\[
\left|  \int f (x) g(\phi^t_{\omega, u}(x)) \dx \right| 
\leq 
e^{\hat{\gamma}''} \hat D(\omega, u)  V(u)  \| f  \|_{H^s} \| g \circ \phi^{\hat t}_{\theta^n \omega, u_n} \|_{H^s}  e^{- \hat{\gamma}'' t} \, .
\]

We will estimate the factor $\| g \circ \phi^{\hat t}_{\theta^n \omega, u_n} \|_{H^s}$ by
\begin{align}\label{eq:boundContTime}
\sup_{\hat t \in [0,1)} \| g \circ \phi^{\hat t}_{\omega', u'}\|_{H^s} \leq \Gamma(\omega', u')  \| g \|_{H^s} 
\end{align}
where $\Gamma : \Omega \times \Hbf \to [1,\infty)$ is defined by (for some suitable constant $C> 0$), 
\[
\Gamma(\omega', u') = \exp \left( C \int_0^{1} \| \nabla u_t'\|_{L^\infty} \dt \right) \, .
\]
To verify \eqref{eq:boundContTime}, we use the characterization
$\| f \|_{H^s} \approx \| f \|_{L^2} + \left( \int \int \frac{|f(x) - f(y)|^2}{|x - y|^{2 s + d}} \dx \dy \right)^{1/2}$
for the $H^s$ norm, $s \in (0,1)$, to estimate 
\[
	\|g\circ \phi^{\hat{t}}_{\omega',u'}\|_{H^s} \lesssim \| g \|_{L^2} + \left( \int \int \frac{|g(x_{\hat t}) - g(y_{\hat t})|^2}{|x - y|^{2 s + d} } \dee x \dee y \right)^{1/2} \lesssim \| g \|_{H^s}  \| D \phi^{\hat t}_{\omega', u'} \|_{L^\infty}^{s + d/2} \, .
\]
The bound \eqref{eq:boundContTime} now follows from Gr\"onwall's Inequality.

By a combination of Lemma \ref{lem:TwistBd} and Sobolev embedding, 
it follows that $\Gamma$ has $\P$-moments of all orders (c.f. \eqref{eq:momentEstimatePhi} below). 
By Chebyshev's inequality, for $\epsilon > 0$ we have
\footnote{Below, $\theta^t : \Omega \to \Omega, t \geq 0$ denotes the Wiener shift on 
canonical space-- see, e.g., the notation in \cite{BBPS18}.}
\[
\P  \{   \Gamma(\theta^n \omega, u_n) \geq e^{\epsilon n}\}
\leq e^{- q \epsilon n} \E  \Gamma^q \, , 
\]
which is is summable in $n$ for all $q \geq 1$. By Borel Cantelli, the function 
\[
M = M(\omega, u) = \max\{ n : \Gamma(\theta^n \omega, u_n) \geq e^{\epsilon n}\} 
\]
is finite $\P$-a.s., and satisfies $\P \{ M > n\} \lesssim e^{- q \epsilon n} \E  \Gamma^q $ for all $q \geq 1$.
In particular,
\[
\Gamma(\theta^n \omega, u_n) \leq e^{\epsilon (n + M)} \, .
\]
In total, we conclude
\[
\left| \int f(x) g(\phi^t_{\omega, u}(x)) \dx\right| \leq \underbrace{e^{\hat{\gamma}'' } \hat D  V e^{\epsilon M}}_{=: D(\omega, u)} \| f \|_{H^s} \| g \|_{H^s} e^{- (\hat{\gamma}'' - \epsilon) t}
\]
Finally, we set $\epsilon = \frac{1}{100} \hat{\gamma}''$ and $\hat{\gamma} = \hat{\gamma}'' - \epsilon$.
This completes the proof of the continuous-time correlation estimate as in item (2)
for any sufficiently small value of $\hat{\gamma}$.

\subsection{Estimating moments of $D(\omega, u)$}

With $u$ fixed, let us now estimate moments $\E D(\omega, u)^p$ by breaking up $D \approx \hat D  V e^{\epsilon M}$ as above. 

Let us first estimate the $\P$-moments of $e^{\epsilon M}$. We claim that for all $p > 0$, we have
\begin{align}\label{eq:improveMomentFINAL}
\E e^{p \epsilon M} \lesssim_{p, c} V^p(u) \, .
\end{align}
Fixing $q = p + 1$, we have
\[
\P \{ M(\omega, u) > n\} \lesssim \sum_{m = n+1}^\infty e^{- \epsilon q n} \E \Gamma^q(\theta^m \omega, u_m) \, .
\]
By Lemma \ref{lem:TwistBd}, we have that $\E (\Gamma^q(\theta^m \omega, u_m) | \mathscr F_m) \lesssim_{p} V(u_m)^{p}$ for all $n$.
Therefore, 
\begin{align}\label{eq:momentEstimatePhi}
\E \Gamma^q(\theta^m \omega, u_m) \lesssim \E V^p(u_m) \lesssim \int  V^{p} \dee \mu + e^{- \hat \alpha_0 m}  V^p(u)
\lesssim V^p(u)
\end{align}
by geometric ergodicity of the $(u_t)$ process (Proposition \ref{prop:CVspecGapProj}), where $\hat \alpha_0 > 0$ is a constant.

Collecting, we estimate
\[
\E  e^{\epsilon p M}  \lesssim \sum_{n = 1}^\infty e^{\epsilon p n} \cdot \P (M = n) 
\lesssim \sum_n e^{\epsilon (p - q) n} V^p(u)   \lesssim  V^p(u) \, .
\]
We conclude that $e^{\epsilon M}$ has $\P$-moments of all orders, and that for a fixed moment $p$ we have that the $p$-th moment w.r.t. $\P$ is controlled as in \eqref{eq:improveMomentFINAL}.

It remains to estimate $p$-th moments of $\hat D$. 
Note that although $\hat D$ depends nontrivially on the initial velocity $u$, all of our tails estimates for $\hat D$
are uniform in $u$. In particular, all of the following estimates are uniform in $u$.
\begin{align*}
\E  \hat D^p \,  & = 
\sum_{K_0 = 1}^\infty \E  {\bf 1}_{K = K_0} \max_{|k|, |k'| \leq K_0}  e^{\hat{\gamma}' p N_{k, k'}}  \\
& \leq \sum_{K_0 = 1}^\infty (\P \{ K = K_0\})^{1/2} \cdot \left\| \max_{|k|, |k'| \leq K_0}  e^{\hat{\gamma}' p N_{k, k'}}\right\|_{L^2(\Omega)} \\
& \lesssim \sum_{K_0 = 1}^\infty K_0^{\frac{d}{2} - \frac{ \hat{\alpha} - 2 \hat{\gamma}'}{2 \hat{\gamma}'}} \sum_{|k|, |k'| \leq K_0} 
\left\| e^{\hat{\gamma}' p N_{k, k'} } \right\|_{L^2(\Omega )} \, .
\end{align*}
The $\| e^{\hat{\gamma}' p N_{k, k'}}\|_{L^2(\Omega)}$ terms on the RHS are all uniformly bounded (independent 
of $k, k'$ and $u$) by \eqref{eq:Nkkprimeboundtail} if $\hat{\gamma}'$ is sufficiently small. Since
there are $\approx K_0^{2d}$ such terms, we obtain
\[
\E  \hat D^p  \lesssim \sum_{K_0 = 1}^\infty K_0^{\frac{5 d}{2} - \frac{\hat{\alpha} - 2 \hat{\gamma}'}{2 \hat{\gamma}'}} \, ,
\]
which provides a finite moment estimate when $\hat{\gamma}'$ is sufficiently small.

\bibliographystyle{abbrv}
\bibliography{bibliography}

\begin{bibdiv}
\begin{biblist}

\bib{aimino2015annealed}{article}{
      author={Aimino, Romain},
      author={Nicol, Matthew},
      author={Vaienti, Sandro},
       title={Annealed and quenched limit theorems for random expanding
  dynamical systems},
        date={2015},
     journal={Probability Theory and Related Fields},
      volume={162},
      number={1-2},
       pages={233\ndash 274},
}

\bib{GGM19}{article}{
      author={Alberti, Giovanni},
      author={Crippa, Gianluca},
      author={Mazzucato, Anna},
       title={Exponential self-similar mixing by incompressible flows},
        date={2019},
     journal={Journal of the American Mathematical Society},
      volume={32},
      number={2},
       pages={445\ndash 490},
}

\bib{arendt1986one}{book}{
      author={Arendt, Wolfgang},
      author={Grabosch, Annette},
      author={Greiner, G{\"u}nther},
      author={Moustakas, Ulrich},
      author={Nagel, Rainer},
      author={Schlotterbeck, Ulf},
      author={Groh, Ulrich},
      author={Lotz, Heinrich~P},
      author={Neubrander, Frank},
       title={One-parameter semigroups of positive operators},
   publisher={Springer},
        date={1986},
      volume={1184},
}

\bib{arnold1984formula}{article}{
      author={Arnold, Ludwig},
       title={A formula connecting sample and moment stability of linear
  stochastic systems},
        date={1984},
     journal={SIAM Journal on Applied Mathematics},
      volume={44},
      number={4},
       pages={793\ndash 802},
}

\bib{arnold1986lyapunov}{incollection}{
      author={Arnold, Ludwig},
      author={Kliemann, W},
      author={Oeljeklaus, E},
       title={Lyapunov exponents of linear stochastic systems},
        date={1986},
   booktitle={Lyapunov exponents},
   publisher={Springer},
       pages={85\ndash 125},
}

\bib{arnold1987large}{incollection}{
      author={Arnold, Ludwig},
      author={Kliemann, Wolfgang},
       title={Large deviations of linear stochastic differential equations},
        date={1987},
   booktitle={Stochastic differential systems},
   publisher={Springer},
       pages={115\ndash 151},
}

\bib{ayyer2007quenched}{article}{
      author={Ayyer, Arvind},
      author={Liverani, Carlangelo},
      author={Stenlund, Mikko},
       title={Quenched {CLT} for random toral automorphism},
        date={2009},
        ISSN={1078-0947},
     journal={Discrete \& Continuous Dynamical Systems - A},
      volume={24},
       pages={331},
  url={http://aimsciences.org//article/id/9c1b097a-0622-492f-a473-02d4ba7af386},
}

\bib{BBPV94}{article}{
      author={Babiano, A},
      author={Boffetta, Guido},
      author={Provenzale, A},
      author={Vulpiani, A},
       title={Chaotic advection in point vortex models and two-dimensional
  turbulence},
        date={1994},
     journal={Physics of Fluids},
      volume={6},
      number={7},
       pages={2465\ndash 2474},
}

\bib{BabianoProvenzale07}{article}{
      author={Babiano, Armando},
      author={Provenzale, Antonello},
       title={Coherent vortices and tracer cascades in two-dimensional
  turbulence},
        date={2007},
     journal={Journal of Fluid Mechanics},
      volume={574},
       pages={429\ndash 448},
}

\bib{bahsoun2017quenched}{article}{
      author={Bahsoun, Wael},
      author={Bose, Christopher},
      author={Ruziboev, Marks},
       title={Quenched decay of correlations for slowly mixing systems},
        date={2017},
     journal={arXiv preprint arXiv:1706.04158},
}

\bib{baladi2002almost}{inproceedings}{
      author={Baladi, Viviane},
      author={Benedicks, Michael},
      author={Maume-Deschamps, V{\'e}ronique},
       title={Almost sure rates of mixing for {IID} unimodal maps},
organization={Elsevier},
        date={2002},
   booktitle={Annales scientifiques de l'{\'e}cole normale sup{\'e}rieure},
      volume={35},
       pages={77\ndash 126},
}

\bib{baxendale1993kinematic}{article}{
      author={Baxendale, PH},
      author={Rozovsky, BL},
       title={Kinematic dynamo and intermittence in a turbulent flow},
        date={1993},
     journal={Geophysical \& Astrophysical Fluid Dynamics},
      volume={73},
      number={1-4},
       pages={33\ndash 60},
}

\bib{baxendale1988large}{article}{
      author={Baxendale, PH},
      author={Stroock, DW},
       title={Large deviations and stochastic flows of diffeomorphisms},
        date={1988},
     journal={Probability Theory and Related Fields},
      volume={80},
      number={2},
       pages={169\ndash 215},
}

\bib{BBPS18}{article}{
      author={Bedrossian, Jacob},
      author={Blumenthal, Alex},
      author={Punshon-Smith, Samuel},
       title={Lagrangian chaos and scalar advection in stochastic fluid
  mechanics},
        date={2018},
     journal={arXiv preprint arXiv:1809.06484},
}

\bib{Bressan03}{article}{
      author={Bressan, Alberto},
       title={A lemma and a conjecture on the cost of rearrangements},
        date={2003},
     journal={Rend. Sem. Mat. Univ. Padova},
      volume={110},
       pages={97\ndash 102},
}

\bib{bricmont2002exponential}{article}{
      author={Bricmont, Jean},
      author={Kupiainen, Antti},
      author={Lefevere, R},
       title={Exponential mixing of the 2d stochastic {N}avier-{S}tokes
  dynamics},
        date={2002},
     journal={Communications in mathematical physics},
      volume={230},
      number={1},
       pages={87\ndash 132},
}

\bib{CZDT18}{article}{
      author={Coti~Zelati, Michele},
      author={Delgadino, Matias~G},
      author={Elgindi, Tarek~M},
       title={On the relation between enhanced dissipation time-scales and
  mixing rates},
        date={2018},
     journal={arXiv preprint arXiv:1806.03258},
}

\bib{CRS19}{article}{
      author={Crippa, Gianluca},
      author={Luc{\`a}, Renato},
      author={Schulze, Christian},
       title={Polynomial mixing under a certain stationary {E}uler flow},
        date={2019},
     journal={Physica D: Nonlinear Phenomena},
}

\bib{DPZ96}{book}{
      author={Da~Prato, G.},
      author={Zabczyk, J.},
       title={Ergodicity for infinite-dimensional systems},
      series={London Mathematical Society Lecture Note Series},
   publisher={Cambridge University Press, Cambridge},
        date={1996},
      volume={229},
}

\bib{DaPrato14}{book}{
      author={Da~Prato, Giuseppe},
       title={Introduction to stochastic analysis and {M}alliavin calculus},
   publisher={Springer},
        date={2014},
      volume={13},
}

\bib{dolgopyat2004sample}{article}{
      author={Dolgopyat, Dmitry},
      author={Kaloshin, Vadim},
      author={Koralov, Leonid},
      author={others},
       title={Sample path properties of the stochastic flows},
        date={2004},
     journal={The Annals of Probability},
      volume={32},
      number={1A},
       pages={1\ndash 27},
}

\bib{dragivcevic2018spectral}{article}{
      author={Dragi{\v{c}}evi{\'c}, Davor},
      author={Froyland, Gary},
      author={Gonzalez-Tokman, Cecilia},
      author={Vaienti, Sandro},
       title={A spectral approach for quenched limit theorems for random
  expanding dynamical systems},
        date={2018},
     journal={Communications in Mathematical Physics},
       pages={1\ndash 67},
}

\bib{dunford1957linear}{book}{
      author={Dunford, Nelson},
      author={Schwartz, Jacob~T},
       title={Linear operators. part 1: General theory},
   publisher={New York Interscience},
        date={1957},
}

\bib{E2001-lg}{article}{
      author={E, Weinan},
      author={Mattingly, Jonathan~C},
       title={Ergodicity for the {Navier-Stokes} equation with degenerate
  random forcing: Finite-dimensional approximation},
        date={2001-11},
     journal={Commun. Pure Appl. Math.},
      volume={54},
      number={11},
       pages={1386\ndash 1402},
}

\bib{EH01}{article}{
      author={Eckmann, J-P},
      author={Hairer, Martin},
       title={Uniqueness of the invariant measure for a stochastic {PDE} driven
  by degenerate noise},
        date={2001},
     journal={Comm. Math. Phys.},
      volume={219},
      number={3},
       pages={523\ndash 565},
}

\bib{TZ18}{article}{
      author={Elgindi, Tarek~M},
      author={Zlato{\v{s}}, Andrej},
       title={Universal mixers in all dimensions},
        date={2018},
     journal={arXiv preprint arXiv:1809.09614},
}

\bib{IF19}{article}{
      author={Feng, Yuanyuan},
      author={Iyer, Gautam},
       title={Dissipation enhancement by mixing},
        date={2019},
     journal={Nonlinearity},
      volume={32},
      number={5},
       pages={1810},
}

\bib{FM95}{article}{
      author={Flandoli, Franco},
      author={Maslowski, Bohdan},
       title={Ergodicity of the {2-D Navier-Stokes} equation under random
  perturbations},
        date={1995},
     journal={Comm. in Math. Phys.},
      volume={172},
      number={1},
       pages={119\ndash 141},
}

\bib{GHHM18}{article}{
      author={Glatt-Holtz, Nathan~E},
      author={Herzog, David~P},
      author={Mattingly, Jonathan~C},
       title={Scaling and saturation in infinite-dimensional control problems
  with applications to stochastic partial differential equations},
        date={2018},
     journal={Annals of PDE},
      volume={4},
      number={2},
       pages={16},
}

\bib{goldys2005exponential}{article}{
      author={Goldys, B},
      author={Maslowski, B},
       title={Exponential ergodicity for stochastic {B}urgers and 2{D}
  {N}avier--{S}tokes equations},
        date={2005},
     journal={Journal of Functional Analysis},
      volume={226},
      number={1},
       pages={230\ndash 255},
}

\bib{Hairer11notes}{article}{
      author={Hairer, Martin},
       title={On {Malliavin's} proof of {H{\"o}rmander's} theorem},
        date={2011},
     journal={Bulletin des sciences mathematiques},
      volume={135},
      number={6-7},
       pages={650\ndash 666},
}

\bib{HM06}{article}{
      author={Hairer, Martin},
      author={Mattingly, Jonathan~C.},
       title={Ergodicity of the 2{D} {N}avier-{S}tokes equations with
  degenerate stochastic forcing},
        date={2006},
     journal={Ann. of Math.},
      volume={164},
      number={3},
       pages={993\ndash 1032},
}

\bib{HM08}{article}{
      author={Hairer, Martin},
      author={Mattingly, Jonathan~C.},
       title={Spectral gaps in {W}asserstein distances and the 2{D} stochastic
  {N}avier-{S}tokes equations},
        date={2008},
     journal={Ann. Probab.},
      volume={36},
      number={6},
       pages={2050\ndash 2091},
}

\bib{HM11}{article}{
      author={Hairer, Martin},
      author={Mattingly, Jonathan~C.},
       title={A theory of hypoellipticity and unique ergodicity for semilinear
  stochastic {PDE}s},
        date={2011},
     journal={Electron. J. Probab.},
      volume={16},
       pages={no. 23, 658\ndash 738},
}

\bib{hairer2011yet}{inproceedings}{
      author={Hairer, Martin},
      author={Mattingly, Jonathan~C},
       title={Yet another look at {H}arris's ergodic theorem for {M}arkov
  chains},
organization={Springer},
        date={2011},
   booktitle={Seminar on stochastic analysis, random fields and applications
  vi},
       pages={109\ndash 117},
}

\bib{IyerXu14}{article}{
      author={Iyer, Gautam},
      author={Kiselev, Alexander},
      author={Xu, Xiaoqian},
       title={Lower bounds on the mix norm of passive scalars advected by
  incompressible enstrophy-constrained flows},
        date={2014},
     journal={Nonlinearity},
      volume={27},
      number={5},
       pages={973},
}

\bib{KNS18}{article}{
      author={Kuksin, Sergei},
      author={Nersesyan, Vahagn},
      author={Shirikyan, Armen},
       title={Exponential mixing for a class of dissipative pdes with bounded
  degenerate noise},
        date={2018},
     journal={arXiv preprint arXiv:1802.03250},
}

\bib{kuksin2002coupling}{article}{
      author={Kuksin, Sergei},
      author={Piatniski, Andrey},
      author={Shirikyan, Armen},
       title={A coupling approach to randomly forced nonlinear {PDE}'s, {II}},
        date={2002},
     journal={Communications in mathematical physics},
      volume={230},
      number={1},
       pages={81\ndash 85},
}

\bib{kuksin2002coupling2}{article}{
      author={Kuksin, Sergei},
      author={Shirikyan, Armen},
       title={Coupling approach to white-forced nonlinear {PDE}s},
        date={2002},
     journal={Journal de math{\'e}matiques pures et appliqu{\'e}es},
      volume={81},
      number={6},
       pages={567\ndash 602},
}

\bib{KS}{book}{
      author={Kuksin, Sergei},
      author={Shirikyan, Armen},
       title={Mathematics of two-dimensional turbulence},
   publisher={Cambridge University Press},
        date={2012},
      volume={194},
}

\bib{LDT11}{article}{
      author={Lin, Zhi},
      author={Thiffeault, Jean-Luc},
      author={Doering, Charles~R},
       title={Optimal stirring strategies for passive scalar mixing},
        date={2011},
     journal={Journal of Fluid Mechanics},
      volume={675},
       pages={465\ndash 476},
}

\bib{liverani2004contact}{article}{
      author={Liverani, Carlangelo},
       title={On contact {A}nosov flows},
        date={2004},
     journal={Annals of mathematics},
       pages={1275\ndash 1312},
}

\bib{LLNMD12}{article}{
      author={Lunasin, Evelyn},
      author={Lin, Zhi},
      author={Novikov, Alexei},
      author={Mazzucato, Anna},
      author={Doering, Charles~R},
       title={Optimal mixing and optimal stirring for fixed energy, fixed
  power, or fixed palenstrophy flows},
        date={2012},
     journal={Journal of Mathematical Physics},
      volume={53},
      number={11},
       pages={115611},
}

\bib{MajdaBertozzi}{book}{
      author={Majda, Andrew~J},
      author={Bertozzi, Andrea~L},
       title={Vorticity and incompressible flow},
   publisher={Cambridge University Press},
        date={2002},
      volume={27},
}

\bib{mattingly2002exponential}{article}{
      author={Mattingly, Jonathan~C},
       title={Exponential convergence for the stochastically forced
  navier-stokes equations and other partially dissipative dynamics},
        date={2002},
     journal={Communications in mathematical physics},
      volume={230},
      number={3},
       pages={421\ndash 462},
}

\bib{meyn2012markov}{book}{
      author={Meyn, Sean~P},
      author={Tweedie, Richard~L},
       title={Markov chains and stochastic stability},
   publisher={Springer Science \& Business Media},
        date={2012},
}

\bib{MC18}{article}{
      author={Miles, Christopher~J},
      author={Doering, Charles~R},
       title={Diffusion-limited mixing by incompressible flows},
        date={2018},
     journal={Nonlinearity},
      volume={31},
      number={5},
       pages={2346},
}

\bib{Nualart06}{book}{
      author={Nualart, David},
       title={The {M}alliavin calculus and related topics},
   publisher={Springer},
        date={2006},
      volume={1995},
}

\bib{PBBPW08}{incollection}{
      author={Provenzale, A},
      author={Babiano, A},
      author={Bracco, A},
      author={Pasquero, C},
      author={Weiss, JB},
       title={Coherent vortices and tracer transport},
        date={2008},
   booktitle={Transport and mixing in geophysical flows},
   publisher={Springer},
       pages={101\ndash 118},
}

\bib{Provenzale1999}{article}{
      author={Provenzale, Antonello},
       title={Transport by coherent barotropic vortices},
        date={1999},
     journal={Annual review of fluid mechanics},
      volume={31},
      number={1},
       pages={55\ndash 93},
}

\bib{Romito2004-rc}{article}{
      author={Romito, Marco},
       title={Ergodicity of the finite dimensional approximation of the {3D}
  {Navier--Stokes} equations forced by a degenerate noise},
        date={2004-01},
     journal={J. Stat. Phys.},
      volume={114},
      number={1},
       pages={155\ndash 177},
}

\bib{ruelle1983flows}{article}{
      author={Ruelle, D},
       title={Flows which do not exponentially mix},
        date={1983},
     journal={Comptes rendus de l'academie des sciences, Serie I-Mathematique},
      volume={296},
      number={4},
       pages={191\ndash 193},
}

\bib{sarig2002subexponential}{article}{
      author={Sarig, Omri},
       title={Subexponential decay of correlations},
        date={2002},
     journal={Inventiones mathematicae},
      volume={150},
      number={3},
       pages={629\ndash 653},
}

\bib{S13}{article}{
      author={Seis, Christian},
       title={Maximal mixing by incompressible fluid flows},
        date={2013},
     journal={Nonlinearity},
      volume={26},
      number={12},
       pages={3279},
}

\bib{ShraimanSiggia00}{article}{
      author={Shraiman, Boris~I},
      author={Siggia, Eric~D},
       title={Scalar turbulence},
        date={2000},
     journal={Nature},
      volume={405},
      number={6787},
       pages={639},
}

\bib{varadhan1984large}{book}{
      author={Varadhan, SR~Srinivasa},
       title={Large deviations and applications},
   publisher={Siam},
        date={1984},
      volume={46},
}

\bib{W00}{article}{
      author={Warhaft, Zellman},
       title={Passive scalars in turbulent flows},
        date={2000},
     journal={Annual Review of Fluid Mechanics},
      volume={32},
      number={1},
       pages={203\ndash 240},
}

\bib{WF04}{article}{
      author={Wunsch, Carl},
      author={Ferrari, Raffaele},
       title={Vertical mixing, energy, and the general circulation of the
  oceans},
        date={2004},
     journal={Annu. Rev. Fluid Mech.},
      volume={36},
       pages={281\ndash 314},
}

\bib{YZ14}{article}{
      author={Yao, Yao},
      author={Zlato{\v{s}}, Andrej},
       title={Mixing and un-mixing by incompressible flows},
        date={2017},
     journal={Journal of the European Mathematical Society},
      volume={19},
      number={7},
       pages={1911\ndash 1948},
}

\end{biblist}
\end{bibdiv}

\end{document}